\documentclass[preprint]{elsarticle}
\usepackage{amsmath, amssymb, amsthm, dsfont, bm, mathtools, cancel}

\usepackage{algorithm}
\usepackage{algpseudocode}

\usepackage{graphicx, epsfig, epstopdf}
\usepackage{caption}
\usepackage{subcaption}
\usepackage{lscape}
\usepackage{placeins}

\usepackage{booktabs}
\usepackage{multirow}
\usepackage[table]{xcolor}

\usepackage{endnotes}
\usepackage{setspace}
\usepackage{enumerate}
\usepackage[shortlabels]{enumitem}
\usepackage{ulem}
\usepackage{color}
\usepackage{dirtytalk}
\usepackage{xspace}
\usepackage{xcolor}

\biboptions{sort&compress} 

\usepackage[T1]{fontenc}

\usepackage[colorlinks]{hyperref}
\hypersetup{colorlinks=true,
  linkcolor=red,
  citecolor=blue,
  filecolor=blue,
  urlcolor=blue}
\usepackage{bookmark}

\usepackage[utf8]{inputenc}
\usepackage[toc,page,title,titletoc]{appendix}

\allowdisplaybreaks

\theoremstyle{definition}
\newtheorem{defin}{Definition}
\newtheorem{theorem}{Theorem}
\newtheorem{proposition}{Proposition}
\newtheorem{lemma}[proposition]{Lemma}
\newtheorem{corollary}[proposition]{Corollary}

\newtheorem{condition}{Condition}
\newtheorem{remark}{Remark}

\begin{document}
\begin{frontmatter}
\title{Wasserstein convergence rates for empirical measures of point processes}

\author[colorado]{Dongzhou~Huang}
\ead{dongzhou.huang@colostate.edu}

\author[colorado]{Tianyi~Jiang}
\ead{tianyi.jiang@colostate.edu}

\author[colorado]{Haonan~Wang}
\ead{haonan.wang@colostate.edu}

\affiliation[colorado]{organization={Department of Statistics, Colorado State University},
            addressline={851 Oval Dr.}, 
            city={Fort Collins},
            state={CO},
            postcode={80523}, 
            country={USA}}

\journal{Stochastic Processes and their Applications}

\begin{abstract}
In this paper, we establish upper and lower bounds on the convergence rate of empirical measures of point processes under the Wasserstein distance. To this end, we first introduce a new metric on the space of finite counting measures. Unlike many existing metrics in the literature, our proposed metric incorporates the spatial locations of unmatched points. Drawing on both this new metric and existing metrics from the literature, we identify a set of common structural properties and develop a general framework for Wasserstein distances between point-process distributions based on metrics satisfying these properties. We then use this framework to investigate the convergence rates of empirical measures, providing quantitative guarantees for approximating the distribution of the underlying point process, and further establish corresponding concentration inequalities. These theoretical results may provide ingredients for constructing oracle goodness-of-fit tests for a specified point-process distribution. We also apply our theory to the approximation of Campbell measures and discuss how the resulting bounds can help assess Wasserstein-based learning of generative point-process models.
\end{abstract}

\end{frontmatter}

\section{Introduction}

Point processes constitute a fundamental component of modern probabilistic theory and have been widely applied across diverse scientific fields. However, certain fundamental questions remain inadequately resolved. A central example is the problem of determining the distribution of a point process from observed realizations. A substantial body of literature addresses related inference problems through the estimation of intensity functions, employing both parametric and nonparametric approaches; see, for example, \cite{CRONIE2018,Diggle1985,guan2015,guan2010,Heikkinen1998,vanLieshout2024,vanLieshout2012}. However, methods based solely on the estimation of deterministic intensity functions may not directly accommodate point processes with stochastic intensity structures, such as Cox and Hawkes processes.

In this paper, we propose an alternative approach to the problem of determining the distributions of point processes. Empirical measures have proven to be powerful tools for distributional identification, as exemplified by their use in the Kolmogorov–Smirnov and Cram\'{e}r–von Mises tests. Motivated by these successes, we investigate the empirical measures of point processes. The primary aim of this work is to establish analytical results concerning these empirical measures, thereby providing theoretical ingredients that may help characterize the distributions of point processes.

Before proceeding to empirical measures, we first establish some preliminary concepts. In the study of stochastic processes, temporal point processes on $[0,1]$ are typically treated in two ways. The first and more common approach regards a point process as a collection of time-dependent random variables. The second approach views it as a random element taking values in the space of counting measures on $[0,1]$. Although the latter framework involves a more complex value space, it offers substantial analytical power and admits generalization beyond temporal point processes to abstract point processes. In light of these advantages, we adopt the second perspective.

When working with the value space, i.e., the space of counting measures, a challenging task is to define a rigorous and meaningful metric on this space. Several approaches have been proposed in the literature. For counting measures with the same number of points, often referred to as the cardinality, there is little disagreement: the standard practice is to assign equal weights to all points and to employ the Wasserstein distance for discrete measures. The main difficulty arises when comparing counting measures with different numbers of points, in which case the definition of a suitable distance is far from straightforward and is related to the broader framework of unbalanced optimal transport; see, for example, \cite{chizat2018scaling,CHIZAT2018,Savar2024}. Early works (e.g., \cite{Chen2004}, \cite{schuhmacher2005}) used a bounded metric that takes values in $[0,1]$, assigning the maximum distance $1$ whenever the cardinalities differ. This construction has one main drawback: it disregards discrepancies between measures beyond their cardinalities. To address this, Schuhmacher and Xia  \cite{Schuhmacher2007} proposed a refinement for the case of counting measures with $m$ and $n$ points ($m<n$). 
Specifically, $m$ points are selected from the $n$-point measure and matched with all $m$ points in the $m$-point measure for a Wasserstein-type comparison, while the remaining $n-m$ points incur an additional penalty proportional to $n-m$. The resulting distance is obtained by summing the matching cost and the penalty for unmatched points, followed by an appropriate normalization. A more flexible construction is the transport--transform (TT) metric introduced in \cite{muller2020}, which allows an equal number of points to be selected from each counting measure for matching, without requiring all points in the $m$-point measure to be matched. Each remaining unmatched point incurs the same fixed penalty, with the optimal matching and the corresponding unmatched points determined through minimization over partial matchings.

In addition to the existing metrics discussed above, we consider a new metric for finite counting measures that accounts for the spatial locations of unmatched points.
Our construction is motivated by \cite{xiao2017wasserstein}, which was shown to be useful in the context of generative point process models. More specifically, we augment the underlying metric space with an auxiliary point $s_{\alpha}$ that is separated from $S$ by a strictly positive distance. For two counting measures with $m$ and $n$ points (with $m<n$), we extend the $m$-point measure by adding $n-m$ copies of $s_{\alpha}$, thereby forming a new $n$-point measure. The distance between the original two measures is then defined through the Wasserstein distance between these two $n$-point measures. This construction imposes a positive baseline penalty on unmatched points while allowing their locations to contribute to the distance. In Section~\ref{sec: ANewMetric}, we discuss in detail the choice of the auxiliary point $s_{\alpha}$ and its influence on the resulting distance, as well as the relationship between our proposed metric and the two metrics discussed above.  From a broader perspective, the idea of augmenting the underlying space with an auxiliary point is also related to constructions considered in the literature on unbalanced optimal transport; see, for example, \cite{Savar2024}.


With the metrics introduced above on the space of finite counting measures, point processes can be viewed as random elements taking values in this space. This perspective allows us to define Wasserstein distances between point-process distributions using the standard Wasserstein construction; see Section~\ref{sec:WDforPointProcess} for a detailed definition. These distances provide a natural framework for quantifying discrepancies between empirical measures of point processes and their underlying population distributions.

By comparing our proposed metric on the space of finite counting measures with existing metrics in the literature, we identify several common structural properties, which are summarized in Condition~\ref{assume:metric}. These properties allow us to establish quantitative convergence results for the Wasserstein distance between the empirical measure of a point process and its underlying distribution, regardless of the particular counting-measure metric used, provided that it satisfies Condition~\ref{assume:metric}. We therefore develop a general framework for Wasserstein distances between point-process distributions based on metrics satisfying these properties. Furthermore, these properties may provide useful insights into the construction of counting-measure metrics suitable for such quantitative analysis.

Before proceeding further, it is useful to distinguish the two levels at which Wasserstein-type distances arise in our framework. At the first level, the metrics considered in Section~\ref{sec: ANewMetric} compare individual counting measures and, when the two measures have the same cardinality, reduce to the standard Wasserstein matching distance between their point configurations. At the second level, in Section~\ref{sec:WDforPointProcess}, we view point processes as random elements taking values in the metric space of counting measures and define a Wasserstein distance between their probability distributions, thereby inducing a distance between point processes themselves. This latter distance forms the basis of our subsequent analysis of point processes.


We now turn to the study of empirical measures of point processes. Let $\eta$ denote a point process with independent realizations $\eta_1,\ldots,\eta_n$. The associated empirical measure is defined as the random probability measure that assigns mass $1/n$ to each realization. We consider the Wasserstein distance between this empirical measure and the underlying distribution of the point process $\eta$ and establish upper and lower bounds on its expectation. These bounds may provide useful insights into statistical inference problems involving point processes by quantifying the accuracy of empirical approximations to the underlying point-process distribution. Establishing these quantitative bounds constitutes the principal contribution of the present work.

A substantial body of literature has investigated the expected Wasserstein distance, with relevant work including, but not limited to, \cite{fournier2015rate,dereich2013constructive,dudley1969speed,boissard2014mean,weed2019sharp,singh2018minimax,LeiJing2020Caco}. Focusing on measures over Euclidean spaces, \cite{fournier2015rate,dereich2013constructive} established upper bounds on the convergence rate of empirical measures to the underlying population measure under the Wasserstein distance of order $p$. 
For more general settings, Dudley \cite{dudley1969speed} was the first to derive an upper bound for the Wasserstein distance of order $1$ between probability measures defined on bounded Polish metric spaces. This analysis was later extended by Boissard and Le Gouic \cite{boissard2014mean} to arbitrary $p \geq 1$ on Polish metric spaces with finite Minkowski dimension, under additional moment assumptions. Weed and Bach \cite{weed2019sharp} further refined these results for compact metric spaces. However, these findings are not directly applicable to our setting, since the space of counting measures that we consider is both unbounded and infinite-dimensional.
The works most closely related to our study are \cite{LeiJing2020Caco,singh2018minimax}, which addressed the expected Wasserstein distance for measures over unbounded spaces. Lei \cite{LeiJing2020Caco} investigated an unbounded, infinite-dimensional Banach space and established sharp upper bounds under suitable moment conditions. Nevertheless, this result does not transfer directly to our setting, as the space of counting measures is not a Banach space. Singh and P\'{o}czos \cite{singh2018minimax} established minimax rates of empirical measures under the Wasserstein distance by partitioning potentially unbounded spaces into sequences of ``thick spherical shells.'' In our case, however, such shells lead to a complicated subspace, making the analysis intractable; thus, their results cannot be directly applied here.

Another line of research relevant to our work concerns Poisson process approximation and its subsequent extensions. This literature investigates the discrepancy between a given point process and a target Poisson process via Stein’s method; in the extended setting, the focus is on discrepancies between two general point processes. Relevant references include, but are not limited to, \cite{decreusefond2016functional,Barbour1988Stein,barbour1992Poisson,Brown1995Metrics,Chen2004,Schuhmacher2005Upper,Schuhmacher2009Stein}.  For example, \cite{Brown1995Metrics} study Wasserstein-type metrics for point-process approximation and derive Poisson process approximation bounds under these metrics, while \cite{decreusefond2016functional} construct point processes from tuples of points and quantify how close they are to a Poisson process using the Kantorovich–Rubinstein distance. This body of research differs from ours not only in its primary objects of study---since we investigate the empirical measure of independent realizations of a point process rather than the behavior of a single given point process---but also in the analytical tools employed.

Separately, a growing body of literature has investigated distributional limits and statistical inference for empirical Wasserstein distances and, more broadly, empirical optimal transport; see, e.g., \cite{klatt2020,sommerfeld2018inference,Tameling2019Empirical,goldfeld2024limit,del2024central}. Sommerfeld and Munk \cite{sommerfeld2018inference} established generally non-Gaussian limiting distributions for empirical Wasserstein distances on finite metric spaces, and Tameling et al. \cite{Tameling2019Empirical} subsequently extended these results to countable metric spaces. Klatt et al. \cite{klatt2020} studied regularized optimal transport on finite metric spaces. More recently, Goldfeld et al. \cite{goldfeld2024limit} developed limit distribution theory for Gaussian-smoothed Wasserstein distances, while del Barrio et al. \cite{del2024central} established central limit theorems for general transportation costs on Euclidean spaces under suitable moment and regularity assumptions. This line of research is promising, although much of the existing literature focuses on finite, countable, or Euclidean spaces. Extending such distributional results to more general metric spaces, including spaces of finite counting measures, would be an interesting direction for future research.


We now turn to the bounds for the aforementioned expected Wasserstein distance.
The derivation of the upper bound relies on two key ingredients: a truncation argument and a telescoping technique. The truncation argument reduces the value space, namely, the space of counting measures, by restricting attention to a subspace in which the number of points does not exceed a fixed constant. This restricted space is compact, which makes our analysis more tractable than in the original unbounded setting. The telescoping technique, inspired by \cite{LeiJing2020Caco}, decomposes complex terms into simpler components, thereby refining the analysis and enabling sharper estimates. By combining these two tools, we derive an upper bound for the expected Wasserstein distance. The full details are provided in Section \ref{sec:upperbound}. Furthermore, we demonstrate the applicability of our results to Poisson point processes and linear Hawkes processes.

It is worth noting that the proof of the lower bound differs in nature from that of the upper bound, and the lower bound itself possesses several advantages. In particular, it holds not only in expectation but also almost surely. Moreover, its analysis avoids the truncation and telescoping arguments required for establishing the upper bound. This simplicity, however, comes at a cost: we require a lower bound on the Minkowski-type dimension of the principal component of the support of the distribution of the underlying point process. To this end, we introduce the Janossy measure of the point process and impose a condition on this measure. Although this condition is somewhat restrictive, to the best of our knowledge, it represents the most suitable assumption currently available. In future work, we aim to relax this condition by developing alternative analytical approaches. We will present the details in Section~\ref{sec:lowerbound}.

Finally, in Section \ref{sec:concentration}, we establish concentration results for the Wasserstein distance between the empirical measure and the distribution of the underlying point process. These results, combined with the convergence-rate bounds established in the preceding sections, provide ingredients for goodness-of-fit assessment for point-process data. In Section~\ref{sec:insights}, we formulate an oracle goodness-of-fit test for a specified point-process distribution and illustrate how the theoretical bounds may enter the construction of a rejection threshold. Since this threshold depends on constants that are generally not available explicitly, the resulting procedure is oracle in nature, and we also discuss simulation-based calibration as a practical alternative. We also discuss how our methods compare with and differ from standard goodness-of-fit assessment approaches, such as the $K$-function \cite{Moller2003}. In addition, Section~\ref{sec:insights}  considers two further implications of our results.
First, we discuss how the Wasserstein convergence bounds may help in the approximation of Campbell measures. Second, in the context of generative point-process models, we discuss how these bounds may help assess the discrepancy between a target point-process distribution and the distribution induced by a trained generator. In these discussions, we point out some limitations of our theoretical results and clarify the settings in which they may be applied.

The rest of the paper is organized as follows. Section~\ref{sec: ANewMetric} reviews several existing metrics in the literature and provides a detailed construction of our proposed metric for finite counting measures. We then extract common structural properties shared by these metrics and show that any metric satisfying these properties induces a metric space with the same topological properties. Section~\ref{sec:WDforPointProcess} formally defines Wasserstein distances between point-process distributions induced by counting-measure metrics satisfying these common structural properties. In Section~\ref{sec:upperbound}, we develop upper bounds for the expected Wasserstein distance under two types of constraints: one concerning the Minkowski dimension of the underlying space, and the other involving a tail estimate for the point process. Section~\ref{sec:lowerbound} presents the corresponding lower bound, including conditions imposed on the Janossy measure. Section~\ref{sec:concentration} establishes concentration results for the Wasserstein distance between the empirical measure and the underlying point-process distribution. Finally,
Section~\ref{sec:insights} discusses several implications of the theoretical results, including an oracle goodness-of-fit test, approximation of Campbell measures, and Wasserstein-based learning of generative point-process models.

\section{Metrics for finite counting measures}
\label{sec: ANewMetric}

Consider a compact metric space $(S,\rho)$. Let $\mathfrak{N}(S)$ (abbreviated as $\mathfrak{N}$ when the context is clear) denote the space of all finite counting measures on $(S, \rho)$. The vague topology $\mathcal{T}$ on $\mathfrak{N}(S)$ is defined as the smallest topology that renders the mappings $\mu \mapsto \int f \, d\mu$ continuous for all continuous, real-valued functions $f$ on $S$ with compact support. Since $S$ is compact, the vague topology coincides with the weak topology on $\mathfrak{N}(S)$, implying that vague convergence and weak convergence are equivalent under this setting. Furthermore, the $\sigma$-algebra $\mathcal{N}_{\mathcal{T}}$ on $\mathfrak{N}(S)$ is defined as the $\sigma$-algebra generated by the vague topology $\mathcal{T}$. Notably, this $\sigma$-algebra coincides with the smallest $\sigma$-algebra that makes the point count mappings $\mu \mapsto \mu(B)$  measurable for all Borel subsets $B \subseteq S$ (see, for example, \cite{kallenberg1986random}) in our setting.


\subsection{Existing metrics for finite counting measures}
In this subsection, we review several metrics for finite counting processes that have been studied in the literature.

Consider two counting measures $\mu_{1} = \sum_{i=1}^{n}\delta_{x_{i}}$ and $\mu_{2} = \sum_{i=1}^{m}\delta_{y_{i}}$ in $\mathfrak{N}(S)$, where $\delta_{x}$ denotes the Dirac measure concentrated at $x$. When $n = m$, a natural choice of distance between $\mu_{1}$ and $\mu_{2}$ is given by their \textit{first-moment Wasserstein distance}: 
$$
D_{w}(\mu_{1},\mu_{2}) = \min_{\pi \in \Pi_{n}} \sum_{i=1}^{n}\rho(x_i,y_{\pi(i)}),
$$ 
where $|\mu_1| = |\mu_2| = n$ is the cardinality, and $\Pi_{n}$ is the collection of all permutations of ${1,\dots,n}$. 
It can be seen that the distance between $\mu_{1}$ and $\mu_{2}$ is determined by the pairwise distances between the two corresponding point sets in $S$, namely $\{x_{1}, \ldots, x_{n}\}$ and $\{y_{1}, \ldots, y_{n}\}$, through the quantities $\rho(x_{i}, y_{\pi(i)})$. 
Related cardinality-matching Wasserstein-type distances on configuration spaces have been studied for finiteness of the induced transport cost and uniqueness of optimal couplings \cite{Decreusefond2008Wasserstein}.

When $n \neq m$, the counting measures $\mu_1$ and $\mu_2$ have different cardinalities, and extending the above matching distance requires additional considerations. More generally, comparing nonnegative measures of unequal total mass falls within the framework of unbalanced optimal transport; see, for example, \cite{chizat2018scaling,CHIZAT2018,Savar2024}. Here, we restrict our attention to finite counting measures.

Without loss of generality, assume that $n\leq m$. Early studies, such as \cite{Chen2004,schuhmacher2005}, addressed this issue by assigning a pre-determined maximal distance, typically $1$, whenever the two finite counting measures have different cardinalities. In a subsequent work \cite{Schuhmacher2007}, the authors considered the distance
$$
D_S(\mu_{1},\mu_{2}) = \frac{1}{m} \left[ \min_{\pi \in \Pi_{m}} \sum_{i=1}^{n}\rho(x_i,y_{\pi(i)}) + (m-n)\right],
$$
where $\rho$ is bounded above by $1$. Under this definition, each unmatched point incurs a unit penalty before normalization, contributing $1/m$ to the resulting distance. This construction extends naturally to any metric space with a bounded metric by replacing the unit penalty with an upper bound on $\rho$. However, all unmatched points incur the same penalty, irrespective of their locations.

A further development along this direction is the transport--transform (TT) metric introduced in \cite{muller2020}. The TT metric retains the matching structure of $D_S$ while allowing greater flexibility in how points are matched. Whereas $D_S$ matches every point of the smaller measure to a point of the larger measure, the TT metric allows points in either measure to remain unmatched, with a penalty assigned to each such point. More specifically, the TT metric in \cite{muller2020} is defined as 
\begin{equation*}
D_{TT} \left( \mu_1, \mu_2 \right) = \left(  \min_{(i_1, \dots, i_l; j_1, \dots, j_l) \in R(n,m)} \left(  (n +m-2l) C^{p} + \sum_{k=1}^{l} \rho( x_{i_k}, y_{j_k} )^p \right) \right)^{1/p},
\end{equation*}
where $C>0$ and $p\in[1,\infty)$ are given parameters, and $R(n,m)$ is the collection of partial matchings and is defined by
\begin{align*}
R(n,m):=  &\big\{ (i_1, \dots, i_l; j_1, \dots, j_l): l \in \{0, 1, \dots, \min\{n,m\}\}, \\
& \qquad \quad i_1, \dots, i_l \in \{1, \dots, n\} \text{ pairwise different}, \  j_1, \dots, j_l \in \{1, \dots, m\} \text{ pairwise different}
\big\}.
\end{align*}
For each partial matching, the cost consists of a penalty of $C^p$ for each unmatched point of either measure and a transportation cost of $\rho(x_{i_k},y_{j_k})^p$ for each matched pair. The metric is obtained by minimizing this total cost over all partial matchings and taking the $p$th root. Thus, compared with $D_S$, the TT metric offers greater flexibility through the penalty parameter $C$ and the optimization over the number of matched points.

The TT metric also admits an alternative representation. Let $S_\ast=S\cup\{s_\ast\}$, where $s_\ast\notin S$ is an auxiliary point, and define a metric $\rho_\ast$ on $S_\ast$ by
\[
\rho_\ast(x,y)
=
\begin{cases}
\rho(x,y)\wedge 2^{1/p}C, & x,y\in S,\\
C, & s_\ast\in\{x,y\},\ x\neq y,\\
0, & x=y=s_\ast.
\end{cases}
\]
For $n\leq m$, set $x_i=s_\ast$ for $i=n+1,\ldots,m$. Then Theorem 2.1 of \cite{muller2020} yields
\[
D_{\mathrm{TT}}(\mu_1,\mu_2)
=
\left(
\min_{\pi\in \Pi_m}
\sum_{i=1}^m
\rho_\ast(x_i,y_{\pi(i)})^p
\right)^{1/p}.
\]
Thus, the TT metric can be computed by augmenting the counting measure with fewer points with the appropriate number of copies of the auxiliary point $s_\ast$, so that the two measures have equal cardinality, and then applying a Wasserstein-type matching distance of order $p$ based on $\rho_\ast$.


A similar idea of adding points to equalize cardinalities was used in \citep{xiao2017wasserstein}, where the authors defined a metric on finite counting measures over $S=[0,T)$ by adding points at the boundary point $T$. For a general state space $S$, this idea may be extended to the following matching cost for $n\leq m$:
\begin{equation}
D_{0}(\mu_{1},\mu_{2})  \coloneqq \min_{\pi \in \Pi_{m}} \left[\sum_{i=1}^{n}\rho(x_i,y_{\pi(i)}) + \sum_{i=n+1}^{m}\rho(s,y_{\pi(i)}) \right].  \label{eqnew:initialD0}
\end{equation}
Here, we assume that $S$ can be embedded into a larger metric space, with the original metric $\rho$ extended to this larger space (still denoted by $\rho$), and that $s$ is a fixed boundary point of $S$ in this larger space.

\subsection{Our proposed distance for finite counting measures}
The three examples presented in the preceding subsection illustrate two common strategies for defining distances between finite counting measures. The first assigns a uniform penalty to each unmatched point. The second equalizes the cardinalities by adding copies of an auxiliary point to the measure with fewer points and then computing a Wasserstein-type matching distance. We adopt the second strategy, specifically the construction illustrated by the third example, in which the penalty for each unmatched point is determined by its distance to a designated point. This construction therefore accounts for the locations of unmatched points rather than assigning a uniform penalty.


Although the distance proposed in \cite{xiao2017wasserstein} for finite counting measures on $S=[0,T)$ is a valid metric, its direct extension to a general space $S$, as given in \eqref{eqnew:initialD0}, may not preserve the metric properties. 
In particular, when $S$ is compact, $D_0$ can fail to be a metric because it does not satisfy the positive definiteness condition. To see this, consider the one-dimensional Euclidean space with $S=[0,T]$ and take the boundary point $s=T$. Then
\[
D_0(\delta_0+\delta_T,\delta_0)=0,
\]
although $\delta_0+\delta_T\neq\delta_0$.

Second, excluding the boundary point from $S$ avoids this failure of positive definiteness, but the resulting boundary behaviour may be undesirable for comparison of point configurations. Suppose that $S=[0,T)$, with $T>1$, and that $\rho$ is the Euclidean distance. Consider the two sequences of counting measures
\[
\left\{\delta_0+\sum_{i=1}^{n}\delta_{T-1/i}\right\}_{n=1}^{\infty},
\qquad
\left\{\delta_0+\sum_{i=1}^{n}\delta_{T-1/(i+n)^2}\right\}_{n=1}^{\infty}.
\]
In both sequences, each configuration contains a point at $0$ together with an increasing number of points accumulating near $T$. However, a straightforward calculation yields
\[
D_0\left(\delta_0+\sum_{i=1}^{n}\delta_{T-1/i},\delta_0\right)
=\sum_{i=1}^{n}\frac{1}{i}\longrightarrow\infty,
\qquad
D_0\left(\delta_0+\sum_{i=1}^{n}\delta_{T-1/(i+n)^2},\delta_0\right)
=\sum_{i=1}^{n}\frac{1}{(i+n)^2}\longrightarrow 0.
\]
Thus, despite their similar qualitative clustering behavior, the two sequences exhibit completely different behavior under $D_0$. Indeed, the contribution of the first cluster diverges to infinity, whereas that of the second converges to zero. This example illustrates that the contribution of unmatched points can depend sensitively on their precise locations relative to the auxiliary point, even when the corresponding point configurations have similar clustering patterns.


Despite this sensitivity, the locations of unmatched points carry useful information that should be retained in the distance. At the same time, the two examples above suggest that this dependence should be quantitatively controlled: configurations with similar clustering behavior should not produce arbitrarily different results solely because the points lie at slightly different distances from the boundary. This motivates our refinement of $D_0$, which preserves the contribution of location while imposing a positive baseline penalty on each unmatched point.

Specifically, we construct a distance between finite counting measures as follows. First, we adjoin an auxiliary point $s_\alpha\notin S$ to $S$ and extend the original metric $\rho$ to $S\cup\{s_\alpha\}$. With a slight abuse of notation, we denote the extended metric by $\rho$ and require that
\begin{equation}
\inf_{s\in S}\rho(s,s_\alpha)=\alpha,  \label{eqnew:requirementsalpha}
\end{equation}
where $\alpha >0$ is a given constant.
The choice of the auxiliary point $s_\alpha$ and the extension of the metric $\rho$ will be discussed in detail later.
Using this extended metric, we define the distance $D_1$ for $n\leq m$ by
\begin{equation}\label{eqn:metricD1}
D_{1}(\mu_{1},\mu_{2})
\coloneqq
\min_{\pi\in\Pi_m}
\left[
\sum_{i=1}^{n}\rho(x_i,y_{\pi(i)})
+
\sum_{i=n+1}^{m}\rho(s_\alpha,y_{\pi(i)})
\right].
\end{equation}
Equivalently, $D_1$ can be viewed as the Wasserstein matching distance obtained by augmenting the smaller configuration with copies of the auxiliary point $s_\alpha$. Specifically, define
\[
\mu_1^*
=
\mu_1+\sum_{i=n+1}^{m}\delta_{s_\alpha}.
\]
Then $\mu_1^*$ and $\mu_2$ have the same cardinality, and
\begin{equation}
\label{eqn:equaltyD1}
D_{1}(\mu_{1},\mu_{2})
=
D_{w}(\mu_1^*,\mu_{2}).
\end{equation}
Since $\mu_1^*\notin\mathfrak{N}(S)$, we extend the definition of $D_w$ to $\mathfrak{N}(S \cup \{s_{\alpha}\} )$.

Our construction combines location-dependent penalties with a tunable positive lower bound. Each unmatched point $s$ contributes a penalty of $\rho(s,s_\alpha)\geq\alpha>0$, allowing its location relative to the auxiliary point to influence the distance while preventing the total contribution of unmatched points from vanishing as they accumulate near the boundary. Moreover, the choice of $\alpha$ and the auxiliary point $s_\alpha$, subject to the requirement in \eqref{eqnew:requirementsalpha}, provides additional flexibility. In particular, these choices may be tailored to incorporate modeling considerations or prior knowledge about the spatial distribution of the point process. We briefly discuss this possibility below when considering the choice of $s_{\alpha}$ and the extension of $\rho$.

We next discuss the choice of the auxiliary point $s_\alpha$ and the corresponding extension of the metric $\rho$. This choice can be straightforward when $S$ is a subset of $\mathbb{R}$, whereas a general space $S$ may require a more careful construction. We present one possible construction, without excluding other alternatives. Fix $\alpha>0$ and define the product space
\[
\widetilde S :=S\times\{0,\alpha\}.
\]
For $(s_1,i_1),(s_2,i_2)\in\widetilde S$, define the following three metrics:
\[
\widetilde\rho_1\bigl((s_1,i_1),(s_2,i_2)\bigr)
=
\rho(s_1,s_2)+|i_1-i_2|,
\]
\[
\widetilde\rho_2\bigl((s_1,i_1),(s_2,i_2)\bigr)=\sqrt{\rho(s_1,s_2)^2+|i_1-i_2|^2},
\]
and
\[
\widetilde\rho_\infty\bigl((s_1,i_1),(s_2,i_2)\bigr)=\max\bigl\{\rho(s_1,s_2),|i_1-i_2|\bigr\}.
\]
These metrics correspond to the $L_1$, $L_2$, and $L_\infty$ norms, respectively. Similarly, one can also define the extended metric for any $L_p$ norm. For each $k\in\{1,2,\infty\}$, the map $s\mapsto(s,0)$ is an isometry from $(S,\rho)$ onto $S\times \{0\}$ equipped with the metric induced by $\widetilde{\rho}_k$. We therefore identify $S$ with $S\times\{0\}$ and regard $\widetilde{\rho}_{k}$ as an extension of $\rho$. It remains to select an auxiliary point in $S\times\{\alpha\}$. To this end, fix $s_*\in S$ and set
\[
s_\alpha=(s_*,\alpha)\in S\times\{\alpha\}.
\]
Then $s_\alpha\notin S\times\{0\}$ and, for $k\in\{1,2,\infty\}$,
\[
\inf_{s\in S}\widetilde\rho_k\bigl((s,0),s_\alpha\bigr)=\alpha.
\]

The above construction yields the required auxiliary point and extended metric, which further define the distance $D_1$. Under this distance, each unmatched point incurs a penalty of at least $\alpha$, while its spatial position relative to $s_*$ may provide an additional contribution. We next discuss how different choices of $\alpha$ and $s_*$ affect the behavior of $D_1$.


When $\alpha$ is small, the choice of $s_{*}$ may have a substantial influence on the resulting metric. In such a situation, $s_{*}$ may be chosen based on modeling considerations or prior knowledge. For example, if the process is expected to be inhomogeneous, with points occurring more frequently in a certain region, one may choose $s_{*}$ to lie in that region. Unmatched points located near this region then make a relatively modest contribution to $D_1$, whereas those located far from the region, where point occurrences are expected to be less frequent, make a substantially larger contribution. In this way, the distance $D_1$ reflects the greater discrepancy between two point measures when unmatched points occur at outlying locations.

By contrast, if $\alpha$ is large, then the choice of $s_{*}$ has less influence on the penalty assigned to unmatched points. For example, suppose that $\mathrm{diam}(S)\leq 1$ and choose $\alpha=1$. If $\widetilde\rho_\infty$ is used, then
\[
\widetilde\rho_\infty\bigl((s,0),s_\alpha\bigr)=1
\]
for every $s \in S$. Hence, each unmatched point incurs the same unit penalty, and the resulting distance recovers the unnormalized version of $D_S$. Furthermore, let $\alpha = C$ and suppose that $\mathrm{diam}(S)\leq  2C$ and $s_{*}$ can be chosen such that $\rho(s,s_{*})\leq C$ for all $s\in S$. Then $ \widetilde\rho_\infty $ coincides with the metric $\rho_{\ast}$ appearing in the alternative expression of TT metric when $p =1$. Consequently, in this setting, our distance $D_1$ coincides with $D_{\mathrm{TT}}$.

Our goal in this paper is restricted to finite counting measures, and we therefore consider the proposed distance specifically within this setting. If one considers more general nonnegative measures, our construction is conceptually related to the framework of unbalanced optimal transport in \cite{Savar2024}. In that work, unbalanced transport costs between measures of possibly different total masses are obtained through a general relaxation procedure starting from costs between weighted Dirac masses. From this broader perspective, the location-dependent penalty assigned to unmatched points in our construction may be viewed as analogous to a location-dependent cost for variation of mass. We leave a precise characterization of the relationship between metrics such as $D_0$ and $D_1$ and the framework of \cite{Savar2024} for future work.





We close this subsection with a theorem establishing that the proposed distance $D_1$ is a valid metric and providing corresponding upper and lower bounds. The proof of the theorem is deferred to Appendix \ref{appendix1}.

\begin{theorem}
\label{thm:d1metric} 
$D_1$ is a proper metric and satisfies
\[ 
D_{1} ( \mu_1, \mu_2) \leq \left( |\mu_1| \vee |\mu_2|  \right) \left( \mathrm{diam}(S) + \alpha \right), \]
 and
\[D_1(\mu_1,\mu_2)\geq \bigl||\mu_1|-|\mu_2|\bigr|\alpha,\]
where $|\mu_1|$ and $|\mu_2|$ denote the cardinalities of $\mu_1$ and $\mu_2$, respectively.
\end{theorem}

\subsection{An example of the proposed distance}

In this subsection, we provide an illustrative example of the proposed distance $D_1$. Suppose that $S$ is an interval on the real line; for concreteness, let $S=[0,T]$, and let $\rho$ be the Euclidean distance. In this setting, the expression for $D_1$ can be simplified. We may choose the auxiliary point as $s_{\alpha}=T+\alpha$. Consider two counting measures $\mu_{1} = \sum_{i=1}^{n}\delta_{x_{i}}$ and $\mu_{2} = \sum_{i=1}^{m}\delta_{y_{i}}$ with $n \leq m$ and $m \geq 1$. Without loss of generality, we assume that $x_1 \le x_2 \le \dots \le x_n$ and $y_1 \le y_2 \le \dots \le y_m$. 


For any counting measure $\mu$ with $|\mu|=m$, define $F_{\mu}(t) = \mu([0,t])/m$ for $t \in \mathbb{R}^{+}$. For $0 < u \leq 1$, let $F^{-1}_{\mu}(u) = \inf\{t \in \mathbb{R}^{+} : F_{\mu}(t) \ge u\}$ denote its inverse function. Note that $F_{\mu}$ and $F^{-1}_{\mu}$ coincide with the cumulative distribution function and the corresponding quantile function of the normalized measure $\mu/m$, respectively; hence, we refer to them by these names. It follows from the definitions that $F^{-1}_{\mu_1^{*}}(i/m) = x_i$ and $F^{-1}_{\mu_2}(i/m) = y_i$ for each $i = 1, \dots, m$. Under this setup, the distance $D_{1}(\mu_{1},\mu_{2})$ can be expressed as
\[
D_{1}(\mu_{1},\mu_{2}) = \min_{\pi \in \Pi_{m}}\sum_{i=1}^{m}\rho \big(F_{\mu^{*}_1}^{-1}(i/m), F_{\mu_2}^{-1}(\pi(i)/m)\big).
\]
Since $\rho$ is the Euclidean distance, the expression above can be further simplified
\[
D_{1}(\mu_{1},\mu_{2}) = \sum_{i=1}^{m}\big|x_i - y_i\big| = \sum_{i=1}^{m}\big|F_{\mu^{*}_1}^{-1}(i/m) - F_{\mu_2}^{-1}(i/m)\big| = m\int_{0}^{1}\!\big|F_{\mu^{*}_1}^{-1}(u) - F_{\mu_2}^{-1}(u)\big|\,du,
\]
where the last equality follows from a straightforward calculation since $F^{-1}_{\mu}$ is a step function. This integral form resembles the standard quantile-function expression of the Wasserstein distance (see, for example, \citep{Panaretos2019}). Further, by the well-known correspondence between cumulative distribution and quantile functions on $\mathbb{R}$, we have
\[
D_{1}(\mu_{1},\mu_{2}) = m\int_{-\infty}^{\infty}\!\big|F_{\mu^{*}_1}(x) - F_{\mu_2}(x)\big|\,dx = m\int_{0}^{\infty}\!\big|F_{\mu^{*}_1}(x) - F_{\mu_2}(x)\big|\,dx.
\]
Recall that $\mu_1^* = \mu_1+\sum_{i=n+1}^{m}\delta_{s_\alpha}$.

For illustration purpose, we consider two counting measures $\mu_1 = \sum_{i=1}^{4}\delta_{x_i}$ and $\mu_2 = \sum_{i=1}^{5}\delta_{y_i}$ where $\{x_1, \ldots, x_4\}$ and $\{y_1, \ldots, y_5\}$ lie within the interval $[0, T]$ and $\rho$ is the Euclidean distance. The precise locations of the points are $\{x_1, \ldots, x_4\} = \{0.5, 1.0, 3.8, 7.2\}$, $\{y_1, \ldots, y_5\} = \{0.7, 2.0, 2.0, 5.0, 6.0\}$, and $T = 8.0$. Let $\alpha = 1$ and $s_{\alpha} = T + \alpha = 9.0$. Following the earlier construction, we have the augmented measure $\mu_1^{*} = \sum_{i=1}^{5}\delta_{x_i}$ where $x_5 = s_{\alpha}$. By definition, we have $D_1(\mu_1,\mu_2) = 8.2$. 

The left panel of Figure~\ref{fig:metrc_illustration} illustrates the functions $F_{\mu^{*}_1}$ and $F_{\mu_2}$, and the right panel illustrates the inverse functions $F_{\mu^{*}_1}^{-1}$ and $F_{\mu_2}^{-1}$. The shaded area in the left panel of Figure~\ref{fig:metrc_illustration} represents $\int_{0}^{\infty}\!\big|F_{\mu^{*}_1}(x) - F_{\mu_2}(x)\big|\,dx$, while the shaded area in the right panel represents $\int_{0}^{1}\!\big|F_{\mu^{*}_1}^{-1}(u) - F_{\mu_2}^{-1}(u)\big|\,du$. The areas of both regions are equal to~1.64 as expected, since the shaded area multiplied by $m=5$ equals $D_1(\mu_1,\mu_2)$ as demonstrated.
\begin{figure}[!ht]
    \centering
    \includegraphics[width=0.7\linewidth]{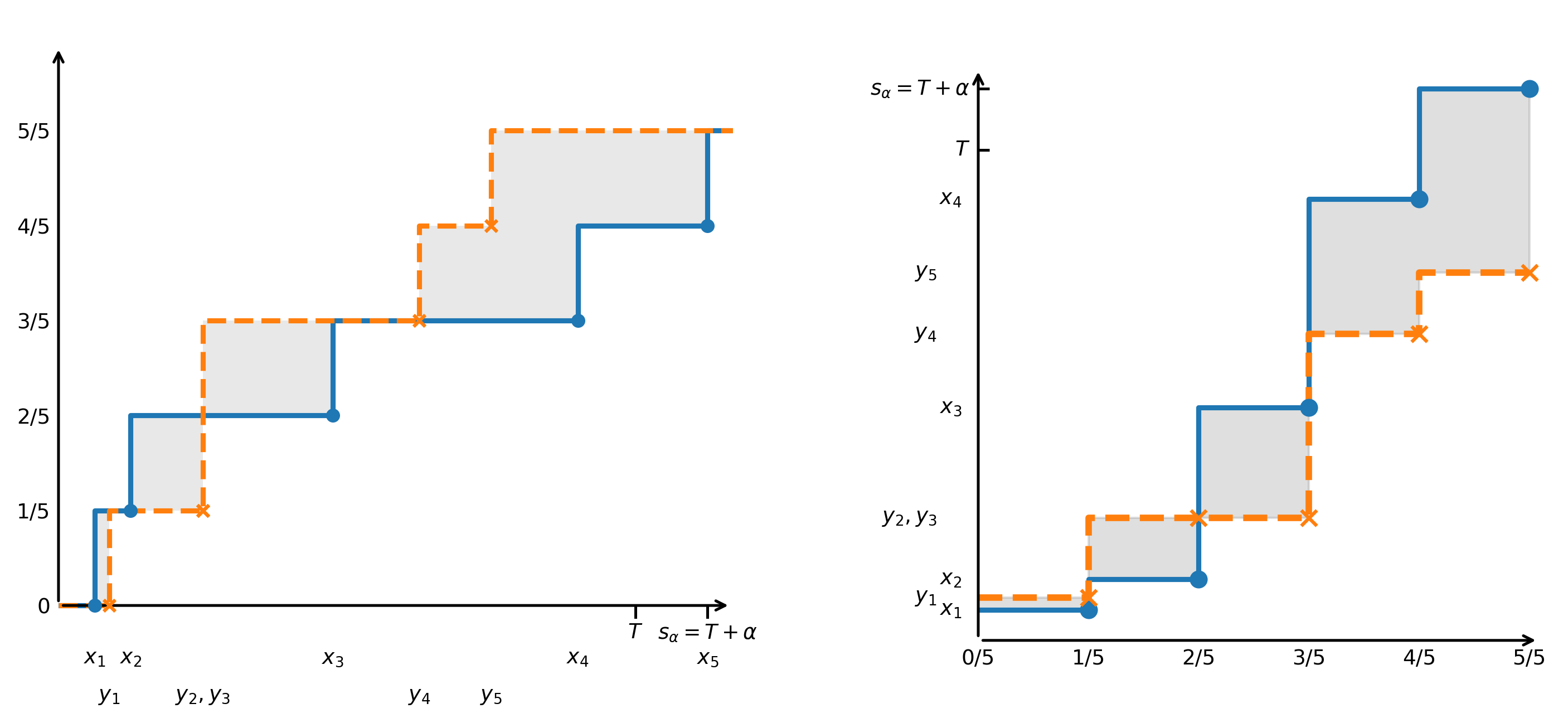}
    \caption{Left Panel: The solid and dashed lines are the cumulative distribution functions for $\mu_1^*$ and $\mu_2$ respectively. The shaded area equals $\int_{0}^{\infty}\!\big|F_{\mu^{*}_1}(x) - F_{\mu_2}(x)\big|\,dx$. Right Panel: The solid and dashed lines are the inverse function for $\mu_1^*$ and $\mu_2$ respectively. The shaded area equals $\int_{0}^{1}\!\big|F_{\mu^{*}_1}^{-1}(u) - F_{\mu_2}^{-1}(u)\big|\,du$.}
    \label{fig:metrc_illustration}
\end{figure}

\subsection{Common structural and topological properties}

Many of the distances above share a common structure. For example, when the cardinalities are equal, they reduce to a matching or Wasserstein-type distance based on the locations of the points, while differences in cardinality are controlled through an additional penalty.

Since our primary interest lies in establishing Wasserstein convergence rates for empirical laws of point processes, rather than in any particular choice of distance, we extract several common structural properties from the distances introduced above, as summarized in Condition~\ref{assume:metric}, and formulate our main results for any distance on $\mathfrak{N}(S)$ satisfying these properties. Consequently, the results developed below apply to a broad class of distances between counting measures with unequal cardinalities.

Furthermore, these structural properties also provide some insights into how metrics for unbalanced point measures should be designed if one wishes to obtain quantitative Wasserstein convergence results. 

\begin{condition}\label{assume:metric}
Let $D$ be a metric on $\mathfrak{N}(S)$. For $\mu\in\mathfrak{N}(S)$, we denote its cardinality by $|\mu|$. There exist finite positive constants $c_D, C_D$ such that the following hold:
\begin{enumerate}[(i)]
\item If
\[
\mu_1=\sum_{i=1}^{m}\delta_{x_i},\qquad
\mu_2=\sum_{i=1}^{m}\delta_{y_i}
\]
have the same cardinality, then
\[
D(\mu_1,\mu_2)
=\min_{\pi\in\Pi_m}\sum_{i=1}^{m}\rho(x_i,y_{\pi(i)}),
\]
where $\rho$ is the underlying metric on $S$.
\item For all $\mu_1,\mu_2\in\mathfrak{N}(S)$,
\[
D(\mu_1,\mu_2)\leq C_D\bigl( |\mu_1|\vee|\mu_2| \bigr).
\]
\item For all $\mu_1,\mu_2\in\mathfrak{N}(S)$,
\[
D(\mu_1,\mu_2)\geq c_D \bigl||\mu_1|-|\mu_2|\bigr|.
\]
\end{enumerate}
\end{condition}

For Condition~\ref{assume:metric} to hold, it is necessary that $c_D\leq C_D$. Indeed, $\bigl||\mu_1|-|\mu_2|\bigr| = \left(|\mu_1|\vee|\mu_2|\right) - \left(|\mu_1|\wedge|\mu_2|\right)$ and the minimum possible value of $|\mu_1|\vee|\mu_2|/ \left[ \left(|\mu_1|\vee|\mu_2|\right) -\left(|\mu_1|\wedge|\mu_2|\right) \right]  $ is $1$. Throughout the remainder of the paper, whenever we consider a metric $D$ on $\mathfrak{N}(S)$, we assume that Condition~\ref{assume:metric} is satisfied.


Several of the distances discussed in the preceding subsections satisfy these structural assumptions. For example, when $\rho\leq 1$, the unnormalized version of $D_S$,
\[
\widetilde D_S(\mu_1,\mu_2)
=
\min_{\pi\in\Pi_m}
\left[
\sum_{i=1}^{n}\rho(x_i,y_{\pi(i)})+(m-n)
\right],
\qquad n\leq m,
\]
satisfies Condition~\ref{assume:metric} with $c_D=C_D=1$. More generally, if $\rho$ is bounded but its upper bound exceeds $1$, the unit penalty assigned to each unmatched point can be replaced by an upper bound for $\rho$. With this modification, the resulting distance is expected to continue to satisfy Condition~\ref{assume:metric}, with $c_D$ and $C_D$ adjusted according to the chosen upper bound on $\rho$.


The order-$1$ TT metric $D_{\mathrm{TT}}$ introduced above, corresponding to the case $p=1$, also satisfies Condition~\ref{assume:metric} with $c_D=C$ and $C_D=2C$, as can be seen from its alternative representation. In particular, this shows that a total-variation penalized special case of the unbalanced optimal transport formulation of \cite{chizat2018scaling} falls within our framework by Remark~3 of \cite{muller2020}. One caveat is that, when applying our following results to the TT metric, the truncated metric $\rho_{\ast}=\rho\wedge 2C$ should be regarded as the underlying metric on $S$. Accordingly, geometric quantities associated with $S$, such as covering numbers and the Minkowski dimension, are understood with respect to $\rho_{\ast}$. Higher-order TT metrics can be treated in a similar manner; however, their $p$-power structure does not conform exactly to Condition~\ref{assume:metric} and requires corresponding modifications to the orders appearing therein. Since introducing these modifications would add notation without altering the main ideas, we restrict our discussion to the order-$1$ case. Nevertheless, we expect that the arguments can be extended to higher-order TT metrics with appropriate adjustments.


In contrast, when $S=[0,T)$ and $\rho$ is the Euclidean metric, the corresponding distance $D_0$ does not, in general, satisfy Condition~\ref{assume:metric}, as it fails to provide the uniform positive lower bound on cardinality discrepancies required by Condition~\ref{assume:metric}(iii). This limitation provides further motivation for introducing the proposed metric $D_1$.

The metric $D_1$ defined in \eqref{eqn:metricD1} satisfies Condition~\ref{assume:metric}. Indeed, when the two counting measures have equal cardinalities, $D_1$ coincides with $D_w$ by construction. Moreover, Theorem~\ref{thm:d1metric} implies that Condition~\ref{assume:metric}(ii) holds with
\[
C_D=\mathrm{diam}(S)+\alpha,
\]
while Condition~\ref{assume:metric}(iii) holds with
\[
c_D=\alpha.
\]



Finally, we establish several key properties of a general metric $D$ satisfying Condition~\ref{assume:metric}. In particular, we show that the metric space $(\mathfrak{N}(S),D)$ is $\sigma$-compact, complete, and separable. These properties are not only of independent interest but also play important technical roles in the subsequent analysis. For example, the $\sigma$-compactness established in the following proposition is essential to the telescoping argument employed later in the paper.

\begin{proposition}\label{topology1}
The following properties hold:
\begin{enumerate}[(i)]
\item The metric $D$ induces the vague topology $\mathcal{T}$ on $\mathfrak{N}(S)$. Furthermore, the $\sigma$-algebra $\mathcal{N}_{\mathcal{T}}$ coincides with the Borel $\sigma$-algebra generated by $D$.
\item For each $m \in \mathbb{N}^{+}$, the subspace $ (\mathfrak{N}(S) \cap \{\mu \in \mathfrak{N}(S): |\mu| \le m\}, D)$ is compact. Consequently, $ (\mathfrak{N}(S), D ) $ is $\sigma$-compact.
\item The metric space $(\mathfrak{N}(S), D)$ is complete and separable.
\end{enumerate}
\end{proposition}

The proof is provided in Appendix \ref{appendix1}.

\section{Wasserstein distance for point processes} 
\label{sec:WDforPointProcess}

In Section~\ref{sec: ANewMetric}, we identified the general conditions on a metric $D$ required for our analysis and studied the corresponding topological properties. Our next objective is to use the metric $D$ as a building block for quantifying discrepancies between point processes. Given a probability space $(\Omega, \mathcal{F}, P)$, a point process can be viewed as a random counting measure from $\Omega$ to $\mathfrak{N}$. This allows us to define a Wasserstein-type distance between two point processes. 

Let $\mathcal{P}(\mathfrak{N})$ denote the space of Borel probability measures on $(\mathfrak{N}, \mathcal{N}_{\mathcal{T}})$. Since a point process is a random element taking values in $\mathfrak{N}$, its distribution can naturally be identified with an element of $\mathcal{P}(\mathfrak{N})$. For $Q_1,Q_2\in\mathcal{P}(\mathfrak{N})$ and $p\in[1,\infty)$, the Wasserstein distance of order $p$ between $Q_1$ and $Q_2$, induced by the metric $D$ on $\mathfrak{N}$, is defined by
\begin{align*}
W_p(Q_1, Q_2) &=  \left(\inf_{\gamma \in \Upsilon(Q_1,Q_2)} \int_{\mathfrak{N} \times \mathfrak{N} } D^p (\mu,\nu) d\gamma(\mu,\nu)  \right)^{\frac{1}{p}}  \\
&= \left( \inf_{\gamma \in \Upsilon(Q_1,Q_2)} E_{\mu,\nu \sim \gamma} \left[ D^p (\mu,\nu) \right] \right)^{ \frac{1}{p} },
\end{align*}
where $\Upsilon (Q_1, Q_2)$ denotes the collections of measures on $ \mathfrak{N} \times \mathfrak{N} $ whose marginal measures are $Q_1$ and $Q_2$, respectively. It is a standard result that $W_p$, as defined above, is a valid metric (see, for instance, \cite[Chapter 7]{villani2021topics}). Moreover, if $\eta_1, \eta_2: \Omega \rightarrow \mathfrak{N}$ are two point processes, and let $\mathcal{L}(\eta_1)$ and $\mathcal{L}(\eta_2)$ denote the distributions (laws) of $\eta_1$ and $\eta_2$, respectively, then we have
\begin{equation*}
W_p(\mathcal{L}(\eta_1),\mathcal{L}(\eta_2)) = \left( \inf_{\gamma \in \Upsilon(\mathcal{L}(\eta_1),\mathcal{L}(\eta_2))} \int_{\mathfrak{N} \times \mathfrak{N} } D^p (\mu,\nu) d\gamma(\mu,\nu) \right)^{\frac{1}{p}}.
\end{equation*}

The following proposition characterizes the relationship between convergence with respect to the distance $W_p$ and weak convergence on $\mathcal{P}(\mathfrak{N})$.

\begin{proposition}\label{topology2} 
Suppose that $\eta, \eta_1, \eta_2, \dots$ are point processes.
The convergence $W_p^p(\mathcal{L}(\eta_n),\mathcal{L}(\eta)) \rightarrow 0$ implies that $\eta_n$ converges weakly (or, in distribution) to $\eta$, i.e., $E[f(\eta_n)] \rightarrow E[f(\eta)]$ for every continuous and bounded $f: (\mathfrak{N},\mathcal{T}) \rightarrow \mathbb{R}$. Conversely, if $\sup_{n}\sum_{i=0}^{\infty} i^{p_{1}} P(|\eta_{n}| = i) < \infty$ for some $p_1 > p$, then weak convergence of $\eta_n$ to $\eta$ implies $W_p^p(\mathcal{L}(\eta_n),\mathcal{L}(\eta)) \rightarrow 0$.
\end{proposition}

\begin{proof}
The first statement is a well-known result, since $\mathfrak{N}$ is a separable metric space. A proof can be found, for example, in Proposition 7.1.5 of \citep{GradientFlows2008}. 

For the second statement, fix $\mu_0 \in \mathfrak{N}$ with $|\mu_0| = m$. Then, by Condition~\ref{assume:metric} (ii) and the assumption $\sup_{n}\sum_{i=0}^{\infty} i^{p_{1}} P(|\eta_{n}| = i) < \infty$, we have
\begin{eqnarray*}
&&\sup_{n}\int_{\mathfrak{N}}D^{p_1}(\mu,\mu_0)d \mathcal{L}(\eta_{n})(\mu) = \sup_{n}\sum_{i=0}^{\infty} \int_{\{\mu: |\mu| = i\}}D^{p_1}(\mu,\mu_0)d \mathcal{L}(\eta_{n})(\mu)\\
&=& \sup_{n} \left( \sum_{i=0}^{m} \int_{\{\mu: |\mu| = i\}}D^{p_1}(\mu,\mu_0)d \mathcal{L}(\eta_{n})(\mu) + \sum_{i=m+1}^{\infty} \int_{\{\mu: |\mu| = i\}}D^{p_1}(\mu,\mu_0)d \mathcal{L}(\eta_{n})(\mu) \right) \\
&\leq& \sup_{n} \left( \sum_{i=0}^{m} \int_{\{\mu: |\mu| = i\}}m^{p_1} C_D^{p_1} d \mathcal{L}(\eta_{n})(\mu) + \sum_{i=m+1}^{\infty} \int_{\{\mu: |\mu| = i\}}i^{p_1} C_D^{p_1} d \mathcal{L}(\eta_{n})(\mu)  \right)\\
&=& \sup_{n} \left( m^{p_1} C_D^{p_1} \sum_{i=0}^{m}  P(|\eta_{n}| = i) + \sum_{i=m+1}^{\infty} i^{p_1} C_D^{p_1} P(|\eta_{n}| = i)  \right)\\
&\leq& \sup_{n}  \left(  m^{p_1} C_D^{p_1} + C_D^{p_1} \sum_{i=0}^{\infty} i^{p_1} P(|\eta_{n}| = i)  \right)\\
&=& m^{p_1} C_D^{p_1}  + C_D^{p_1}\sup_{n} \sum_{i=0}^{\infty} i^{p_1} P(|\eta_{n}| = i) \\
&<& \infty.
\end{eqnarray*}
This implies that $\{\mathcal{L}(\eta_n): n = 1,2,\dots\}$ has uniformly integrable $p$-moments. By Proposition 7.1.5 of \citep{GradientFlows2008}, the proof is concluded.
\end{proof}

Let $\mathcal{P}_p(\mathfrak{N})$ be the collection of probability measures $Q$ on $\mathfrak{N}$ with $p$-th finite moment, that is, 
\begin{equation*}
\int_{\mathfrak{N}} D^{p}(\mu, \mu_0) \, d Q(\mu) < \infty
\end{equation*}
for some (and hence all) $\mu_0 \in \mathfrak{N}$. Then, we have $\mathcal{L}(\eta) \in \mathcal{P}_p(\mathfrak{N})$ if and only if 
\begin{equation} \label{pp_space}
\sum_{i=0}^{\infty} i^{p} P(|\eta| = i) < \infty.    
\end{equation}
The `if' direction is a byproduct of the proposition above. For the `only if' direction, by Condition~\ref{assume:metric} (iii),
\begin{eqnarray*}
&&\int_{\mathfrak{N}}D^{p}(\mu,\mu_0)d \mathcal{L}(\eta)(\mu) = \sum_{i=0}^{\infty} \int_{\{\mu: |\mu| = i\}}D^{p}(\mu,\mu_0)d \mathcal{L}(\eta)(\mu)\\
&=&  \left( \sum_{i=0}^{m} \int_{\{\mu: |\mu| = i\}}D^{p}(\mu,\mu_0)d \mathcal{L}(\eta)(\mu) + \sum_{i=m+1}^{\infty} \int_{\{\mu: |\mu| = i\}}D^{p}(\mu,\mu_0)d \mathcal{L}(\eta)(\mu) \right) \\
&\geq& \sum_{i=m+1}^{\infty} \int_{\{\mu: |\mu| = i\}}D^{p}(\mu,\mu_0)d \mathcal{L}(\eta)(\mu) \geq \sum_{i=m+1}^{\infty} \int_{\{\mu: |\mu| = i\}}(i-m)^{p} c_D^{p}d \mathcal{L}(\eta)(\mu)\\
&=& c_D^{p} \sum_{i=m+1}^{\infty} (i-m)^{p}  P(|\eta| = i).
\end{eqnarray*}
Since $m$ is fixed, there exists some constant $C>0$ such that $(i-m)^{p} \geq C \, i^{p}$ for all $i \geq m+1$, which yields the desired conclusion.

\section{Upper bounds for the Wasserstein distance} \label{sec:upperbound}

In this section, we study the convergence of empirical measures derived from samples of point processes. Let $\eta_1, \dots, \eta_n$ be $n$ independent and identically distributed (i.i.d.) realizations of a finite point process $\eta$. The corresponding empirical measure is defined by
\begin{equation*}
 \widehat{\mathcal{L}}_{n}(\eta)   := \frac{1}{n} \sum_{i=1}^{n} \delta_{\eta_i},
\end{equation*}
which assigns equal mass $1/n$ to each observed sample $\eta_i$, for $i = 1, \dots, n$. Indeed, $\widehat{\mathcal{L}}_{n}(\eta)$ is a random (probability) measure defined on the measurable space $(\mathfrak{N}, \mathcal{N}_{\mathcal{T}})$, and thus can be viewed as a random element taking values in the space of probability measures $\mathcal{P}(\mathfrak{N})$. Furthermore, $\widehat{\mathcal{L}}_{n}(\eta)$ serves as a natural estimator of the distribution (law) of $\eta$.

Our goal is to study the convergence behavior of the empirical measure $\widehat{\mathcal{L}}_{n}(\eta)$ toward the distribution $\mathcal{L}(\eta)$ of $\eta$. By Varadarajan’s Theorem (see \cite{varadarajan1958convergence} or \cite[Theorem 11.4.1]{Dudley_2002}) and Proposition~\ref{topology1}, it can be seen that $\widehat{\mathcal{L}}_{n}(\eta)$, viewed as a random element in the space $\mathcal{P}(\mathfrak{N})$, converges almost surely to $\mathcal{L}(\eta)$ with respect to both the vague and weak topologies.

In this paper, we go beyond this qualitative result by quantifying the rate of convergence in terms of the Wasserstein distance introduced in Section~\ref{sec:WDforPointProcess}. Specifically, we aim to establish an upper bound for
\begin{equation*}
    E \left[ W_p \left( \widehat{\mathcal{L}}_{n}(\eta), \mathcal{L}(\eta) \right) \right].
\end{equation*}

\subsection{Covering numbers}
The purpose of this subsection is to introduce a notion of covering number, which will be applied in bounding above the rates of convergence of the empirical measures for samples of point processes.

\begin{defin}
Let $(X, r)$ be a metric space, where $r$ denotes the metric on $X$. For any $x \in X$ and $\varepsilon > 0$, we define the open ball of radius $\varepsilon/2$ centered at $x$ as $B_{r}(x, \varepsilon/2) := \{y \in X : r(x, y) < \varepsilon/2\}$, which has diameter $\varepsilon$. Similarly, we denote by $\overline{B_{r}(x, \varepsilon/2)} := \{y \in X : r(x, y) \leq \varepsilon/2\}$ the corresponding closed ball centered at $x$ with the same diameter. 
For any subset $Y \subseteq X$, the $\varepsilon$-covering number of $Y$, denoted by $\mathcal{M}_{\varepsilon}(Y)$, is the smallest integer $m$ such that there exist closed balls $B_1, \dots, B_m$ of diameter $\varepsilon$ with $Y \subseteq \bigcup_{1 \leq i \leq m} B_i$.
\end{defin}

When a measure is defined on the space $X$, it is often acceptable to disregard sets of small measure without significantly affecting the overall analysis. This motivates the introduction of a measure-sensitive notion of covering number, as introduced in \cite{dudley1969speed,weed2019sharp}.

\begin{defin}
Given a measure $\mu$ on $X$, the $(\varepsilon, \tau)$-covering number is defined as
\begin{equation*}
    \mathcal{M}_{\varepsilon}(\mu, \tau) := \inf \{ \mathcal{M}_{\varepsilon}(Y) :  \mu(Y) \geq 1- \tau \}.
\end{equation*}
\end{defin}

We now introduce the Minkowski dimension, also referred to as the Minkowski–Bouligand dimension or box-counting dimension.

\begin{defin}
The Minkowski dimension of a set $Y \subseteq X$, denoted by $\mathrm{dim}_{M}(Y)$, is defined as
\begin{equation*}
\mathrm{dim}_{M}(Y) := \limsup_{\varepsilon \rightarrow 0} \frac{ \log \mathcal{M}_{\varepsilon} (Y) }{- \log \varepsilon}.
\end{equation*}
\end{defin}

With the above preparation, we are ready to evaluate the $\varepsilon$-covering number of the subspace $\mathfrak{N}_{m}$ of $\mathfrak{N}$, defined as $\mathfrak{N}_{m} := \mathfrak{N} \cap \{ \mu \in \mathfrak{N} : |\mu| =m \}$, for $m =0,1 , 2,\dots$.

\begin{lemma}   \label{lm:coveringnumber}
For $m \in \mathbb{N}^{+}$ and $\varepsilon >0$, we have
\begin{equation}
\mathcal{M}_{\varepsilon} \left( \mathfrak{N}_{m} \right) \leq 
e^{m} \left( 1 + \frac{\mathcal{M}_{\varepsilon/m}(S)}{m}  \right)^{m}.  \label{eqc:coveringnumberbound}
\end{equation}
Furthermore, if $\mathrm{dim}_{M}(S) < \infty $, then for any $\kappa >0$, there exists $\varepsilon' >0$ such that for $ 0 <\varepsilon \leq \varepsilon'$ and $m \in \mathbb{N}^{+}$,
\begin{equation*}
\mathcal{M}_{\varepsilon} \left( \mathfrak{N}_{m} \right) \leq 
e^{m} \left( 1 +  m^{\mathrm{dim}_{M}(S) + \kappa -1  }  \varepsilon^{ - ( \mathrm{dim}_{M}(S) + \kappa) }  \right)^{m}.
\end{equation*}
\end{lemma}

\begin{proof}
Consider the product space $S^{m}$ with the metric $\rho^{m}$ defined as 
\begin{equation*}
\rho^{m} \left( (x_1, \dots, x_m), (y_1, \dots, y_m) \right)
= \sum_{i=1}^{m} \rho(x_i, y_i),
\end{equation*}
for $(x_1, \dots, x_m), (y_1, \dots, y_m) \in S^{m}$. We define the following mapping from $(S^{m}, \rho^{m})$ to $(\mathfrak{N}_{m}, D)$:
\begin{equation*}
\begin{array}{ccccl}
\phi \ \  \colon & (S^{m}, \rho^{m}) & \longrightarrow &(\mathfrak{N}_{m}, D) &\\
 &(x_1, \dots, x_m) &\mapsto & \sum_{i=1}^{m} \delta_{x_i} & .
\end{array}
\end{equation*}
It is clear that $\phi$ is surjective. Moreover, $\phi$ is a Lipschitz continuous mapping with Lipschitz constant equal to $1$. Indeed, by Condition~\ref{assume:metric} (i),
\begin{eqnarray*}
&& D \left( \phi\left( (x_1, \dots, x_m)\right) , \phi\left( (y_1, \dots, y_m)\right) \right)
= \min_{\pi \in \Pi_m} \sum_{i=1}^{m} \rho \left(x_i, y_{\pi(i)} \right) \\
&\leq & \sum_{i=1}^{m} \rho \left(x_i, y_i \right)
= \rho^{m} \left( (x_1, \dots, x_m), (y_1, \dots, y_m) \right).
\end{eqnarray*}

We then evaluate $\mathcal{M}_{\varepsilon} \left( \mathfrak{N}_{m} \right)$ using $\mathcal{M}_{\varepsilon/m}(S)$. For simplicity of notation we denote $\mathcal{M}_{\varepsilon/m}(S)$ by $k$. It follows by the definition of $\mathcal{M}_{\varepsilon/m}(S)$ that there exist $k$ closed balls $B_1, \dots, B_k$ of diameter $\varepsilon/m$ such that $B_1, \dots, B_k$ cover $S$. Suppose that $B_j$ is centered at $z_j$ for $j = 1, \dots, k$.

For any $(x_1, \dots, x_m) \in S^{m}$, we have $x_i \in S$. Since $ S \subseteq \bigcup_{j=1}^{k} B_j $, it follows that $x_i \in B_{j_i}$ for some $1 \leq j_i \leq k$, and thus,
\begin{equation*}
    \rho\left( x_i, z_{j_i} \right) \leq \frac{\varepsilon}{2m}.
\end{equation*}
Therefore,
\begin{eqnarray*}
\rho^{m} \left( (x_1, \dots, x_m), \left( z_{j_1}, \dots, z_{j_m} \right)  \right)
= \sum_{i=1}^{m} \rho\left( x_i, z_{j_i} \right) 
\leq \sum_{i=1}^{m} \frac{\varepsilon}{2m} = \frac{\varepsilon}{2}.
\end{eqnarray*}
This implies that
\begin{equation*}
(x_1, \dots, x_m) \in \overline{ B_{\rho^{m}} \left( \left( z_{j_1}, \dots, z_{j_m} \right), \frac{\varepsilon}{2}  \right)  },
\end{equation*}
and thus,
\begin{equation*}
S^{m} \subseteq \bigcup_{ j_1, \cdots, j_m =1 }^{k}  \overline{ B_{\rho^{m}} \left( \left( z_{j_1}, \dots, z_{j_m} \right), \frac{\varepsilon}{2}  \right)  }.
\end{equation*}
Recalling the definition and properties of the mapping $\phi$, we have
\begin{eqnarray*}
    \mathfrak{N}_{m} = \phi\left( S^{m} \right) 
    &\subseteq& \phi\left(  \bigcup_{ j_1, \cdots, j_m =1 }^{k}  \overline{ B_{\rho^{m}} \left( \left( z_{j_1}, \dots, z_{j_m} \right), \frac{\varepsilon}{2}  \right)  }  \right)  \\
    &\subseteq& \bigcup_{ j_1, \cdots, j_m =1 }^{k} \phi\left( \overline{ B_{\rho^{m}} \left( \left( z_{j_1}, \dots, z_{j_m} \right), \frac{\varepsilon}{2}  \right)  }  \right)  \\
     &\subseteq& \bigcup_{ j_1, \cdots, j_m =1 }^{k} \overline{  B_{D} \left( \phi\left( \left( z_{j_1}, \dots, z_{j_m} \right) \right), \frac{\varepsilon}{2}  \right) },
\end{eqnarray*}
where the last inclusion follows by the fact that $\phi$ is a Lipschitz continuous mapping with Lipschitz constant equal to $1$.
    
Therefore, to evaluate the $\varepsilon$-covering number of $\mathfrak{N}_{m}$, it suffices to count the number of closed balls of the form $ \overline{  B_{D} \left( \phi( ( z_{j_1}, \dots, z_{j_m} )), \varepsilon/2   \right) } $. However, some of these balls may be identical. In particular, if two index multisets $\{ j_1, \dots, j_m \}$ and $\{ l_1, \dots, l_m \}$ are equal (i.e., they contain the same elements with the same multiplicities, possibly in a different order), then the corresponding balls coincide, since $\phi( ( z_{j_1}, \dots, z_{j_m} )) = \phi( ( z_{l_1}, \dots, z_{l_m} ))$. As a result, the total number of distinct such balls is equal to the number of distinct multisets of size $m$ formed from the set $\{1, 2, \dots, k\}$. This is a classical combinatorial problem of sampling with replacement, and the number of such multisets is given by:
\begin{equation*}
\binom{k+m-1}{k-1} = \binom{k+m-1}{m}.
\end{equation*}
Thus, it follows
\begin{eqnarray*}
\mathcal{M}_{\varepsilon} \left( \mathfrak{N}_{m} \right) \leq \binom{k+m-1}{m} 
\leq  \frac{ (k+m)^{m} }{m!} \leq \frac{ (k+m)^{m} }{ \sqrt{2 \pi m} \, e^{-m} m^m },
\end{eqnarray*}
where the last inequality results from quantitative forms of the DeMoivre-Stirling Theorem (see, for instance, \cite[Theorem 5.44]{Stromberg1981}). Noting that $k = \mathcal{M}_{\varepsilon/m}(s)$, \eqref{eqc:coveringnumberbound} follows by a direct computation.

Finally, we turn to the proof of the second assertion. Since $\mathrm{dim}_{M}(S) < \infty $, it follows by the definition of Minkowski dimension that for any $\kappa >0$, there exists $\varepsilon'>0$ such that for any $0 < \varepsilon \leq \varepsilon'$ and any $m \in \mathbb{N}^{+}$
\begin{equation*}
\frac{ \log \mathcal{M}_{ \varepsilon/m }(S) }{ - \log (\varepsilon/m) } < \mathrm{dim}_{M}(S) + \kappa,
\end{equation*}
namely,
\begin{equation*}
\mathcal{M}_{ \varepsilon/m }(S) < \left(  \frac{m}{\varepsilon} \right)^{  \mathrm{dim}_{M}(S) + \kappa }.
\end{equation*}
Plugging the above result into \eqref{eqc:coveringnumberbound}, the desired result follows.
\end{proof}

\subsection{Truncation and telescope decomposition}  \label{subsec:truncationandtelescope}
Note that $ \eta $ is a random element taking values in the metric space $(\mathfrak{N}, D)$, which is neither finite-dimensional nor bounded. Consequently, directly analyzing the rate of convergence of the corresponding empirical measure is challenging. To address this, we begin this subsection by applying a truncation argument, restricting the range of $ \eta $ to the compact subset $ \{ \mu \in \mathfrak{N} : |\mu| \leq m \} $ under the metric $ D $.

Let $\mathbf{0}$ denote the zero counting measure on $S$, that is, $\mathbf{0}(A) = 0$ for every measurable set $A \subseteq S$. Define the mapping $\Phi_m : \mathfrak{N} \to \{ \mu \in \mathfrak{N} : |\mu| \leq m \}$ by
\begin{equation*}
\Phi_m(\mu) =
\begin{cases}
\mu, & \text{if } |\mu| \leq m, \\
\mathbf{0}, & \text{if } |\mu| > m.
\end{cases}
\end{equation*}
Recall that $\eta_1, \dots, \eta_n$ are i.i.d. realizations of the point process $\eta$. Define the truncated versions $\xi = \Phi_m(\eta)$ and $\xi_i = \Phi_m(\eta_i)$ for $i = 1, 2, \dots, n$, where each $\xi_i$ takes values in the compact subset $\{ \mu \in \mathfrak{N} : |\mu| \leq m \}$ of $\mathfrak{N}$. Based on these, we define the corresponding empirical measure:
\begin{equation*}
\widehat{\mathcal{L}}_{n}(\xi)  := \frac{1}{n} \sum_{i=1}^{n} \delta_{\xi_i}.
\end{equation*}

The following lemma establishes an evaluation of the Wasserstein distance between $\widehat{\mathcal{L}}_{n}(\eta)$ and $\mathcal{L}(\eta)$ in terms of the Wasserstein distance between $\widehat{\mathcal{L}}_{n}(\xi)$ and $\mathcal{L}(\xi)$, along with the tail behavior of $\eta$.

\begin{lemma}   \label{lm:truncation}
For $m \in \mathbb{N}^{+}$, we have
\begin{equation*}
E \left[ \left|  W_{p}\left( \widehat{\mathcal{L}}_{n}(\eta) , \mathcal{L} (\eta) \right)   -  W_{p} \left( \widehat{\mathcal{L}}_{n}(\xi),  \mathcal{L} (\xi) \right) \right| \right]
\leq 2 C_D \left( E\left[ |\eta|^{p} \mathds{1}_{  \{ |\eta| >m \} } \right] \right)^{\frac{1}{p}}.
\end{equation*}
\end{lemma}

\begin{proof}
Since $W_p$ is a metric, it follows by the triangle inequality that
\begin{equation}
\label{eqb:triangleinequality}
\left|  W_{p} \left( \widehat{\mathcal{L}}_{n}(\eta) , \mathcal{L}(\eta) \right)   -  W_{p} \left( \widehat{\mathcal{L}}_{n}(\xi),  \mathcal{L}(\xi) \right) \right| \leq  W_{p} \left( \widehat{\mathcal{L}}_{n}(\eta) ,  \widehat{\mathcal{L}}_{n}(\xi) \right) + W_{p} \left( \mathcal{L}(\eta), \mathcal{L}( \xi) \right).  
\end{equation}
Denote by $\Xi$ the probability measure on $\mathfrak{N}$ induced by $\eta$. Consider the probability measure $ \Xi \circ \left( \mathrm{Id} \times \Phi_{m} \right)^{-1} $ on $ \mathfrak{N} \times \mathfrak{N} $, which is the pushforward of $\Xi$ under the mapping:
\begin{equation*}
\begin{array}{ccccl}
\mathrm{Id} \times \Phi_{m} \ \  \colon & ( \mathfrak{N}, \mathcal{N}_{\mathcal{T}} ) & \longrightarrow &(\mathfrak{N} \times \mathfrak{N} , \mathcal{N}_{\mathcal{T}} \otimes \mathcal{N}_{\mathcal{T}}) &\\
 & \mu &\mapsto &  \left( \mu , \Phi_{m}( \mu )\right)  & .
\end{array}
\end{equation*}
It can be easily verified that the marginal probability measures of $ \Xi \circ \left( \mathrm{Id} \times \Phi_{m} \right)^{-1} $ are $\Xi$ and $ \Xi \circ \Phi_{m}^{-1} $, which correspond to the laws of $\eta$ and $\xi$, respectively. Therefore, $ \Xi \circ \left( \mathrm{Id} \times \Phi_{m} \right)^{-1} $ belongs to $ \Upsilon \left( \mathcal{L}(\eta), \mathcal{L}(\xi) \right) $. By the definition of the Wasserstein distance $W_p$, we have
\begin{align}
W_{p}^{p}\left( \mathcal{L}(\eta), \mathcal{L}(\xi)  \right) 
& \leq \int_{ \mathfrak{N} \times \mathfrak{N} } D^{p} ( \mu, \nu ) \, d \left( \Xi \circ \left( \mathrm{Id} \times \Phi_{m} \right)^{-1} \right)(\mu, \nu) 
 = \int_{ \mathfrak{N} } D^{p} \left( \mu , \Phi_{m}(\mu) \right) \, d \Xi (\mu) \notag \\
& = \int_{ \mathfrak{N} } D^{p} \left( \mu , \mathbf{0} \right) \mathds{1}_{ \{ |\mu| > m \} } \, d \Xi (\mu)
\leq \int_{ \mathfrak{N} } |\mu|^{p} C_D^{p} \mathds{1}_{ \{ |\mu| > m \} } \, d \Xi (\mu) \notag  \\
&= C_D^{p} \, E\left[ |\eta|^{p} \mathds{1}_{ \{ |\eta| > m \} } \right], \label{eqb:boundtermdirect}
\end{align}
where the first equality follows from the change-of-variables formula for pushforward measures, the second equality uses the definition of $\Phi_m$, and the last inequality is a consequence of Condition~\ref{assume:metric} (ii).

Furthermore, observe that both $\widehat{\mathcal{L}}_{n}(\eta)$ and $\widehat{\mathcal{L}}_{n}(\xi)$ are discrete probability measures supported on $n$ points in $\mathfrak{N}$, each assigning equal mass to their support elements. It follows by a standard result of Wasserstein distance (see, for instance, \cite[Lemma 3.1.A]{schuhmacher2005}) that
\begin{equation*}
W_{p}^{p} \left( \widehat{\mathcal{L}}_{n}(\eta), \widehat{\mathcal{L}}_{n}(\xi) \right) 
= \frac{1}{n} \min_{ \pi \in \Pi_n }  \sum_{i=1}^{n} D^{p} \left( \eta_i, \xi_{ \pi(i) } \right)
\leq \frac{1}{n} \sum_{i=1}^{n} D^{p} \left( \eta_i, \xi_{ i } \right).
\end{equation*}
Since $\xi_i = \Phi_{m}(\eta_i)$, it follows again by the definition of $\Phi_m$ that
\begin{align*}
E\left[ D^{p}\left(  \eta_i, \xi_{ i } \right) \right]
&= E \left[ D^{p} \left( \eta_i, \mathbf{0} \right) \mathds{1}_{ \{ |\eta_i| > m \} } \right]
\leq E\left[ |\eta_i|^{p} C_D^{p} \mathds{1}_{ \{ |\eta_i| > m \} } \right] \\
&= C_D^{p} \, E\left[ |\eta_i|^{p} \mathds{1}_{ \{ |\eta_i| > m \} } \right]
= C_D^{p} \, E\left[ |\eta|^{p} \mathds{1}_{ \{ |\eta| > m \} } \right].
\end{align*}
Thus,
\begin{equation}
E\left[ W_{p}^{p} \left( \widehat{\mathcal{L}}_{n}(\eta), \widehat{\mathcal{L}}_{n}(\xi) \right)  \right] \leq C_D^{p} \, E\left[ |\eta|^{p} \mathds{1}_{ \{ |\eta| > m \} } \right]. \label{eqb:boundtermempirical}
\end{equation}
Combining \eqref{eqb:triangleinequality}, \eqref{eqb:boundtermdirect}, and \eqref{eqb:boundtermempirical}, the desired result follows.
\end{proof}

We next examine the Wasserstein distance $W_p ( \widehat{\mathcal{L}}_{n}(\xi), \mathcal{L}(\xi) )$. Since the law $\mathcal{L}(\xi)$ is supported on the compact subset $\{ \mu \in \mathfrak{N} : |\mu| \leq m \} = \bigcup_{j=0}^{m} \mathfrak{N}_{j} $, we employ a telescoping decomposition to analyze $W_p ( \widehat{\mathcal{L}}_{n}(\xi), \mathcal{L}(\xi) )$. The details are summarized in the following lemma, whose proof is similar to those of \cite{LeiJing2020Caco,fournier2015rate} and is contained in Appendix~\ref{appendix3}

\begin{lemma} \label{lm:resultaftertelescope}
For each $j = 0, 1, 2, \dots, m$, if $\mathcal{L}(\xi)(\mathfrak{N}_j) > 0$, we define $\mathcal{L}(\xi)^{\mathfrak{N}_j}$ as the probability measure obtained by restricting $\mathcal{L}(\xi)$ to $\mathfrak{N}_j$ and normalizing. Specifically, for any $\mathfrak{B} \in \mathcal{N}_{\mathcal{T}}$, $\mathcal{L}(\xi)^{\mathfrak{N}_j}(\mathfrak{B}) := \mathcal{L}(\xi)(\mathfrak{B} \cap \mathfrak{N}_j) / \mathcal{L}(\xi)(\mathfrak{N}_j)$. If instead $\mathcal{L}(\xi)(\mathfrak{N}_j) = 0$, then $\mathcal{L}(\xi)^{\mathfrak{N}_j}$ is defined as any probability measure supported on $\mathfrak{N}_j$.
Similarly, if $\widehat{\mathcal{L}}_{n}(\xi)(\mathfrak{N}_j) > 0$, we define $\widehat{\mathcal{L}}_{n}(\xi)^{\mathfrak{N}_j}$ by restricting $\widehat{\mathcal{L}}_{n}(\xi)$ to $\mathfrak{N}_j$ and normalizing in the same way. Otherwise, if $\widehat{\mathcal{L}}_{n}(\xi)(\mathfrak{N}_j) = 0$, we set $\widehat{\mathcal{L}}_{n}(\xi)^{\mathfrak{N}_j} := \mathcal{L}(\xi)^{\mathfrak{N}_j}$. Then, we have
\begin{align*}
W_{p}^{p} \left( \widehat{\mathcal{L}}_{n}(\xi), \mathcal{L}(\xi) \right)
\leq& \sum_{j=1}^{m} \left( \widehat{\mathcal{L}}_{n}(\xi)(\mathfrak{N}_{j}) \wedge \mathcal{L}(\xi)(\mathfrak{N}_{j}) \right) W_{p}^{p} \left( \widehat{\mathcal{L}}_{n}(\xi)^{\mathfrak{N}_j}, \mathcal{L}(\xi)^{\mathfrak{N}_j} \right) \mathds{1}_{\{ \widehat{\mathcal{L}}_{n}(\xi)(\mathfrak{N}_j) >0 \}}  \\
 & +  2^{p} C_D^{p} \sum_{j=1}^{m} j^{p} \left| \widehat{\mathcal{L}}_{n}(\xi)(\mathfrak{N}_{j}) - \mathcal{L}(\xi)(\mathfrak{N}_{j}) \right|.
\end{align*}
\end{lemma}

In the next step, we aim to quantify $W_{p}^{p} ( \widehat{\mathcal{L}}_{n}(\xi)^{\mathfrak{N}_j}, \mathcal{L}(\xi)^{\mathfrak{N}_j} )$, which, in turn, improves the results in Lemma~\ref{lm:resultaftertelescope}. The refined result is summarized in Lemma~\ref{lm:exactboundWassersteinxi}, and its proof is relegated to Appendix~\ref{appendix3}. Furthermore, our theoretical development relies on Propositions~1 and~3 from \cite{weed2019sharp}, which are restated in Appendix~\ref{appendix2} for completeness and the reader’s convenience.

\begin{lemma}  \label{lm:exactboundWassersteinxi}
For any $k_{1}^{*}, \dots, k_{m}^{*} \in  \{0\}  \cup \mathbb{N}^{+}$ and $p \geq 1$, we have
\begin{align*}
E\left[  W_{p}^{p} \left( \widehat{\mathcal{L}}_{n}(\xi), \mathcal{L}(\xi) \right) \right] \leq& 
2 \times 3^{p} ( 1 \vee C_D )^{p} \sum_{j=1}^{m}  j^{p}\Bigg[  3^{-k_{j}^{*} p}  \mathcal{L}( \xi )( \mathfrak{N}_j ) \\
&  \qquad \qquad \qquad + \sum_{k=0}^{k^{*}_{j}} 3^{-kp} \left( 2 \mathcal{L}( \xi ) ( \mathfrak{N}_j ) \right) \wedge \left( \sqrt{ \frac{ \mathcal{M}_{3^{-(k+1)}}( \mathfrak{N}_{j} ) \mathcal{L}( \xi ) ( \mathfrak{N}_j ) }{n} } \right) \Bigg].
\end{align*}
\end{lemma}

\subsection{An upper bound}
In this subsection, we derive an upper bound on the expected Wasserstein distance between $\widehat{\mathcal{L}}_{n}(\eta)$ and $\mathcal{L}(\eta)$, using the results established in Subsection \ref{subsec:truncationandtelescope}. To proceed, we impose the following assumptions:

\begin{enumerate}[label=(A\arabic*)]
\item The Minkowski dimension of $S$ satisfies $0 < \mathrm{dim}_{M}(S) < \infty$;
\item There exist constants $K_1 > 0$ and $\lambda > 0$ such that, for every $m \in \{0\} \cup \mathbb{N}^{+}$, the tail probability satisfies $P ( | \eta | = m) \leq K_1 e^{ - \lambda m}$.
\end{enumerate}

\begin{lemma}  \label{lm:onetermevaluation}
Suppose Assumptions (A1) and (A2) hold. For any $\kappa >0$, define
\begin{equation}
\overline{l}_{j} := \left\lfloor  \frac{1}{\mathrm{dim}_{M}(S) + \kappa}  \frac{ \log \left(n^{\frac{1}{j}} / (2 e^{\lambda+1} j^{ \mathrm{dim}_{M}(S) + \kappa })\right)  }{ \log 3} \right\rfloor -1  \label{equp:loverline}
\end{equation}
and $ k_{j}^{*} :=  \overline{l}_{j} \vee n$, for $j =1, \dots, m$. Then, there exist constants $c$ and $C$ depending on $ \mathrm{dim}_{M}(S) $, $K_1$, $\kappa$, $p$, and $\lambda$ such that, for every $1 \leq j \leq m$
\begin{eqnarray*}
&& 3^{-k_{j}^{*} p}  \mathcal{L}( \xi )( \mathfrak{N}_j )  + \sum_{k=0}^{k^{*}_{j}} 3^{-kp} \left( 2 \mathcal{L}( \xi ) ( \mathfrak{N}_j ) \right) \wedge \left( \sqrt{ \frac{ \mathcal{M}_{3^{-(k+1)}}( \mathfrak{N}_{j} ) \mathcal{L}( \xi ) ( \mathfrak{N}_j ) }{n} } \right) \\
&\leq& C\left( e^{p \log m} e^{-2 \sqrt{  \frac{ p\lambda}{\mathrm{dim}_{M}(S) + \kappa}  } \sqrt{\log n} } + n^{-1/2} \log n + n^{-1/2} e^{c m \log m} \right).
\end{eqnarray*}
\end{lemma}

\begin{proof}
In the rest of the proof, we use $c$ and $C$ to denote constants that depend on $ \mathrm{dim}_{M}(S) $, $K_1$, $\kappa$, $p$, and $\lambda$. Note that the specific values of $c$ and $C$ may vary from line to line.

Since $0 < \mathrm{dim}_{M}(S) < \infty$, it follows by Lemma \ref{lm:coveringnumber} that for any $\kappa >0$, there exists $\varepsilon'$ (depending on $\kappa$) such that for all $0 < \varepsilon \leq \varepsilon'$ and $j \in \mathbb{N}^{+}$
\begin{equation*}
\mathcal{M}_{\varepsilon} ( \mathfrak{N}_j ) \leq e^{j} \left( 1 + j^{ \mathrm{dim}_{M}(S) + \kappa  -1 } \varepsilon^{-(\mathrm{dim}_{M}(S) + \kappa )} \right)^{j}.
\end{equation*}
To simplify the notation, we denote $\mathrm{dim}_{M}(S) + \kappa$ by $\beta$. Define $\underline{l} := \lceil - \log \varepsilon' / \log 3 \rceil -1$. It is easily verified that $ 3^{-(k+1)} \leq \varepsilon' $ if $k \geq \underline{l}$. Therefore, for $k \geq \underline{l}$,
\begin{align}
\mathcal{M}_{3^{-(k+1)}}( \mathfrak{N}_j ) \leq e^{j} \left( 1+ j^{\beta -1} 3^{(k+1) \beta} \right)^{j} \leq  e^{j} \left( 2 j^{\beta } 3^{(k+1) \beta} \right)^{j} = (2 e)^{j} j^{\beta j} 3^{ (k+1) \beta j }.  \label{equp:boundforM}
\end{align}

In the next step, we aim to evaluate 
\begin{equation*}
\sum_{k=0}^{k^{*}_{j}} 3^{-kp} \left( 2 \mathcal{L}( \xi ) ( \mathfrak{N}_j ) \right) \wedge \left( \sqrt{ \frac{ \mathcal{M}_{3^{-(k+1)}}( \mathfrak{N}_{j} ) \mathcal{L}( \xi ) ( \mathfrak{N}_j ) }{n} } \right).
\end{equation*}
To this end, we derive an upper bound for the above summation by decomposing it into three separate summations, i.e.,
\begin{align}
\sum_{k=0}^{k^{*}_{j}} (\cdot) &\leq \sum_{k=0}^{ \underline{l} \vee 0 -1 } (\cdot) + \sum_{k=\underline{l} \vee 0}^{ \overline{l}_{j} \vee 0 -1 } (\cdot)
 + \sum_{k=\overline{l}_{j} \vee 0}^{  k_{j}^{*} } (\cdot)  \notag \\
& = : T_1 + T_2 + T_3.  \label{equp:sumsplitmain}
\end{align}
By convention, when $l_1 > l_2$, the summation $\sum_{k=l_1}^{l_2} (\cdot)$ is understood to be $0$. 

We begin by evaluating $T_1$. For an integer $k$ satisfying $0 \leq k < \underline{l} \vee 0$, it is evident that such a $k$ exists only if $\underline{l} >0$ and $0 \leq k < \underline{l}$. Consequently, for $1 \leq j \leq m$, it follows from \eqref{equp:boundforM} that
\begin{equation*}
\mathcal{M}_{3^{-(k+1)}}( \mathfrak{N}_j ) \leq \mathcal{M}_{3^{-(\underline{l}+1)}}( \mathfrak{N}_j ) \leq (2 e)^{j} j^{\beta j} 3^{ ( \underline{l} +1) \beta j } \leq (2e)^{m} m^{\beta m} 3^{ ( \underline{l} +1) \beta m} \leq e^{c m \log m}.
\end{equation*}
Therefore,
\begin{align}
T_1 & \leq \sum_{k=0}^{ \underline{l} \vee 0 -1 } 3^{-kp}  \sqrt{ \frac{ \mathcal{M}_{3^{-(k+1)}}( \mathfrak{N}_{j} ) \mathcal{L}( \xi ) ( \mathfrak{N}_j ) }{n} } 
\leq  \sum_{k=0}^{ \underline{l} \vee 0 -1 } 3^{-kp} n^{-1/2} e^{cm \log m /2 } \notag \\
& \leq C n^{-1/2} e^{cm \log m  }.  \label{equp:sumsplit1}
\end{align}

We then turn to $T_2$. For $k \geq \underline{l} \vee 0$, it follows by \eqref{equp:boundforM} that $ \mathcal{M}_{3^{-(k+1)}}( \mathfrak{N}_j ) \leq (2 e)^{j} j^{\beta j} 3^{ (k+1) \beta j } $. Furthermore, note that for $1 \leq j \leq m$, $ \mathcal{L}(\xi)( \mathfrak{N}_j ) = P( |\xi| =j ) = P( |\eta| =j ) \leq K_1 e^{-\lambda j} $. Therefore,
\begin{align*}
T_2 &\leq \sum_{k=\underline{l} \vee 0}^{ \overline{l}_{j} \vee 0 -1 } 3^{-kp}  \sqrt{ \frac{ \mathcal{M}_{3^{-(k+1)}}( \mathfrak{N}_{j} ) \mathcal{L}( \xi ) ( \mathfrak{N}_j ) }{n} } 
\leq \sum_{k=\underline{l} \vee 0}^{ \overline{l}_{j} \vee 0 -1 } 3^{-kp} n^{-1/2} (2e)^{j/2} j^{\beta j /2} 3^{ (k+1) \beta j/2 } \sqrt{K_1} e^{-\lambda j/2} \\
&= \sqrt{K_1} 2^{j/2} e^{ (1-\lambda) j/2 } (3j)^{ \beta j/2 } \times  n^{-1/2} \sum_{k=\underline{l} \vee 0}^{ \overline{l}_{j} \vee 0 -1 } 3^{( \beta j /2 -p ) k }.
\end{align*}
If $ j \leq \lfloor 2p/\beta \rfloor $, then $ \beta j /2 - p \leq 0 $, and thus,
\begin{align*}
T_2 \leq C n^{-1/2} \sum_{k=\underline{l} \vee 0}^{ \overline{l}_{j} \vee 0 -1 } 1 \leq C n^{-1/2} ( \overline{l}_{j} \vee 0 - \underline{l} \vee 0 )_{+} \leq C n^{-1/2} \log n,  \label{equp:boundT2first}
\end{align*}
where the last inequality follows from the definition of $\overline{l}_{j}  $ in \eqref{equp:loverline}, together with the setting that $j \geq 1$. If $ j > \lfloor 2p/\beta \rfloor$, then $ \beta j /2 - p \geq \left(\lfloor 2p/\beta \rfloor + 1\right) \beta/2 -p >0 $. Therefore,
\begin{align*}
T_2 &\leq \sqrt{K_1} 2^{j/2} e^{ (1-\lambda) j/2 } (3j)^{ \beta j/2 } \times  n^{-1/2} \sum_{k=\underline{l} \vee 0}^{ \overline{l}_{j} \vee 0 -1 } 3^{( \beta j /2 -p ) k }  \\
& \leq \sqrt{K_1} 2^{j/2} e^{ (1-\lambda) j/2 } (3j)^{ \beta j/2 } \times  n^{-1/2} \, \frac{ 3^{ (\beta j /2 -p) (\overline{l}_{j} \vee 0) } }{ 3^{ (\beta j /2 -p) } -1 } \mathds{1}_{ \{ \overline{l}_{j} \vee 0 >  \underline{l} \vee 0\} } \\
& \leq \sqrt{K_1} 2^{j/2} e^{ (1-\lambda) j/2 } (3j)^{ \beta j/2 } \times  n^{-1/2} \, \frac{ 3^{ (\beta j /2 -p) \overline{l}_{j}  } }{ 3^{ \left( \left(\lfloor 2p/\beta \rfloor + 1\right)\beta  /2 -p \right) } -1 } \mathds{1}_{ \{ \overline{l}_{j} \vee 0 >  \underline{l} \vee 0\} },
\end{align*}
where in the last inequality, we invoke the fact that $\overline{l}_{j} > 0$, which follows from the restriction $\overline{l}_{j} \vee 0 > \underline{l} \vee 0$.
Note that
\begin{equation*}
\overline{l}_{j} \leq \frac{1}{\beta} \, \frac{ \log \left( n^{\frac{1}{j}}/ (2 e^{\lambda +1} j^{\beta}) \right) }{ \log 3 } -1.
\end{equation*}
Plugging the above expression into the second last display, it follows after an arrangement that
\begin{equation*}
T_2 \leq C j^{p} e^{-\lambda j} n^{ - \frac{p}{\beta j} } = C e^{p \log j} e^{ - \lambda j - \frac{p}{\beta j} \log n }.
\end{equation*}
By the inequality of arithmetic and geometric means, $\lambda j + \frac{p}{\beta j} \log n \geq 2 \sqrt{ p \lambda \log n /\beta }$. Furthermore, $\log j \leq \log m$, since $j \leq m$. Thus,
\begin{equation*}
T_2 \leq C e^{p \log m} e^{-2 \sqrt{  p \lambda \log n /\beta }}.
\end{equation*}
In summary,
\begin{equation*}
T_2 \leq 
\begin{cases}
C n^{-1/2} \log n, & \text{if $ j \leq \lfloor 2p/\beta \rfloor $} \\
C e^{p \log m} e^{-2 \sqrt{  p \lambda \log n /\beta }}, & \text{if $ j > \lfloor 2p/\beta \rfloor $}
\end{cases},
\end{equation*}
and thus,
\begin{equation}
T_2 \leq C \left( e^{p \log m} e^{-2 \sqrt{  p \lambda \log n /\beta }} + n^{-1/2} \log n \right).  \label{equp:sumsplit2}
\end{equation}

What remains is $T_3$. Note that
\begin{align}
T_3 &\leq \sum_{k=\overline{l}_{j} \vee 0}^{  k_{j}^{*} } 3^{-kp} \times 2 \mathcal{L}( \xi ) ( \mathfrak{N}_j ) 
\leq \sum_{k=\overline{l}_{j} \vee 0}^{  k_{j}^{*} } 3^{-kp} \times 2  K_1 e^{- \lambda j}  \notag  \\
&\leq 2K_1 \times \frac{ 3^{ - ( \overline{l}_{j} \vee 0 ) p } }{ 1 - 3^{-p} } \times e^{- \lambda j} \leq C \, 3^{ - \overline{l}_{j} p } e^{ -\lambda j }.  \label{equp:midpointT3}
\end{align}
By the definition of $ \overline{l}_{j} $, we have
\begin{equation*}
\overline{l}_{j} > \frac{1}{\beta} \, \frac{ \log \left( n^{\frac{1}{j}}/ (2 e^{\lambda +1} j^{\beta}) \right) }{ \log 3 } -2.
\end{equation*}
Therefore,
\begin{align}
3^{ - \overline{l}_{j} p } e^{ -\lambda j } &\leq 3^{2p} n^{ - \frac{p}{\beta j} } \left( 2 e^{\lambda+1} j^{\beta} \right)^{p/\beta } e^{-\lambda j} \leq C j^{p} e^{-\lambda j} n^{ - \frac{p}{\beta j} }  \notag\\
&  \leq C e^{ p \log m }  e^{ - \lambda j - \frac{p}{\beta j} \log n } \leq C e^{ p \log m } e^{-2 \sqrt{  p \lambda \log n /\beta }} ,  \label{equp:acontrol3power}
\end{align}
where the last inequality is a consequence of the inequality $\lambda j + \frac{p}{\beta j} \log n \geq 2 \sqrt{ p \lambda \log n /\beta }$. By combining \eqref{equp:midpointT3} with \eqref{equp:acontrol3power}, we obtain
\begin{equation}
T_3 \leq C e^{ p \log m } e^{-2 \sqrt{  p \lambda \log n /\beta }}.   \label{equp:sumsplit3}
\end{equation}
Then, combining \eqref{equp:sumsplitmain}, \eqref{equp:sumsplit1}, \eqref{equp:sumsplit2}, and \eqref{equp:sumsplit3}, and substituting $\beta$ with $\mathrm{dim}_{M}(S) + \kappa$, it follows that
\begin{eqnarray*}
&& \sum_{k=0}^{k^{*}_{j}} 3^{-kp} \left( 2 \mathcal{L}( \xi ) ( \mathfrak{N}_j ) \right) \wedge \left( \sqrt{ \frac{ \mathcal{M}_{3^{-(k+1)}}( \mathfrak{N}_{j} ) \mathcal{L}( \xi ) ( \mathfrak{N}_j ) }{n} } \right) \\
&\leq& C\left( e^{p \log m} e^{-2 \sqrt{  \frac{ p\lambda}{\mathrm{dim}_{M}(S) + \kappa}  } \sqrt{\log n} } + n^{-1/2} \log n + n^{-1/2} e^{c m \log m} \right).
\end{eqnarray*}

Finally, we estimate $3^{-k_{j}^{*} p}  \mathcal{L}( \xi )( \mathfrak{N}_j )$. Recalling the definition of $k_{j}^{*}$ and applying \eqref{equp:acontrol3power}, we have
\begin{align*}
3^{-k_{j}^{*} p}  \mathcal{L}( \xi )( \mathfrak{N}_j ) \leq 3^{ - \overline{l}_{j} p } \times K_1 e^{ -\lambda j} \leq C e^{ p \log m } e^{-2 \sqrt{  p \lambda \log n /\beta }} = C e^{p \log m} e^{-2 \sqrt{  \frac{ p\lambda}{\mathrm{dim}_{M}(S) + \kappa}  } \sqrt{\log n} }.
\end{align*}
Combining the last two displays, we conclude the proof.
\end{proof}

With the above preparation, we are now ready to establish an upper bound on the expected Wasserstein distance between $\widehat{\mathcal{L}}_{n}(\eta)$ and $\mathcal{L}(\eta)$, which constitutes the main result of this section.

\begin{theorem} \label{thm:upperbound}
Suppose Assumptions (A1) and (A2) hold. For any $\kappa >0$, there exists a constant $C$ depending on $C_D$, $ \mathrm{dim}_{M}(S) $, $K_1$, $\kappa$, $p$, and $\lambda$ such that, for sufficiently large $n$,
\begin{equation*}
E\left[ W_{p} \left( \widehat{\mathcal{L}}_{n}(\eta), \mathcal{L}(\eta) \right) \right] \leq C e^{ (1 + \frac{1}{2p}) \log \log n } e^{-2 \sqrt{  \frac{ \lambda}{ p(\mathrm{dim}_{M}(S) + \kappa)}  } \sqrt{\log n} }.
\end{equation*}
Consequently, for any $\kappa >0$
\begin{equation*}
E\left[ W_{p} \left( \widehat{\mathcal{L}}_{n}(\eta), \mathcal{L}(\eta) \right) \right] \leq C  e^{-2 \sqrt{  \frac{ \lambda}{ p(\mathrm{dim}_{M}(S) + 2 \kappa)}  } \sqrt{\log n} },
\end{equation*}
for sufficiently large $n$.
\end{theorem}

\begin{proof}
In the rest of the proof, we use $c$ and $C$ to denote constants that depend on $C_D$, $ \mathrm{dim}_{M}(S) $, $K_1$, $\kappa$, $p$, and $\lambda$. Note that the specific values of $c$ and $C$ may vary from line to line. Furthermore, since the second assertion follows directly from the first, it suffices to prove the first.

We begin the proof by setting 
\begin{equation*}
m = m(n) := \left \lfloor 2\sqrt{ \frac{p}{\lambda ( \mathrm{dim}_{M}(S) + \kappa )}  } \sqrt{\log n } \right \rfloor.
\end{equation*}
Therefore,
\begin{equation}
2\sqrt{ \frac{p}{\lambda ( \mathrm{dim}_{M}(S) + \kappa )}  } \sqrt{\log n } -1 < m \leq 2\sqrt{ \frac{p}{\lambda ( \mathrm{dim}_{M}(S) + \kappa )}  } \sqrt{\log n }.  \label{equp:mbound}
\end{equation}
For the specific value of $m$, applying Lemma \ref{lm:truncation} yields
\begin{equation}
E\left[ W_{p} \left( \widehat{\mathcal{L}}_{n}(\eta), \mathcal{L}(\eta) \right) \right]
\leq E \left[ W_{p} \left( \widehat{\mathcal{L}}_{n}(\xi),  \mathcal{L} (\xi) \right) \right] + 2 C_D \left( E\left[ |\eta|^{p} \mathds{1}_{  \{ |\eta| >m \} } \right] \right)^{\frac{1}{p}}.   \label{equp:Wupboundmain}
\end{equation}
Combining Lemma \ref{lm:exactboundWassersteinxi} with Lemma \ref{lm:onetermevaluation}, we have
\begin{eqnarray*}
&& E \left[ W_{p}^{p} \left( \widehat{\mathcal{L}}_{n}(\xi),  \mathcal{L} (\xi) \right) \right]  \\
&\leq& C \sum_{j=1}^{m} j^{p} \left( e^{p \log m} e^{-2 \sqrt{  \frac{ p\lambda}{\mathrm{dim}_{M}(S) + \kappa}  } \sqrt{\log n} } + n^{-1/2} \log n + n^{-1/2} e^{c m \log m} \right)  \\
&\leq& C m^{p+1} \left( e^{p \log m} e^{-2 \sqrt{  \frac{ p\lambda}{\mathrm{dim}_{M}(S) + \kappa}  } \sqrt{\log n} } + n^{-1/2} \log n + n^{-1/2} e^{c m \log m} \right) \\
&=& C \left( e^{(2p+1) \log m} e^{-2 \sqrt{  \frac{ p\lambda}{\mathrm{dim}_{M}(S) + \kappa}  } \sqrt{\log n} } +  e^{(p+1) \log m} n^{-1/2} \log n + n^{-1/2} e^{c m \log m + (p+1) \log m} \right).
\end{eqnarray*}
Together with the upper bound of $m$ in \eqref{equp:mbound}, it follows after a rearrangement that
\begin{equation*}
E \left[ W_{p}^{p} \left( \widehat{\mathcal{L}}_{n}(\xi),  \mathcal{L} (\xi) \right) \right]
\leq C e^{ (p + \frac{1}{2}) \log \log n } e^{-2 \sqrt{  \frac{ p\lambda}{\mathrm{dim}_{M}(S) + \kappa}  } \sqrt{\log n} },
\end{equation*}
for sufficiently large $n$. Applying Jensen's inequality gives
\begin{equation}
E \left[ W_{p} \left( \widehat{\mathcal{L}}_{n}(\xi),  \mathcal{L} (\xi) \right) \right]
\leq C e^{ (1 + \frac{1}{2p}) \log \log n } e^{-2 \sqrt{  \frac{ \lambda}{ p(\mathrm{dim}_{M}(S) + \kappa)}  } \sqrt{\log n} }. \label{equp:Wupbound1}
\end{equation}
Furthermore, noting that $P(|\eta| = j) \leq K_1 e^{-\lambda j}$, we have for sufficiently large $n$,
\begin{align*}
E\left[ |\eta|^{p} \mathds{1}_{  \{ |\eta| >m \} } \right] &= \sum_{j= m+1}^{\infty} j^{p} P ( |\eta| = j ) \leq \sum_{j= m+1}^{\infty} j^{p} \times K_1 e^{ -\lambda j }  \\
&= K_1 m^{p} e^{-\lambda m} \sum_{j = m+1}^{\infty} \left( \frac{j}{m} \right)^{p} e^{-\lambda(j-m)} \\
&= K_1 m^{p} e^{-\lambda m} \sum_{k = 1}^{\infty} \left( 1 + \frac{k}{m} \right)^{p} e^{ -\lambda k } \\
&\leq C m^{p} e^{-\lambda m},
\end{align*}
where in the last equality, we substitute $j$ with $m + k$, and in the final inequality, we use the fact that the series $\{(1 + k/m)^{p} e^{-\lambda k}\}_{k=1}^{\infty}$ converges for all $m \geq 1$ with a uniform upper bound. Combining the last display with \eqref{equp:mbound}, we obtain
\begin{align*}
E\left[ |\eta|^{p} \mathds{1}_{  \{ |\eta| >m \} } \right]
&\leq C \left( 2\sqrt{ \frac{p}{\lambda ( \mathrm{dim}_{M}(S) + \kappa )}  } \sqrt{\log n } \right)^{p} e^{ - \lambda \left(2 \sqrt{ \frac{p}{\lambda ( \mathrm{dim}_{M}(S) + \kappa )}  } \sqrt{\log n } -1\right) }  \\
& \leq C e^{ \frac{p}{2} \log \log n } e^{-2 \sqrt{  \frac{ p\lambda}{\mathrm{dim}_{M}(S) + \kappa}  } \sqrt{\log n} }.
\end{align*}
Therefore,
\begin{equation}
\left( E\left[ |\eta|^{p} \mathds{1}_{  \{ |\eta| >m \} } \right] \right)^{\frac{1}{p}}
\leq C e^{ \frac{1}{2} \log \log n } e^{-2 \sqrt{  \frac{ \lambda}{p(\mathrm{dim}_{M}(S) + \kappa})  } \sqrt{\log n} }  \label{equp:Wupbound2}
\end{equation}
Combining \eqref{equp:Wupboundmain}, \eqref{equp:Wupbound1}, and \eqref{equp:Wupbound2}, the desired result follows.
\end{proof}

\begin{remark}
Let $S = [0, b]$ with the Euclidean metric for some $b > 0$, and suppose that Assumption (A2) holds. It is straightforward to verify that $ \mathcal{M}_{\varepsilon/m}( S ) \leq 2bm/\varepsilon $ for $\varepsilon \leq b$. Plugging this into \eqref{eqc:coveringnumberbound}, we obtain a refined upper bound for the covering number of $\mathfrak{N}_m$, i.e.,
\begin{equation*}
\mathcal{M}_{\varepsilon}( \mathfrak{N}_m ) \leq e^{m} \left( 1 + \frac{2b}{\varepsilon} \right)^{m}.
\end{equation*}
By replicating the arguments used in the proofs of Lemma \ref{lm:onetermevaluation} and Theorem \ref{thm:upperbound}, we can improve the upper bound in Theorem \ref{thm:upperbound} and obtain
\begin{equation*}
E\left[ W_{p} \left( \widehat{\mathcal{L}}_{n}(\eta), \mathcal{L}(\eta) \right) \right] \leq C e^{ ( \frac{1}{2} + \frac{1}{2p}) \log \log n } e^{ -2 \sqrt{ \frac{\lambda}{p} } \sqrt{\log n} },
\end{equation*}
where $C$ is a constant depending on $C_D$, $K_1$, $p$, and $\lambda$.
\end{remark}

\subsection{Examples: Poisson point processes and Hawkes process}  \label{subsec:exampleupperbound}
Now, we consider $\eta$ to be either a Poisson point process, or a linear Hawkes process, and present results on the rates of convergence of the corresponding empirical measures.

For a Poisson process, we have $P(|\eta| = m)= \frac{e^{-\lambda} \lambda^{m}}{m!}$, where $\lambda = E[\eta(S)]$. Applying Stirling’s approximation yields $P(|\eta| = m) \leq C \frac{e^{-\lambda}}{\sqrt{2 \pi m}}(\frac{\lambda e}{m})^{m}$ for some constant $C$. Note that $\left( (\lambda e) / m \right)^{m}$ decays much faster than $e^{-m}$ as in Assumption (A2). Thus, a sharper upper bound for the expected Wasserstein distance shall be obtained by directly applying the above inequality in the proofs of Lemma \ref{lm:onetermevaluation} and Theorem \ref{thm:upperbound}, rather than first establishing a universal constant $K_1$ and subsequently invoking Theorem \ref{thm:upperbound}.

To derive the upper bound, we use similar techniques as used in Lemma \ref{lm:onetermevaluation} and Theorem \ref{thm:upperbound}. We omit the intermediate rearrangements and calculations, presenting only the essential ingredients of the argument. Unless stated otherwise, all notation adheres to the definitions introduced previously. Redefine 
\begin{equation}
\overline{l}_{j} := \left\lfloor  \frac{1}{\mathrm{dim}_{M}(S) + \kappa}  \frac{ \log \left(  \lambda  n^{\frac{1}{j}}  /\left( 2j^{ \mathrm{dim}_{M}(S) + \kappa +1 +\frac{1}{2j} } \right)\right)  }{ \log 3} \right\rfloor -1.
\label{equp:loverline_poisson}
\end{equation}
Subsequently, we have
\begin{equation}
T_1 \leq C n^{-1/2} e^{cm \log m  }, \label{equp:sumsplit1_poisson}
\end{equation}
\begin{align}
T_2 &\leq C \left( \sup_{j \in \mathbb{N}^{+} } e^{ \left( p + \frac{p}{\beta} -\frac{1}{2} \right) \log j + \left( \log \lambda +1 \right) j  - j \log j - \frac{p}{\beta j} \log n } + n^{-1/2} \log n \right) \nonumber \\
&\leq C \left(  e^{ - (1- \chi/2) \sqrt{ \frac{2p}{\beta}  } \sqrt{\log n  \log \log n}  }  +  n^{-1/2} \log n  \right)
\label{equp:sumsplit2_poisson}
\end{align}
and
\begin{align}
T_3 &\leq C \sup_{j \in \mathbb{N}^{+} } e^{ \left( p + \frac{p}{\beta} -\frac{1}{2} \right) \log j + \left( \log \lambda +1 \right) j  - j \log j - \frac{p}{\beta j} \log n }  \nonumber  \\
&\leq C   e^{ - (1- \chi/2) \sqrt{ \frac{2p}{\beta}  } \sqrt{\log n \log \log n}  }   ,
\label{equp:sumsplit3_poisson}
\end{align}
for sufficiently large $n$, for any $0 < \chi <1$. Throughout the remainder of this subsection, the generic constants $C$, including those appearing in \eqref{equp:sumsplit2_poisson} and \eqref{equp:sumsplit3_poisson}, may depend on $C_D$, $\dim_M(S)$, $p$, $\lambda$, and $\chi$, and may vary from line to line.
Then, combining \eqref{equp:sumsplit1_poisson}, \eqref{equp:sumsplit2_poisson}, and \eqref{equp:sumsplit3_poisson}, and substituting $\beta$ with $\mathrm{dim}_{M}(S) + \kappa$, it follows that
\begin{eqnarray*}
&& \sum_{k=0}^{k^{*}_{j}} 3^{-kp} \left( 2 \mathcal{L}( \xi ) ( \mathfrak{N}_j ) \right) \wedge \left( \sqrt{ \frac{ \mathcal{M}_{3^{-(k+1)}}( \mathfrak{N}_{j} ) \mathcal{L}( \xi ) ( \mathfrak{N}_j ) }{n} } \right) \\
&\leq& C\left( e^{ - (1- \chi/2) \sqrt{ \frac{2p}{\mathrm{dim}_{M}(S) + \kappa} } \sqrt{ \log n \log \log n}  } + n^{-1/2} \log n + n^{-1/2} e^{c m \log m} \right).
\end{eqnarray*}
Similarly,
\begin{equation*}
3^{-k_{j}^{*} p}  \mathcal{L}( \xi )( \mathfrak{N}_j ) \leq  C  e^{ - (1- \chi/2) \sqrt{ \frac{2p}{\mathrm{dim}_{M}(S) + \kappa} } \sqrt{ \log n \log \log n}  }.
\end{equation*}
Combining the above two displays, we have 
\begin{eqnarray*}
&& 3^{-k_{j}^{*} p}  \mathcal{L}( \xi )( \mathfrak{N}_j ) + \sum_{k=0}^{k^{*}_{j}} 3^{-kp} \left( 2 \mathcal{L}( \xi ) ( \mathfrak{N}_j ) \right) \wedge \left( \sqrt{ \frac{ \mathcal{M}_{3^{-(k+1)}}( \mathfrak{N}_{j} ) \mathcal{L}( \xi ) ( \mathfrak{N}_j ) }{n} } \right) \\
&\leq& C\left( e^{ - (1- \chi/2) \sqrt{ \frac{2p}{\mathrm{dim}_{M}(S) + \kappa} } \sqrt{ \log n \log \log n}  } + n^{-1/2} \log n + n^{-1/2} e^{c m \log m} \right).
\end{eqnarray*}
Redefine
\begin{equation*}
m = m(n) := \left \lceil  \sqrt{\frac{2p}{ \mathrm{dim}_{M}(S) + \kappa } } \sqrt{\log n}  \right \rceil.
\end{equation*}
Then, applying Lemma~\ref{lm:exactboundWassersteinxi} and plugging in the expression of $m$, it follows after a rearrangement that
\begin{align*}
E \left[ W_{p}^{p} \left( \widehat{\mathcal{L}}_{n}(\xi),  \mathcal{L} (\xi) \right) \right]
&\leq C \sum_{j=1}^{m} j^p \left( e^{ - (1- \chi/2) \sqrt{ \frac{2p}{\mathrm{dim}_{M}(S) + \kappa} } \sqrt{ \log n \log \log n}  } + n^{-1/2} \log n + n^{-1/2} e^{c m \log m} \right) \\
&\leq C m^{p+1} \left( e^{ - (1- \chi/2) \sqrt{ \frac{2p}{\mathrm{dim}_{M}(S) + \kappa} } \sqrt{ \log n \log \log n}  } + n^{-1/2} \log n + n^{-1/2} e^{c m \log m} \right) \\
& \leq C e^{ - (1- \chi) \sqrt{ \frac{2p}{\mathrm{dim}_{M}(S) + \kappa} } \sqrt{ \log n \log \log n}  },
\end{align*}
and hence,
\begin{equation}
E \left[ W_{p} \left( \widehat{\mathcal{L}}_{n}(\xi),  \mathcal{L} (\xi) \right) \right]
\leq C  e^{ - (1- \chi) \sqrt{ \frac{2}{ p (\mathrm{dim}_{M}(S) + \kappa)} } \sqrt{ \log n \log \log n}  }, \label{equp:Wupbound1_poisson}
\end{equation}
for sufficiently large $n$. Furthermore, note that
\begin{align}
E\left[ |\eta|^{p} \mathds{1}_{  \{ |\eta| >m \} } \right] \le C e^{ \left(p  - \frac{1}{2} \right) \log m + (1 + \log \lambda) m - m \log m  } .  \label{eqnew:tailboundtorefree}
\end{align}
Substituting the expression of $m$ into \eqref{eqnew:tailboundtorefree}, the leading term in the exponent on the right-hand side is
\begin{equation*}
-\sqrt{\frac{p}{2(\mathrm{dim}_{M}(S) + \kappa)}} \sqrt{\log n }  \log \log n
\end{equation*}
Noting that $\log \log n \geq 2 \sqrt{ \log \log n }$ for sufficiently large $n$, the leading term is no more than
\begin{equation*}
 -\sqrt{ \frac{2p}{\mathrm{dim}_{M}(S) + \kappa} } \sqrt{ \log n  \log \log n},
\end{equation*}
and hence,
\begin{equation*}
E\left[ |\eta|^{p} \mathds{1}_{  \{ |\eta| >m \} } \right] \leq C  e^{ - (1- \chi) \sqrt{ \frac{2p}{\mathrm{dim}_{M}(S) + \kappa} } \sqrt{ \log n  \log \log n} },
\end{equation*}
for sufficiently large $n$.
Together with \eqref{equp:Wupbound1_poisson}, we conclude that  
\begin{equation*}
E\left[ W_{p} \left( \widehat{\mathcal{L}}_{n}(\eta), \mathcal{L}(\eta) \right) \right] \leq C  e^{ - (1- \chi) \sqrt{ \frac{2}{ p (\mathrm{dim}_{M}(S) + \kappa)} } \sqrt{ \log n \log \log n}  },
\end{equation*}
for sufficiently large $n$. Note that the presence of the term $\sqrt{\log n \log \log n}$ in the exponent provides a slightly sharper bound compared to the term $\sqrt{\log n}$ in Theorem \ref{thm:upperbound}.

Before introducing the second example on the Hawkes process, it is helpful to recall the Galton–Watson process. The Galton–Watson process is a classical discrete-time branching model that describes the evolution of a population across successive generations \cite{athreya2004branching}. Although it is not itself a point process, the number of points generated in certain point processes may exhibit a Galton–Watson–type branching structure. For a Galton–Watson process, the distribution of the total progeny (the total number of points generated from one ancestor including itself) and the rate of convergence to extinction are determined by its offspring distribution. Suppose that $X$ denotes the total progeny of a Galton–Watson process with a single ancestor at time $0$ and an offspring distribution given by $\mathrm{Poisson}(\mu)$ with $\mu < 1$. Then, by Dwass’s theorem \cite{jagers1975branching}, the distribution of $X$ is the Borel distribution:
\begin{equation*}
P(X = m) = \frac{e^{-\mu m}(\mu m)^{m-1}}{m!}.
\end{equation*}
Furthermore, the moment generating function is $M_{X}(z) = \frac{1}{\mu} g^{-1}(\mu e^{-(\mu - z)})$ for $z \le \mu - 1 - \log \mu$, where $g(x) = xe^{-x}$ for $x \in [0,1]$. Thus, the moment generating function is bounded above by $1/\mu$. 

We proceed to investigate a specific point process---the linear Hawkes process. This process has been widely employed to model self-exciting phenomena, such as earthquake aftershocks, financial contagion, and neuronal spike trains, where the occurrence of one event increases the likelihood of subsequent events in its vicinity. Formally, a Hawkes process $N_t$ with baseline intensity $\nu > 0$ and kernel $h:\mathbb{R}^+ \to \mathbb{R}^+$ is the simple point process with stochastic intensity
\begin{equation*}
  \lambda_t \;=\; \nu \;+\; \int_{0}^{t-} h(t-s)\, \mathrm{d}N_s.
\end{equation*} 
Let $\mu := \int_0^\infty h(u)\,du < 1$. In particular, the linear Hawkes process on the real line can be viewed as an immigration–birth system \cite{Hawkes1974,hawkesreview}. Immigrants arrive according to a homogeneous Poisson process with constant rate $\nu$. Each immigrant produces offspring according to a Poisson distribution with parameter $\mu$. Each child reproduces according to the same rule, independently of others, and immigrant families evolve independently. All offspring of a given immigrant, together with the immigrant itself, form a single cluster. Let $S(T)$ denote the number of descendants of an immigrant arriving at time $0$ on the interval $[0,T]$, including the immigrant itself. It is known that $S(\infty)$, the total cluster size, coincides with the total progeny of a Galton–Watson process with a single ancestor and Poisson$(\mu)$ offspring distribution. Let $N(T)$ denote the number of points of the Hawkes process on $[0,T]$. Then, it follows from \cite{Gao2021,ZHU2013} that
\begin{equation*}
E[e^{z N(T)}] = e^{\nu \int_{0}^{T}(E[e^{z S(t)}] - 1)dt}, \quad \text{for $z \le \mu - 1 - \log \mu$},
\end{equation*}
which is bounded above since $S(t)$ is non-decreasing almost surely and $E[e^{zS(\infty)}]$ is bounded above by $1/\mu$. Thus, by Chernoff bound, the Hawkes process on $[0,T]$ exhibits exponential tails for any $z \le \mu - 1 - \log \mu$. Consequently, Assumption (A2) is satisfied. Furthermore, since we are considering point processes on the bounded interval $[0,T]$, Assumption (A1) holds automatically. Therefore, Theorem \ref{thm:upperbound} can be applied directly to establish the upper bound for the rate of convergence of the empirical measure.

\section{Lower bounds for the Wasserstein distance}  \label{sec:lowerbound}
In the previous section, we develop the upper bounds on the rates of convergence of empirical measures for a specific class of proper point processes. Next, we will establish lower bounds on the rates of convergence. A key tool in characterizing this class is the Janossy measure, which will be introduced in the following subsection. In the following subsection, for detailed discussion, see \cite{Daley2003introduction} and \cite{Last_Penrose_2017}.

\subsection{Janossy measures}
Suppose $\mu \in \mathfrak{N}$ is a finite counting measure on $S$ given by
\begin{equation*}
  \mu = \sum_{i=1}^{k} \delta_{x_i},
\end{equation*}
for some $k \in \{0\} \cup \mathbb{N}^{+}$ and (not necessarily distinct) points $x_1, x_2, \dots, x_k \in S$. Then we define a counting measure $ \mu^{(m)}$ on $S^{m}$ by
\begin{equation*}
  \mu^{(m)} (C) = \sideset{}{_{}^{\neq}} \sum_{ 1 \leq i_1 , \dots, i_m \leq k } \mathds{1}_{ \{ (x_{i_1}, \dots, x_{i_m}  ) \in C \} },
\end{equation*}
for any Borel set $C \subseteq S^{m}$, where the superscript $\neq$ indicates summation over $m$-tuples with pairwise difference entries (i.e., $i_1, \dots, i_m$ are pairwise distinct) and where an empty sum is defined as zero. In other words, this means that
\begin{equation*}
  \mu^{(m)} = \sideset{}{_{}^{\neq}} \sum_{ 1 \leq i_1 , \dots, i_m \leq k } \delta_{ ( x_{i_1}, \dots, x_{i_m} ) }.
\end{equation*}

Let $\eta$ be a finite proper point process on $S$. Then any realization $\eta(\omega)$ of $\eta$ is a finite counting measure on $S$. Therefore, if $m \geq 1$, the Janossy measure of order $m$ of $\eta$ is the measure on $S^{m}$ defined by 
\begin{equation*}
  J_{\eta, m} (\cdot) := \frac{1}{m!} E\left[ \mathds{1}_{ \{ |\eta| =m \} } \eta^{(m)} (\cdot) \right].
\end{equation*}
If $m = 0$, the number $J_{\eta, 0} := P\left( |\eta| =0 \right)$ is called the Janossy measure of order $0$.

It is straightforward to verify that the Janossy measures $J_{\eta, m}$ are symmetric and for $m \geq 1$
\begin{equation*}
  J_{\eta, m}( S^{m} ) = P \left( |\eta| =m \right).
\end{equation*}
Therefore, if $P( |\eta| =m ) >0$, we define the normalized Janossy measure of order $m$ to be 
\begin{equation*}
  \widetilde{J}_{\eta, m} (\cdot) := \frac{1}{  J_{\eta, m}( S^{m} ) }  J_{\eta, m}( \cdot ).
\end{equation*}
Thus, $ \widetilde{J}_{\eta, m} $ is a probability measure on $S^{m}$.

With the above preparation, we now proceed to formulate the assumption defining the specific class of point processes.

\begin{enumerate}[(A3)]
\item Suppose $\eta$ is a finite proper point process. For any $m \geq 1$ such that $P(|\eta| =m) >0$, the normalized Janossy measure of order $m$ satisfies
\begin{equation*}
  \widetilde{J}_{\eta, m} (C) \leq \theta^{m}(C)
\end{equation*}
for any Borel set $C \subseteq S^{m}$, where $\theta$ is a positive finite measure on $S$. Moreover, $\theta$ satisfies a local upper bound condition: there exists $\sigma > 0$ such that
\begin{equation*}
  \theta \left( \overline{B_{\rho}\left(x, \varepsilon\right)} \right) \leq K_2 \varepsilon^{\sigma},
\end{equation*}
for all $x \in S$ and $\varepsilon >0 $.
\end{enumerate}

\subsection{A lower bound}
In this subsection, we establish a lower bound for $W_p(\widehat{\mathcal{L}}_{n}(\eta), \mathcal{L}(\eta))$ that holds both in expectation and almost surely. Importantly, the derivation does not rely on structural assumptions regarding the underlying metric space or its associated metric---such as compactness or boundedness. As a result, the analysis avoids the use of truncation and telescoping arguments required for the upper bound, leading to a simpler proof. The key ingredient is the following proposition, originally adapted from \cite{dudley1969speed} and cited in \cite{weed2019sharp}, which applies not only to empirical measures but to any probability measure supported on at most $n$ elements.

\begin{proposition} \label{prop:sixthWeed}
Suppose that there exist positive constants $\varepsilon'$, $\tau$, and $t$ such that 
\begin{equation*}
  \mathcal{M}_{\varepsilon}( \mu, \tau) \geq \varepsilon^{-t}
\end{equation*}
for all $\varepsilon \leq \varepsilon'$. If $n > \varepsilon'^{-t}$ and $\nu$ is any measure supported on at most $n$ points, then
\begin{equation*}
  W_{p}^{p}(\mu, \nu) \geq \tau 4^{-p} n^{-p/t}.
\end{equation*}
\end{proposition}

To apply Proposition \ref{prop:sixthWeed} in our context, it is necessary to evaluate $\mathcal{M}_{\varepsilon}( \mathcal{L}(\eta), \tau )$. This is the primary objective of the following lemma.

\begin{lemma}  \label{lm:coveringlowerbound}
Let $\eta$ be a point process satisfying Assumption (A3), and define $\tau_m = P(|\eta| = m)/2$. Then, for any $m \in \mathbb{N}^{+}$ such that $\tau_m > 0$ and for any $ 0 <\kappa < \sigma $, we have
\begin{equation*}
\mathcal{M}_{\varepsilon}\left( \mathcal{L}(\eta), \tau_m  \right) \geq \varepsilon^{ - m(\sigma -\kappa) }
\end{equation*}
for all $ \varepsilon \leq \frac{  2^{\sigma/\kappa} }{K_2^{1/\kappa} (2 m!)^{1/(m \kappa)}} \wedge c_D $.
\end{lemma}

\begin{proof}
Recalling the definition of $\mathcal{M}_{\varepsilon} (\mu, \tau) $, we have
\begin{equation*}
  \mathcal{M}_{\varepsilon}\left( \mathcal{L}(\eta), \tau_m  \right) = \inf\{ \mathcal{M}_{\varepsilon}(\mathfrak{B}) : \mathcal{L}(\eta) (  \mathfrak{B}) \geq 1- \tau_{m} \}.
\end{equation*}
Therefore, it suffices to show that for any $\mathfrak{B} \subseteq \mathfrak{N}$ such that $\mathcal{L}(\eta) (  \mathfrak{B}) \geq 1- \tau_{m}$,
\begin{equation*}
\mathcal{M}_{\varepsilon}( \mathfrak{B} ) \geq \varepsilon^{-m(\sigma - \kappa)}
\end{equation*}
holds for all $\varepsilon \leq \frac{  2^{\sigma/\kappa} }{K_2^{1/\kappa} (2 m!)^{1/(m \kappa)}} \wedge c_D$.

To establish the above result, we first claim that
\begin{equation}
  \mathcal{L}(\eta)( \mathfrak{B} \cap \mathfrak{N}_{m} ) \geq \frac{1}{2} \mathcal{L}(\eta) ( \mathfrak{N}_m) = \frac{1}{2} P( |\eta| = m ).  \label{eqlo:aboundfortaum}
\end{equation}
Otherwise, $ \mathcal{L}(\eta)( \mathfrak{B} \cap \mathfrak{N}_{m} ) < \frac{1}{2} \mathcal{L}(\eta) ( \mathfrak{N}_m) $. Thus,
\begin{align*}
\mathcal{L}(\eta)( \mathfrak{B} ) &= \sum_{j=0}^{\infty} \mathcal{L}(\eta) ( \mathfrak{B} \cap \mathfrak{N}_{j} ) =  \sum_{j \neq m} \mathcal{L}(\eta) ( \mathfrak{B} \cap \mathfrak{N}_{j} ) + \mathcal{L}(\eta) ( \mathfrak{B} \cap \mathfrak{N}_{m} )  \\
&<  \sum_{j \neq m} \mathcal{L}(\eta) ( \mathfrak{N}_{j} ) + \frac{1}{2} \mathcal{L}(\eta) ( \mathfrak{N}_m) = 1 - \frac{1}{2} \mathcal{L}(\eta) ( \mathfrak{N}_m) \\
& = 1- \tau_{m}.
\end{align*}
This is a contradiction.

In the next step, we consider $ \mathcal{M}_{\varepsilon}( \mathfrak{B} \cap \mathfrak{N}_{m} ) $. For any $0 <\varepsilon \leq c_D$, suppose $\mathcal{M}_{\varepsilon} ( \mathfrak{B} \cap \mathfrak{N}_{m} ) =k$. Then, there exist $k$ closed balls of diameter $\varepsilon$, denoted by $\overline{B_{D}( \mu_1, \varepsilon/2 )}, \dots, \overline{B_{D}( \mu_k, \varepsilon/2 )} $ with centers $\mu_1, \dots, \mu_k \in \mathfrak{N}$, that cover $\mathfrak{B} \cap \mathfrak{N}_{m}$. We claim that all of $\mu_1, \dots, \mu_k$ must belong to $\mathfrak{N}_m$. Otherwise, suppose $\mu_j \notin \mathfrak{N}_m$ for some $j$. However, by Condition~\ref{assume:metric}(iii), we have
\begin{equation*}
  D (\mu_j, \nu)  \geq c_D > \frac{\varepsilon}{2}
\end{equation*}
for any $\nu \in \mathfrak{N}_m$. As a result, $\overline{B_{D}( \mu_j, \varepsilon/2 )} \cap \mathfrak{N}_m = \varnothing$, and thus, only $(k-1)$ closed balls of diameter $\varepsilon$ are enough to cover $\mathfrak{B} \cap \mathfrak{N}_m$. This contradicts the assumption $ \mathcal{M}_{\varepsilon} ( \mathfrak{B} \cap \mathfrak{N}_{m} ) =k $.

Since $\mu_1, \dots, \mu_k \in \mathfrak{N}_m$, for $j=1, \dots,  k$, we write
\begin{equation*}
  \mu_j = \sum_{i=1}^{m} \delta_{x_{i}^{j}},
\end{equation*}
where $ x_{1}^{j}, \dots, x_{m}^{j} \in S $. Again, noting that $D( \mu_j, \nu ) \geq c_D > \varepsilon/2$ for $\nu \notin \mathfrak{N}_m$, we obtain
\begin{equation*}
\overline{B_{D}( \mu_j, \varepsilon/2 )} = \left\{  \sum_{i=1}^{m} \delta_{y_i} : \min_{\pi \in \Pi_m } \sum_{i=1}^{m} \rho \left( x_{i}^{j}, y_{\pi(i)} \right) \leq  \frac{\varepsilon}{2} \right\}.
\end{equation*}
Therefore,
\begin{align*}
 \mathcal{L}(\eta) \left( \overline{B_{D}( \mu_j, \varepsilon/2 )} \right)
&= P\left( \eta \in \left\{  \sum_{i=1}^{m} \delta_{y_i} : \min_{\pi \in \Pi_m } \sum_{i=1}^{m} \rho \left( x_{i}^{j}, y_{\pi(i)} \right) \leq  \frac{\varepsilon}{2} \right\} \right)  \\
&= P\left( \eta \in \left\{  \sum_{i=1}^{m} \delta_{y_i} : \min_{\pi \in \Pi_m } \sum_{i=1}^{m} \rho \left( x_{i}^{j}, y_{\pi(i)} \right) \leq  \frac{\varepsilon}{2} \right\}, \, |\eta| =m \right).
\end{align*}
From the definition of $\eta^{(m)}$, we deduce that, given $|\eta| = m$,
\begin{equation*}
\eta^{(m)}\left( \left\{ (y_1, \dots, y_m): \min_{\pi \in \Pi_m } \sum_{i=1}^{m} \rho \left( x_{i}^{j}, y_{\pi(i)} \right) \leq  \frac{\varepsilon}{2} \right\} \right) = m!,
\end{equation*}
if $  \eta \in \left\{  \sum_{i=1}^{m} \delta_{y_i} : \min_{\pi \in \Pi_m } \sum_{i=1}^{m} \rho ( x_{i}^{j}, y_{\pi(i)} ) \leq  \varepsilon/2 \right\} $, and
\begin{equation*}
\eta^{(m)}\left( \left\{ (y_1, \dots, y_m): \min_{\pi \in \Pi_m } \sum_{i=1}^{m} \rho \left( x_{i}^{j}, y_{\pi(i)} \right) \leq  \frac{\varepsilon}{2} \right\} \right) = 0,
\end{equation*}
if $  \eta \notin \left\{  \sum_{i=1}^{m} \delta_{y_i} : \min_{\pi \in \Pi_m } \sum_{i=1}^{m} \rho ( x_{i}^{j}, y_{\pi(i)} ) \leq  \varepsilon/2 \right\} $. Therefore,
\begin{align*}
 \mathcal{L}(\eta) \left( \overline{B_{D}( \mu_j, \varepsilon/2 )} \right)
 &= P\left( \eta \in \left\{  \sum_{i=1}^{m} \delta_{y_i} : \min_{\pi \in \Pi_m } \sum_{i=1}^{m} \rho \left( x_{i}^{j}, y_{\pi(i)} \right) \leq  \frac{\varepsilon}{2} \right\}, \, |\eta| =m \right) \\
 &= \frac{1}{m!} E\left[ \mathds{1}_{\{ |\eta| =m \} } \, \eta^{(m)}\left( \left\{ (y_1, \dots, y_m): \min_{\pi \in \Pi_m } \sum_{i=1}^{m} \rho \left( x_{i}^{j}, y_{\pi(i)} \right) \leq  \frac{\varepsilon}{2} \right\} \right) \right] \\
 &= J_{\eta, m} \left( \left\{ (y_1, \dots, y_m): \min_{\pi \in \Pi_m } \sum_{i=1}^{m} \rho \left( x_{i}^{j}, y_{\pi(i)} \right) \leq  \frac{\varepsilon}{2} \right\} \right) \\
 &=  P(|\eta| =m) \widetilde{J}_{\eta,m} \left( \left\{ (y_1, \dots, y_m): \min_{\pi \in \Pi_m } \sum_{i=1}^{m} \rho \left( x_{i}^{j}, y_{\pi(i)} \right) \leq  \frac{\varepsilon}{2} \right\} \right),
\end{align*}
where the last equality follows from the definition that $\widetilde{J}_{\eta, m} (\cdot) := J_{\eta, m}( \cdot )/  J_{\eta, m}( S^{m} )$ and the fact $ J_{\eta, m}( S^{m} ) = P( |\eta| =m )$. Note that
\begin{eqnarray*}
&&\left\{ (y_1, \dots, y_m): \min_{\pi \in \Pi_m } \sum_{i=1}^{m} \rho \left( x_{i}^{j}, y_{\pi(i)} \right) \leq  \frac{\varepsilon}{2} \right\} \\
&=& \left\{ (y_1, \dots, y_m): \min_{\pi \in \Pi_m } \sum_{i=1}^{m} \rho \left( x_{\pi(i)}^{j}, y_{i} \right) \leq  \frac{\varepsilon}{2} \right\} \\
&\subset& \bigcup_{ \pi \in \Pi_m } \left\{ (y_1, \dots, y_m):  \sum_{i=1}^{m} \rho \left( x_{\pi(i)}^{j}, y_{i} \right) \leq  \frac{\varepsilon}{2} \right\}  \\
&\subset& \bigcup_{ \pi \in \Pi_m } \left\{ y_1 : \rho\left( x_{\pi(1)}^{j}, y_{1} \right) \leq \frac{\varepsilon}{2} \right\} \times \cdots \times \left\{ y_m : \rho\left( x_{\pi(m)}^{j}, y_{m} \right) \leq \frac{\varepsilon}{2} \right\}  \\
&=& \bigcup_{ \pi \in \Pi_m } \overline{ B_{\rho}\left( x_{\pi(1)}^{j} , \varepsilon/2 \right)} \times \cdots \times \overline{ B_{\rho}\left( x_{\pi(m)}^{j} , \varepsilon/2 \right)}.
\end{eqnarray*}
Applying Assumption (A3) yields
\begin{eqnarray*}
&& \widetilde{J}_{\eta,m} \left( \left\{ (y_1, \dots, y_m): \min_{\pi \in \Pi_m } \sum_{i=1}^{m} \rho \left( x_{i}^{j}, y_{\pi(i)} \right) \leq  \frac{\varepsilon}{2} \right\} \right) \\
&\leq & \theta^{m} \left( \bigcup_{ \pi \in \Pi_m } \overline{ B_{\rho}\left( x_{\pi(1)}^{j} , \varepsilon/2 \right)} \times \cdots \times \overline{ B_{\rho}\left( x_{\pi(m)}^{j} , \varepsilon/2 \right)} \right) \\
&\leq& \sum_{ \pi \in \Pi_m } \theta\left( \overline{ B_{\rho}\left( x_{\pi(1)}^{j} , \varepsilon/2 \right)} \right) \cdots \theta\left( \overline{ B_{\rho}\left( x_{\pi(m)}^{j} , \varepsilon/2 \right)} \right) \\
&\leq& m! K_2^{m} \left( \varepsilon/2 \right)^{m\sigma},
\end{eqnarray*}
and hence,
\begin{equation*}
\mathcal{L}(\eta) \left( \overline{B_{D}( \mu_j, \varepsilon/2 )} \right) \leq m! K_2^{m} \left( \varepsilon/2 \right)^{m\sigma}  P(|\eta| =m),
\end{equation*}
for all $j=1, \dots,  k$.

On the other hand, since $\overline{B_{D}( \mu_1, \varepsilon/2 )}, \dots, \overline{B_{D}( \mu_k, \varepsilon/2 )} $ cover $\mathfrak{B} \cap \mathfrak{N}_m$, it follows by \eqref{eqlo:aboundfortaum} that
\begin{align*}
\frac{1}{2} P( |\eta| =m ) &\leq \mathcal{L}(\eta) ( \mathfrak{B} \cap \mathfrak{N}_m ) \leq \mathcal{L}(\eta) \left( \bigcup_{j=1}^{k} \overline{B_{D}( \mu_j, \varepsilon/2 )} \right) \\
&\leq \sum_{j=1}^{k} \mathcal{L}(\eta) \left(  \overline{B_{D}( \mu_j, \varepsilon/2 )} \right)
\leq k m! K_2^{m} \left( \varepsilon/2 \right)^{m\sigma}  P(|\eta| =m).
\end{align*}
Therefore,
\begin{equation*}
\mathcal{M}_{\varepsilon}(\mathfrak{B}) \geq \mathcal{M}_{\varepsilon}(\mathfrak{B}\cap \mathfrak{N}_{m} ) = k \geq \frac{1}{ 2 m! K_2^{m} \left( \varepsilon/2 \right)^{m\sigma} }.
\end{equation*}
By the arbitrariness of $\mathfrak{B}$, it follows
\begin{equation*}
  \mathcal{M}_{\varepsilon}( \mathcal{L}(\eta), \tau_m ) \geq \frac{1}{ 2 m! K_2^{m} \left( \varepsilon/2 \right)^{m\sigma} }.
\end{equation*}

Finally, it follows by a straightforward calculation that
\begin{equation*}
 \mathcal{M}_{\varepsilon}( \mathcal{L}(\eta), \tau_m ) \geq \frac{1}{ 2 m! K_2^{m} \left( \varepsilon/2 \right)^{m\sigma} } \geq \varepsilon^{-m(\sigma -\kappa)}
\end{equation*}
for all $\varepsilon \leq \frac{  2^{\sigma/\kappa} }{K_2^{1/\kappa} (2 m!)^{1/(m \kappa)}}$. We conclude the proof.
\end{proof}

With the above preparation, we proceed to present the main result regarding the lower bound.

\begin{theorem}  \label{thm:lowerbound}
Let $\eta$ be a point process satisfying Assumption (A3). Then, for any $0 < \kappa < \sigma$, there exists an integer $N$, depending on $\sigma$, $\kappa$, $c_D$, and $K_2$, such that for all $n \geq N$, 
\begin{equation*}
W_{p}^{p} \left( \widehat{\mathcal{L}}_{n}(\eta), \mathcal{L}(\eta) \right) \geq 2^{-(2p+1)} \sup_{\substack{ m: 1 \leq m \leq \left \lfloor ( \log n )^{2/3} \right\rfloor  \\ P( |\eta | =m) >0}} P( |\eta| =m ) n^{- \frac{p}{m(\sigma-\kappa)} }.
\end{equation*}
Consequently, for $n \geq N$,
\begin{equation*}
W_{p} \left( \widehat{\mathcal{L}}_{n}(\eta), \mathcal{L}(\eta) \right) \geq 2^{-2-1/p} \sup_{\substack{ m: 1 \leq m \leq \left \lfloor ( \log n )^{2/3} \right\rfloor  \\ P( |\eta | =m) >0}} \left( P( |\eta| =m ) \right)^{1/p}  n^{- \frac{1}{m(\sigma-\kappa)} }.
\end{equation*}
\end{theorem}

\begin{proof}
Since the second assertion is a direct result of the first, we only prove the first.

For any $0 < \kappa < \sigma$ and any $m \leq \left \lfloor ( \log n )^{2/3} \right\rfloor$,
\begin{eqnarray*}
&& \left( \frac{ 2^{\sigma/\kappa} }{  K_2^{1/\kappa} ( 2 m! )^{1/(m \kappa)} } \wedge c_D \right)^{- m(\sigma -\kappa)} 
= \frac{  K_2^{m(\sigma-\kappa)/\kappa} 2^{(\sigma -\kappa)/\kappa} (m!)^{(\sigma -\kappa)/\kappa} }{ 2^{m\sigma(\sigma-\kappa)/\kappa}} \vee c_D^{- m(\sigma-\kappa)}  \\
&\leq& \left( 2^{(\sigma-\kappa)/\kappa} K_2^{m(\sigma-\kappa)/\kappa} \left(m^{m}\right)^{(\sigma -\kappa)/\kappa} \right) \vee e^{ m |\log c_D| (\sigma-\kappa) } \\
&=& \left( 2^{(\sigma-\kappa)/\kappa} e^{(\sigma -\kappa) m \log m /\kappa +  (\sigma -\kappa) m \log K_2/\kappa } \right) \vee e^{ m |\log c_D| (\sigma-\kappa) } \\
&\leq& \left( 2^{(\sigma-\kappa)/\kappa} e^{ 2(\sigma -\kappa) (\log n)^{2/3} \log \log n /(3\kappa) +  (\sigma -\kappa) (\log n)^{2/3} |\log K_2|/\kappa } \right) \vee e^{  |\log c_D| (\sigma-\kappa) (\log n)^{2/3} },
\end{eqnarray*}
where in the last inequality we have invoked the result $ m \leq \left \lfloor ( \log n )^{2/3} \right\rfloor \leq ( \log n )^{2/3} $. It is straightforward that the right-hand side of the last display is less than $e^{\log n}$ for all $n$ greater than or equal to some integer $N$, where $N$ depends on $\sigma$, $\kappa$, $c_D$, and $K_2$. Therefore, we have
\begin{equation*}
  n > \left( \frac{ 2^{\sigma/\kappa} }{  K_2^{1/\kappa} ( 2 m! )^{1/(m \kappa)} } \wedge c_D \right)^{- m(\sigma -\kappa)} 
\end{equation*}
for all $ 1 \leq  m \leq \left \lfloor ( \log n )^{2/3} \right\rfloor $. Applying Proposition \ref{prop:sixthWeed} with $\mu = \mathcal{L}(\eta)$, $\nu = \widehat{\mathcal{L}}_{n}(\eta)$, $\tau = \tau_m$, $t = m(\sigma-\kappa)$ and $ \varepsilon' = \frac{  2^{\sigma/\kappa} }{K_2^{1/\kappa} (2 m!)^{1/(m \kappa)}} \wedge c_D $, together with Lemma \ref{lm:coveringlowerbound}, we obtain
\begin{align*}
W_{p}^{p} \left( \widehat{\mathcal{L}}_{n}(\eta), \mathcal{L}(\eta) \right) \geq \tau_{m} 4^{-p} n^{-\frac{p}{m(\sigma - \kappa)}} = 2^{-(2p+1)} P( |\eta| =m ) n^{-\frac{p}{m(\sigma - \kappa)}},
\end{align*}
for any $m$ satisfying $ 1 \leq  m \leq \left \lfloor ( \log n )^{2/3} \right\rfloor$ and $  P( |\eta| =m ) >0 $. Taking the supremum over $m$, the desired result follows.
\end{proof}

\begin{remark}
If we assume, in addition, that $P(|\eta| = m) \geq K_3 e^{-\lambda m}$ for all $m \in \mathbb{N}^+$, for some constant $K_3 > 0$, and that the $\sigma$ appearing in Assumption (A3) equals $\mathrm{dim}_{M}(S)$, then
\begin{equation*}
  \left( P( |\eta| =m ) \right)^{1/p}  n^{- \frac{1}{m(\sigma-\kappa)} } \geq K_{3}^{\frac{1}{p}} e^{- \frac{\lambda m}{p} } n^{- \frac{1}{m \left( \mathrm{dim}_{M}(S) -\kappa \right)} }.
\end{equation*}
For sufficiently large $n$, applying Theorem \ref{thm:lowerbound} and setting $ m = \left\lceil \sqrt{ p \log n / \left( \lambda \left(  \mathrm{dim}_{M}(S) -\kappa  \right) \right) } \right\rceil$, it follows after a routine calculation that
\begin{align*}
W_{p} \left( \widehat{\mathcal{L}}_{n}(\eta), \mathcal{L}(\eta) \right) &\geq 2^{-2- \frac{1}{p}} \sup_{\substack{ m: 1 \leq m \leq \left \lfloor ( \log n )^{2/3} \right\rfloor  \\ P( |\eta | =m) >0}} \left( P( |\eta| =m ) \right)^{1/p}  n^{- \frac{1}{m \left( \mathrm{dim}_{M}(S) -\kappa \right)} }   \\
&\geq 2^{-2- \frac{1}{p}} K_{3}^{\frac{1}{p} } \sup_{\substack{ m: 1 \leq m \leq \left \lfloor ( \log n )^{2/3} \right\rfloor  \\ P( |\eta | =m) >0}} e^{- \frac{\lambda m}{p} } n^{- \frac{1}{m \left( \mathrm{dim}_{M}(S) -\kappa \right)} } \\
& \geq 2^{-2- \frac{1}{p}} K_{3}^{\frac{1}{p} } e^{- \frac{ \lambda}{p}} \times e^{- 2 \sqrt{ \frac{ \lambda \log n }{ p \left( \mathrm{dim}_{M}(S) -\kappa \right)} }}.
\end{align*}
A comparison of the last display with Theorem~\ref{thm:upperbound} shows that the convergence rate derived above is nearly optimal, provided that the corresponding assumptions are satisfied.
\end{remark}

\subsection{Example: Poisson point processes}
Here, we will demonstrate the lower bound through a special case of Poisson point process. Let $\eta$ be a Poisson point process with intensity measure $\Lambda$ satisfying $0 < \Lambda(S) = \lambda < \infty$. We first show that the Poisson Point process $\eta$ satisfies Assumption (A3) under mild conditions. Note that the Janossy measure of a Poisson Point process with intensity measure $\Lambda$ is given by $J_{\eta, m}(\cdot) = \frac{e^{-\Lambda(S)}}{m!}\Lambda^{m}(\cdot)$. For more details, see \cite{Last_Penrose_2017}. Therefore, we have
\begin{equation*}
\widetilde{J}_{\eta, m} (C) = \frac{\Lambda^{m}(C)}{\Lambda^{m}(S^{m})},
\end{equation*}
for any Borel set $C \subseteq S^{m}$. Define the normalized intensity measure $\theta(\cdot):=\frac{\Lambda(\cdot)}{\Lambda(S)}$. Then, $\theta^{m}(C)=\frac{\Lambda^{m}(C)}{\Lambda^{m}(S)}$, which implies $\widetilde{J}_{\eta, m} = \theta^{m}$. This verifies the first part of Assumption (A3).

For the second part, additional conditions on $\Lambda$ and the space $S$ may be required. For example, suppose $S$ is a Euclidean space and $\Lambda(C) \leq K \mu(C)$ for any $C \subset S$, where $K$ is a fixed constant and $\mu$ denotes the Lebesgue measure. In this case, $\sigma$ in Assumption (A3) equals $\mathrm{dim}_{M}(S)$, and $K_2$ depends on $K$, $\Lambda(S)$, and $\mathrm{dim}_{M}(S)$. By applying Theorem \ref{thm:lowerbound}, we have
\begin{align}
W_{p} \left( \widehat{\mathcal{L}}_{n}(\eta), \mathcal{L}(\eta) \right) &\geq 2^{-2-1/p} \sup_{\substack{ m: 1 \leq m \leq \left \lfloor ( \log n )^{2/3} \right\rfloor  \\ P( |\eta | =m) >0}} \left( P( |\eta| =m ) \right)^{1/p}  n^{- \frac{1}{m(\mathrm{dim}_{M}(S)-\kappa)} } \nonumber \\
&= c \sup_{\substack{ m: 1 \leq m \leq \left \lfloor ( \log n )^{2/3} \right\rfloor }} \left( \frac{e^{-\lambda} \lambda^{m}}{m!} \right)^{1/p}  n^{- \frac{1}{m(\mathrm{dim}_{M}(S)-\kappa)} }  \nonumber \\
&\geq c \sup_{\substack{ m: 1 \leq m \leq \left \lfloor ( \log n )^{2/3} \right\rfloor }} \left( e^{-\frac{1}{12}} \frac{1}{\sqrt{2\pi m }}\left(\frac{\lambda e}{m}\right)^m \right)^{1/p}  n^{- \frac{1}{m(\mathrm{dim}_{M}(S)-\kappa)} } \nonumber \\
&= c \sup_{\substack{ m: 1 \leq m \leq \left \lfloor ( \log n )^{2/3} \right\rfloor }} e^{-\frac{1}{2p}\log m + \frac{\log \lambda +1}{p} m}  e^{ - \frac{1}{p} m \log m -\frac{1}{m(\mathrm{dim}_{M}(S) - \kappa)} \log n},  \label{eqlow:theexponentexample}
\end{align}
where the last inequality follows from Robbins’ refinement of Stirling’s formula, and the constant $c$ may vary from line to line.
By setting $m = \left\lfloor \sqrt{\frac{2p}{\mathrm{dim}_{M}(S) - \kappa} \frac{\log n}{\log \log n}} \right \rfloor$, we obtain that the leading term in the exponent in \eqref{eqlow:theexponentexample} is $-\sqrt{\frac{2}{p(\mathrm{dim}_{M}(S) - \kappa  )}} \sqrt{\log n \log \log n} $. Therefore, for sufficiently large $n$, we have
\begin{equation*}
W_{p} \left( \widehat{\mathcal{L}}_{n}(\eta), \mathcal{L}(\eta) \right) \geq c \, e^{- (1 + \chi)\sqrt{\frac{2}{p (\mathrm{dim}_{M}(S) - \kappa)}} \sqrt{\log n \log \log n}},
\end{equation*}
for some constant $0 < \chi <1$.

Comparing this rate with that obtained in Section \ref{subsec:exampleupperbound} shows that the rate of convergence we derived for the Poisson point process is nearly optimal.

\section{Concentration for the Wasserstein distance} \label{sec:concentration}
In this section, we establish concentration properties of the Wasserstein distance in the context of point processes, thereby providing insights into the construction of rejection regions in hypothesis testing. A common approach to deriving concentration inequalities is through McDiarmid’s inequality. Let $X_1,\dots,X_n$ be independent random variables. Suppose $f:\mathcal{X}_1 \times \dots \times \mathcal{X}_n \to \mathbb{R}$ is a function satisfying the bounded differences property, that is, substituting the value of the $i$th coordinate $x_i$, while keeping all others fixed, changes the value of $f$ by at most $c_i$. More formally, for all $i \in \{1, \dots, n\}$, all $x_1 \in \mathcal{X}_1, \dots, x_n \in \mathcal{X}_n$, and all $x_i' \in \mathcal{X}_i$,
\[
\big|f(x_1,\dots,x_i,\dots,x_n)-f(x_1,\dots,x_i',\dots,x_n)\big|\le c_i.
\] 
Then, The McDiarmid’s inequality \cite{mcdiarmid1989method} states that, for any $\varepsilon > 0$,
\[
P\left(f(X_1,\dots,X_n)-E[f(X_1,\dots,X_n)] \geq \varepsilon \right)
\leq \exp\left(-\frac{2\varepsilon^2}{\sum_{i=1}^n c_i^2}\right).
\]  

We first observe that the Wasserstein distance between the empirical measure and the distribution of the underlying point process satisfies the bounded differences property. Indeed, it follows from the triangle inequality together with the properties of the Wasserstein distance on point masses that
\begin{eqnarray}
&&\left|W_p \left(\frac{1}{n}\sum_{i=1}^{n}\delta_{\eta_i}, \mathcal{L}(\eta) \right) - W_p\left(   \frac{1}{n}\sum_{i\neq j}  \delta_{\eta_i}+ \frac{1}{n}\delta_{\eta_j^\prime} , \mathcal{L}(\eta)\right)\right| \nonumber  \\
&\leq& \left|W_p \left( \frac{1}{n}\sum_{i=1}^{n}\delta_{\eta_i}, \sum_{i\neq j} \frac{1}{n} \delta_{\eta_i}+  \frac{1}{n} \delta_{\eta_j^\prime} \right)    \right| \nonumber \\
&\leq& \left|n^{-1/p} D(\eta_j, \eta_j^{\prime})\right|.  \label{eqcon:boundforchangeone}
\end{eqnarray}

It is important to emphasize that the specific choice of metric on $\mathfrak{N}(S)$ can significantly affect the concentration properties. For instance, if we replace our metric $D$ with a bounded metric, the bounded differences of
$W_{p}(\widehat{\mathcal{L}}_{n}(\eta), \mathcal{L}(\eta))$ appearing in the last display
scale as $n^{-1/p}$. Consequently, a direct application of McDiarmid’s inequality yields a concentration rate of order $\exp(-n^{2/p-1})$ for $W_{p}(\widehat{\mathcal{L}}_{n}(\eta), \mathcal{L}(\eta))$, which is meaningful only if $p \in [1,2)$. However, since the metric $D$ satisfying Condition~\ref{assume:metric} (iii) is unbounded, an alternative approach is required to refine the results that follow directly from the standard McDiarmid’s inequality. A common strategy is to control the moments, for example, via a Bernstein-type moment condition. Nevertheless, the restriction on $p$ remains, and hence the subsequent results are applicable only when $1 \leq p < 2$.

Next, we begin with a Bernstein-type McDiarmid’s inequality, followed by our concentration results for the Wasserstein distance.

\begin{proposition}\label{prop:Bernstein}
Let $X_1, \dots, X_n$ be independent random elements with $X_i \in \mathcal{X}_i$, and let $X_i'$ denote an independent copy of $X_i$. Define $X= (X_1, \dots, X_n)$ and $X^{\prime}_{(i)} = (X_1, \ldots, X_{i-1}, X_i', X_{i+1}, \ldots, X_n)$. Additionally, let $f : \prod_{i=1}^n \mathcal{X}_i \to \mathbb{R}$ satisfy $\mathbb{E}|f(X)| < \infty$. For each $i$, define the difference
\[
D_i := f(X) - f(X^{\prime}_{(i)}).
\]
Suppose that there exist constants $\sigma_i > 0$ and $M > 0$ such that for all integers $k \geq 2$,
\[
E\left(|D_i|^k \mid X_{-i} \right) \leq \frac{1}{2} \sigma_i^2 k! M^{k-2} \quad \text{a.s.},
\]
where $X_{-i} = ( X_1, \dots, X_{i-1}, X_{i+1}, \dots, X_n ) $ is the random vector  obtained by removing $X_i$ from $X$.
Then for all $t > 0$, the following concentration inequality holds:
\[
P\left(f(X) - Ef(X) \geq t\right) \leq \exp\left(-\frac{t^2}{2\sigma^2 + 2tM}\right),
\]
where $\sigma^2 = \sum_{i=1}^n \sigma_i^2$.
\end{proposition}

\begin{proof}
This result follows directly from Theorem 1.2A in \cite{Victor1999}; see also \cite{van1995exponential} for a related discussion.
\end{proof}

\begin{corollary}  \label{coro:concentration}
Suppose that the metric space $(S, \rho)$ is bounded and that Assumption (A2) holds. Then,
\begin{equation*}
P\left( W_{p} \left( \widehat{\mathcal{L}}_{n}(\eta), \mathcal{L}(\eta) \right) - E\left[ W_{p} \left( \widehat{\mathcal{L}}_{n}(\eta), \mathcal{L}(\eta) \right) \right] > \varepsilon \right) \leq e^{-\frac{\varepsilon^2 \lambda^3 }{ 16 K_1 e^{\lambda}  C_D^2 \, n^{1- 2/p}  + 4 \lambda^2 \varepsilon  C_D \,  n^{- 1/p }  }}.
\end{equation*}
\end{corollary}

\begin{proof}
As before, let $\eta_1, \dots, \eta_n$ be i.i.d. realizations of $\eta$, and let $ \widehat{\mathcal{L}}_{n}(\eta) $ denote the empirical measure $\sum_{i=1}^{n} \delta_{\eta_i}/n $. For each $i$, let $\eta_{i}'$ be an independent copy of $\eta_i$. Define $X=(\eta_1,\dots,\eta_n)$, $X_{-i} = (\eta_1,\dots, \eta_{i-1},   \eta_{i+1}, \dots,\eta_n)$, and $X_{(i)}^{\prime} = (\eta_1,\dots, \eta_{i-1},  \eta_i^{\prime}, \eta_{i+1}, \dots,\eta_n)$ for all $i=1,\dots, n$. Our goal is to apply Proposition \ref{prop:Bernstein} by substituting $\prod_{i=1}^{n} \mathcal{X}_i $ with $\mathfrak{N}^{n}(S)$ and $f(X)$ with $ W_{p} ( \widehat{\mathcal{L}}_{n}(\eta), \mathcal{L}(\eta) ) $, thereby obtaining concentration results for $ W_{p} ( \widehat{\mathcal{L}}_{n}(\eta), \mathcal{L}(\eta) ) $.

Taking conditional expectation on both sides of \eqref{eqcon:boundforchangeone}, we obtain
\begin{eqnarray*}
E\left[ \left| D_j  \right|^{k} \, \Big| \, X_{-j} \right]
\leq E\left[ n^{-k/p} \left| D(\eta_j, \eta_j^{\prime}) \right|^{k} \, \Big| \, X_{-j}  \right]
= E\left[ n^{-k/p} \left| D(\eta_j, \eta_j^{\prime}) \right|^{k}  \right],
\end{eqnarray*}
where the last equality follows by the independence between $(\eta_j,  \eta_{j}^{\prime})$ and $ X_{-j} $. Noting that
\begin{align*}
\left| D(\eta_j, \eta_j^{\prime}) \right|^{k} &\leq \left(  \left(|\eta_j| \vee |\eta_j^{\prime}|\right)  \left( C_D \right)\right)^{k}
\leq \left(  \left(|\eta_j| + |\eta_j^{\prime}|\right)  \left( C_D \right)\right)^{k} \\
&\leq 2^{k-1} \left( C_D \right)^{k}\left( |\eta_j|^{k} + |\eta_j^{\prime}|^{k} \right),
\end{align*}
it follows that
\begin{align*}
E\left[ \left| D_j  \right|^{k} \, \Big| \, X_{-j} \right] &\leq n^{-k/p} 2^{k-1} \left( C_D \right)^{k} \left( E\left[ |\eta_j|^{k} \right]  +  E\left[ |\eta_j^{\prime}|^{k} \right] \right) \\
& = n^{-k/p} 2^{k} \left( C_D \right)^{k} E\left[ |\eta_j|^{k} \right].
\end{align*}
Furthermore, by Assumption (A2), we have
\begin{align*}
E\left[ |\eta_j|^{k} \right] &= \sum_{m=0}^{\infty} m^{k} P\left( |\eta_{j}| =m \right) \leq \sum_{m=0}^{\infty} m^{k} \times K_1 e^{-\lambda m} =  K_1 \sum_{m=1}^{\infty} m^{k}  e^{-\lambda m} \\
& \leq K_1 \int_{1}^{\infty} x^{k} e^{-\lambda (x-1) } \, dx  \leq K_1 e^{\lambda} \int_{0}^{\infty} x^{k} e^{-\lambda x}\, dx = K_1 e^{\lambda} k! \lambda^{-(k+1)}.
\end{align*}
Combining the last two displays, it follows that
\begin{align*}
E\left[ \left| D_j  \right|^{k} \, \Big| \, X_{-j} \right] &\leq K_1 e^{\lambda} k!  n^{-k/p} 2^{k} \left( C_D \right)^{k} \lambda^{-(k+1)} \\
&= \frac{1}{2} \left(  8 K_1 e^{\lambda} C_D^2 \lambda^{-3}  n^{-2/p}  \right) \times k! \times \left( 2  C_D  \lambda^{-1}  n^{-1/p} \right)^{k-2}.
\end{align*}
Hence, the desired result follows directly from Proposition \ref{prop:Bernstein} after a suitable rearrangement.
\end{proof}

\begin{remark}
In the proof of Corollary \ref{coro:concentration}, if we apply Proposition \ref{prop:Bernstein} with $f(X) = - W_{p} ( \widehat{\mathcal{L}}_{n}(\eta), \mathcal{L}(\eta) )$, we obtain
\begin{equation*}
P\left( W_{p} \left( \widehat{\mathcal{L}}_{n}(\eta), \mathcal{L}(\eta) \right) - E\left[ W_{p} \left( \widehat{\mathcal{L}}_{n}(\eta), \mathcal{L}(\eta) \right) \right] <- \varepsilon \right) \leq e^{-\frac{\varepsilon^2 \lambda^3 }{ 16 K_1 e^{\lambda}  C_D^2 \, n^{1- 2/p}  + 4 \lambda^2 \varepsilon  C_D \,  n^{- 1/p }  }}.
\end{equation*}
Therefore, it follows by the union bound that
\begin{equation*}
P\left(\left| W_{p} \left(  \widehat{\mathcal{L}}_{n}(\eta), \mathcal{L}(\eta) \right) - E\left[ W_{p} \left( \widehat{\mathcal{L}}_{n}(\eta), \mathcal{L}(\eta) \right)  \right]\right| > \varepsilon \right) \leq 2 e^{-\frac{\varepsilon^2 \lambda^3 }{ 16 K_1 e^{\lambda}  C_D^2 \, n^{1- 2/p}  + 4 \lambda^2 \varepsilon  C_D \,  n^{- 1/p }  }}.
\end{equation*}
\end{remark}

\section{Further implications}
\label{sec:insights}
This section presents three examples illustrating how the preceding theoretical results may inform related problems. These examples highlight potential implications of the theory and suggest possible directions for its broader application.

\subsection{Oracle goodness-of-fit test}

We first illustrate how the concentration inequality can be combined with an upper bound for the expected Wasserstein distance to obtain a theoretical goodness-of-fit test for a fully specified point-process law. The result has an oracle character in two respects.

First, its rejection threshold depends on a deterministic upper bound for the expected Wasserstein distance, whose
multiplicative constant depends, for example, on the tail behavior of the
point counts and the geometry of the underlying space, and is therefore not
explicit in our general convergence-rate result. 


Second, the quantity $W_p( \widehat{\mathcal{L}}_n(\eta), \mathcal{L}(\eta_0) )$
may itself be difficult to evaluate when the null distribution $\mathcal{L}(\eta_0)$ is not available in an explicit form. A natural practical alternative is to approximate $\mathcal{L}(\eta_0)$ by an empirical measure $\widehat{\mathcal{L}}_m(\eta_0)$ constructed from $m$ independent realizations simulated under the null, and to replace the above quantity by $W_p( \widehat{\mathcal{L}}_n(\eta), \widehat{\mathcal{L}}_m(\eta_0) )$. The additional approximation error can be controlled using the triangle inequality together with Theorem~\ref{thm:upperbound}, although the resulting bound involves a non-explicit multiplicative constant. Since the reference sample is generated independently of the observed data, its size $m$ can be chosen separately and taken sufficiently large. In the numerical experiments reported in this section, we use $m=500$ simulated null realizations to construct the empirical reference measure and use it whenever an evaluation involving $\mathcal{L}(\eta_0)$ is required. For simplicity, the resulting approximation error is treated as negligible in the numerical analysis, while all theoretical statements continue to be formulated in terms of the exact null distribution $\mathcal{L}(\eta_0)$.


Many classical goodness-of-fit procedures for spatial point processes assess model fit by comparing estimated summary functions, such as Ripley's $K$-function, with their theoretical counterparts or simulation envelopes under the hypothesized model; see \cite{Ripley1977} and \cite{Moller2003}. Such procedures are often applied to a single observed realization of the point process and focus on selected characteristics of the process. In particular, the $K$-function summarizes aspects of the second-order structure but does not, in general, determine the full law of the process. For example, all homogeneous Poisson processes on $\mathbb{R}^2$ have $K(r)=\pi r^2$, irrespective of their intensity. Thus, agreement of $K$-functions alone does not imply equality of point-process laws. The approach considered here uses i.i.d. point-process realizations and compares their empirical distribution directly with the hypothesized distribution on the space of finite counting measures. Accordingly, the discrepancy is measured at the level of the full law of the point process rather than through a particular summary function.

Let $\eta_0$ be a point process with fully specified distribution
$\mathcal{L}(\eta_0)$. Suppose that $\eta_1,\dots,\eta_n$ are i.i.d.
realizations of a point process $\eta$, and consider
\begin{equation*}
H_0:\mathcal{L}(\eta)=\mathcal{L}(\eta_0) \qquad\text{versus}\qquad
H_1:\mathcal{L}(\eta)\neq\mathcal{L}(\eta_0).
\end{equation*}
The test statistic is
\begin{equation*}
T_n:=W_p\left(\widehat{\mathcal{L}}_n(\eta), \mathcal{L}(\eta_0)
\right).
\end{equation*}

For any point process $\eta$ satisfying Assumption (A2) and any $\kappa  >0$, define
\begin{equation*}
r_n(\eta) :=
C_\eta e^{(1+\frac{1}{2p})\log\log n}
e^{-2\sqrt{\frac{\lambda_\eta}
{p(\mathrm{dim}_M(S)+\kappa)}}\sqrt{\log n}},
\end{equation*}
and hence, 
\begin{equation}
E\left[
W_p\left(\widehat{\mathcal{L}}_n(\eta),\mathcal{L}(\eta)\right)\right]
\leq r_n(\eta).
\label{eqcon:rn}
\end{equation}
for sufficiently large $n$.
Here $C_\eta$ denotes the constant appearing in Theorem~\ref{thm:upperbound}, while $\lambda_\eta$ is the corresponding constant in Assumption~(A2).

Similarly, define
\begin{equation*}
\Psi_{n,\eta}(\varepsilon):=\exp\left\{-\frac{\varepsilon^2\lambda_\eta^3}
{16K_{1,\eta}e^{\lambda_\eta}C_D^2 \, n^{1-2/p}+4\lambda_\eta^2\varepsilon C_D \, n^{-1/p}}
\right\}.
\end{equation*}

\begin{corollary}
\label{coro:hypothesistesting}
Suppose that $1\leq p<2$ and that $\eta_0$ satisfies Assumptions (A1)
and (A2), with $K_{1,\eta_0}$ and $\lambda_{\eta_0}$ denoting the corresponding constants in Assumption~(A2), which are assumed to be known. For any significance level
$\gamma\in(0,1)$, define 
\begin{equation*}
\varepsilon_{n,\eta_0,\gamma} = \inf \left\{\varepsilon >0:  \Psi_{n,\eta_0}(\varepsilon)\leq\gamma\right\}.
\end{equation*}
We define the test by
\begin{equation*}
\phi_n=\mathds{1}_{\left\{T_n>r_n(\eta_0)+\varepsilon_{n,\eta_0,\gamma}\right\}},
\end{equation*}
where $\phi_n=1$ indicates rejection of $H_0$.
Then, for sufficiently large $n$, the test has level at most $\gamma$; that is,
\begin{equation*}
P_{H_0}(\phi_n=1)\leq\gamma.
\end{equation*}

Moreover, fix a point process $\widetilde{\eta}$ satisfying Assumptions (A1) and (A2), with $\mathcal{L}(\widetilde{\eta})\neq\mathcal{L}(\eta_0)$. Under the alternative $\mathcal{L}(\eta)=\mathcal{L}(\widetilde{\eta})$, if
\begin{equation}
W_p\left(\mathcal{L}(\widetilde{\eta}),\mathcal{L}(\eta_0)\right)
>
r_n(\eta_0)+\varepsilon_{n,\eta_0,\gamma}+r_n(\widetilde{\eta}),
\label{eqcon:separation}
\end{equation}
then the type II error probability satisfies
\begin{equation*}
P_{\mathcal{L}(\widetilde{\eta})}(\phi_n=0)
\leq
\Psi_{n,\widetilde{\eta}}\left(
W_p\left(\mathcal{L}(\widetilde{\eta}),\mathcal{L}(\eta_0)\right)
-r_n(\eta_0)-\varepsilon_{n,\eta_0,\gamma}
-r_n(\widetilde{\eta})
\right),
\end{equation*}
where $P_{\mathcal{L}(\widetilde{\eta})}$ denotes probability when the observations are i.i.d. point processes with law $\mathcal{L}(\widetilde{\eta})$.
%
\end{corollary}

\begin{proof}
Under $H_0$, it follows from Theorem~\ref{thm:upperbound} that
\begin{equation*}
E[T_n]\leq r_n(\eta_0).
\end{equation*}
Together with Corollary~\ref{coro:concentration}, we have
\begin{align*}
P_{H_0}(\phi_n=1)
&= P_{H_0}\left( T_n>r_n(\eta_0)+\varepsilon_{n,\eta_0,\gamma} \right)
\leq P_{H_0}\left( T_n-E[T_n]>\varepsilon_{n,\eta_0,\gamma} \right)
\leq \Psi_{n,\eta_0}(\varepsilon_{n,\eta_0,\gamma}) \leq\gamma,
\end{align*}
where the last inequality follows from the definition of $\varepsilon_{n,\eta_0,\gamma}$ and the continuity of $\Psi_{n,\eta_0}$.

Under $H_1$, the triangle inequality gives
\begin{equation*}
W_p\left(\mathcal{L}(\widetilde{\eta}),\mathcal{L}(\eta_0)\right)
\leq W_p\left(\widehat{\mathcal{L}}_n(\eta),\mathcal{L}(\widetilde{\eta})\right)+T_n.
\end{equation*}
Thus, on the event $\{\phi_n=0\}$,
\begin{equation*}
W_p\left(\widehat{\mathcal{L}}_n(\eta),\mathcal{L}(\widetilde{\eta})
\right)\geq W_p\left(\mathcal{L}(\widetilde{\eta}), \mathcal{L}(\eta_0)\right)
-r_n(\eta_0)-\varepsilon_{n,\eta_0,\gamma}.
\end{equation*}
Applying \eqref{eqcon:rn} with $\eta$ replaced by $\widetilde{\eta}$, it follows that
\begin{align*}
P_{\mathcal{L}(\widetilde{\eta})}(\phi_n=0)
&\leq P_{\mathcal{L}(\widetilde{\eta})}\Bigg(W_p\left(\widehat{\mathcal{L}}_n(\eta),\mathcal{L}(\widetilde{\eta})\right)
-
E\left[W_p\left(\widehat{\mathcal{L}}_n(\eta),\mathcal{L}(\widetilde{\eta})\right)\right]\\
&\hspace{2cm}\geq W_p\left(\mathcal{L}(\widetilde{\eta}),
\mathcal{L}(\eta_0)\right)
-r_n(\eta_0)-\varepsilon_{n,\eta_0,\gamma}-r_n(\widetilde{\eta})
\Bigg).
\end{align*}
By \eqref{eqcon:separation}, the right side is positive, and
Corollary~\ref{coro:concentration} gives the desired result.
\end{proof}

Note that condition~\eqref{eqcon:separation} is automatically satisfied for sufficiently large n, since each term on the right-hand side converges to zero.  


The above result provides a theoretically valid rejection rule. However, the upper bounds used in Theorem~\ref{thm:upperbound} and Corollary~\ref{coro:concentration} may yield a conservative rejection threshold in finite samples. We conduct a simple simulation study to investigate this potential conservativeness and its effect on testing performance.


In this simulation study, the null model is a homogeneous Poisson process on $[0,1]$ with rate $2$, and we evaluate the power of the test against the fixed alternative of a homogeneous Poisson process on the same interval with rate $8$. We use $D_1$ and set $\alpha=1$. Specifically, we take $p=1$, set $\kappa= 0.1$, and choose the auxiliary point $s_\alpha=2$, with $\rho$ given by the Euclidean distance on the real line. Since the constant $C_{\eta_0}$ in $r_n(\eta_0)$ is not available explicitly, we conduct a sensitivity analysis using the candidate values
\[
C_{\eta_0}\in\{0.1,0.5,1,2,5\}.
\]
These values are used for illustration and are not guaranteed to satisfy the theoretical upper bound.
We also need to specify the pair $(K_{1,\eta_0},\lambda_{\eta_0})$ in Assumption~(A2), as different admissible choices may lead to different rejection thresholds. Since the theoretical result applies to any pair satisfying Assumption~(A2), we select an admissible pair that minimizes $\varepsilon_{n,\eta_0,\gamma}$ over a prescribed grid, thereby reducing the contribution of the concentration term to the rejection threshold. We again emphasize that this numerical study is intended for illustrative purposes, since $C_{\eta_0}$ in theory also depends on the pair $(K_{1,\eta_0},\lambda_{\eta_0})$.


To reduce the computational cost, we restrict our attention to $n\in\{10,25,50,100\}$, each of which divides $m=500$. In Figure~\ref{fig:GOF1}, we report the Monte Carlo estimates of the Type I error probabilities based on $R=100$ independent replications for each sample size $n$. Figure~\ref{fig:GOF2} presents the corresponding Monte Carlo estimates of the Type II error probabilities, together with approximations of the theoretical upper bounds established in Corollary~\ref{coro:hypothesistesting} using the candidate constants.

\begin{figure}[!ht]
    \centering    \includegraphics[width=0.6\linewidth]{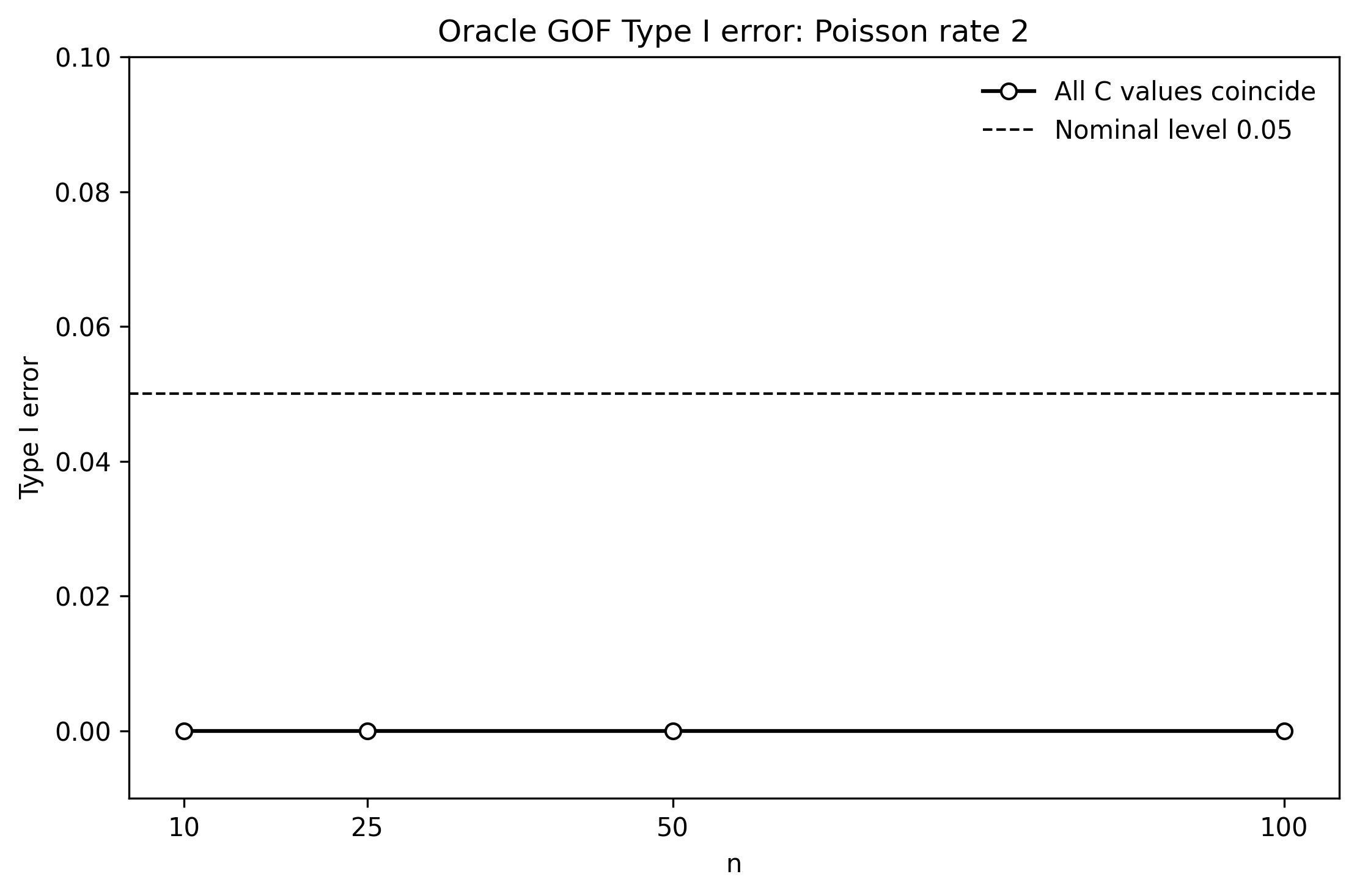}
    \caption{Type I error probabilities of the oracle goodness-of-fit test under the null. The null and alternative models are homogeneous Poisson processes on $[0,1]$ with rates $2$ and $8$, respectively.}
    \label{fig:GOF1}
\end{figure}

\begin{figure}[!ht]
    \centering
    \includegraphics[width=0.6\linewidth]{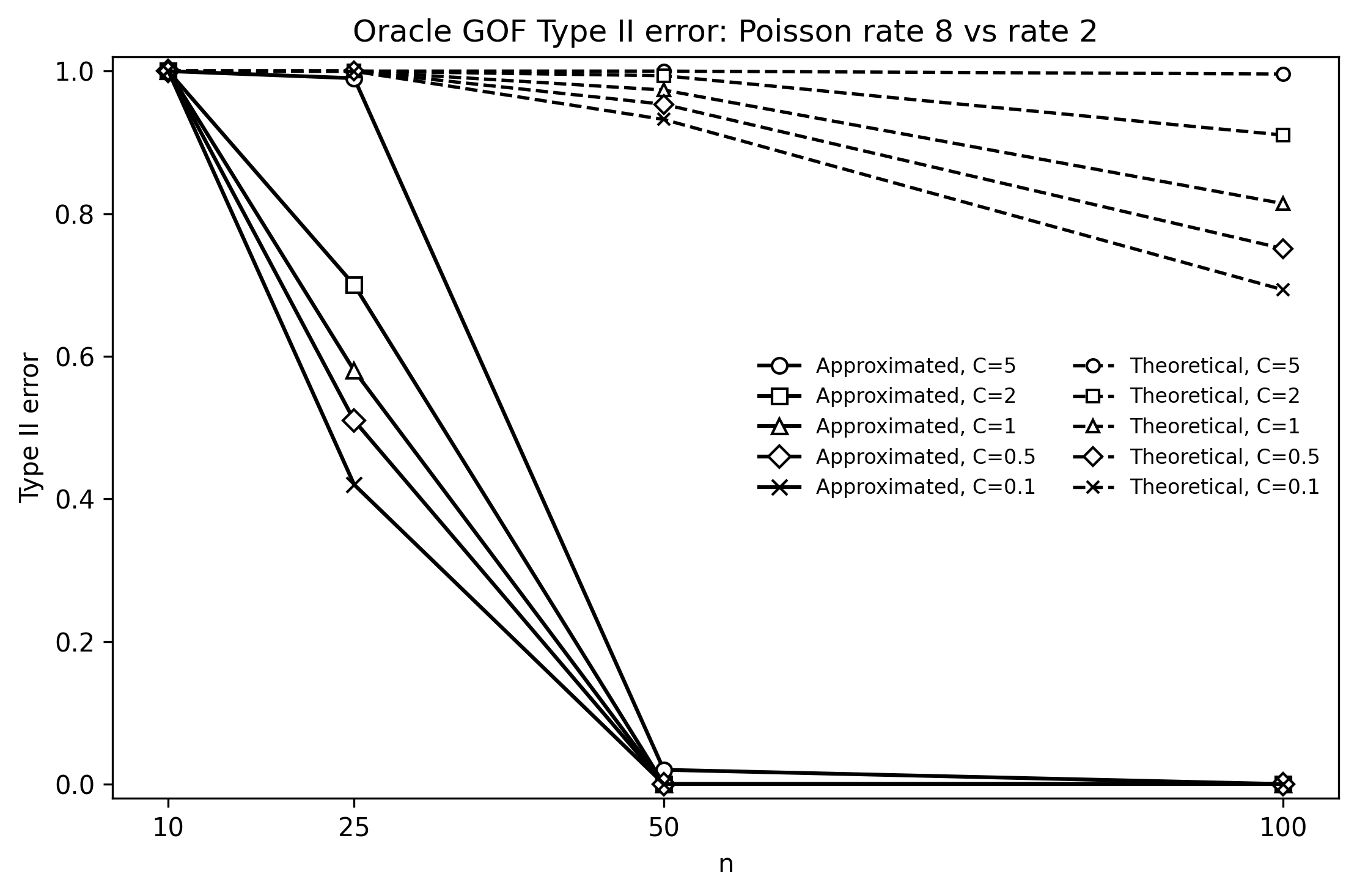}
    \caption{Type II error probabilities of the oracle goodness-of-fit test under a fixed alternative. The null and alternative models are homogeneous Poisson processes on $[0,1]$ with rates $2$ and $8$, respectively.}
    \label{fig:GOF2}
\end{figure}

The estimated Type I error probabilities remain well below the nominal level of $0.05$ even for small candidate values of $C_{\eta_0}$ since $\varepsilon_{n,\eta_0,\gamma}$ is fixed and already conservative. For small $n$, the concentration term $\varepsilon_{n,\eta_0,\gamma}$ is the main contribution to the threshold. Consequently, the test has relatively low power, and variations in $C_{\eta_0}$ have little effect on its performance. As $n$ increases, $\varepsilon_{n,\eta_0,\gamma}$ decreases, and the contribution of $r_n(\eta_0)$ to the rejection threshold becomes relatively more pronounced. As a result, different choices of $C_{\eta_0}$ begin to produce visible differences in the Type II error. For larger $n$, both terms become small relative to the separation term $W_p\left(\mathcal{L}(\widetilde{\eta}),
\mathcal{L}(\eta_0)\right)$ between the null and alternative distributions, and the Monte Carlo estimates of the Type II error probabilities approach zero. By contrast, although the theoretical upper bound on the Type II error exhibits a decreasing trend, it remains relatively large because of the conservativeness of the underlying concentration bound.


Indeed, one could further calibrate $\varepsilon_{n,\eta_0,\gamma}$ by introducing a tuning parameter to obtain a less conservative threshold for a given sample size $n$ and significance level $\gamma$. Alternatively, the theoretical threshold $r_n(\eta_0)+\varepsilon_{n,\eta_0,\gamma}$ could be replaced entirely by a Monte Carlo critical value. Monte Carlo calibration has previously been used for goodness-of-fit tests based on Wasserstein distances induced by matching metrics on finite point measures; see, for example, \cite{Schuhmacher2007}. Our results, however, provide additional insights for this approach by showing that the Type II error remains controlled when a Monte Carlo critical value is employed.



More specifically, independently of the observed data, generate $B$ independent datasets under $H_0$, each consisting of $n$ independent realizations of $\eta_0$. Denote this collection of calibration datasets by $\mathcal{C}$ and compute
\[
T_n^{(b)}
=
W_p\left(
\widehat{\mathcal{L}}_n^{(b)}(\eta_0),
\mathcal{L}(\eta_0)
\right),
\qquad b=1,\ldots,B.
\]
Let $T_n^{1-\gamma}$ denote the empirical $(1-\gamma)$-quantile of $T_n^{(1)},\ldots,T_n^{(B)}$. The resulting test is
\[
\phi_n=\mathds{1}_{\left\{T_n>T_n^{1-\gamma}\right\}},
\]
where $\phi_n=1$ indicates rejection of the null hypothesis.

Conditional on $\mathcal{C}$, the critical value is fixed, and the argument of Corollary~\ref{coro:hypothesistesting} applies. We therefore summarize the resulting conclusion in the following corollary and omit the proof.

\begin{corollary}  \label{coro:hypothesistesting_MC}
Retain the assumptions on $p$, $\eta_0$, and $\widetilde{\eta}$ in Corollary~\ref{coro:hypothesistesting}, and suppose that $\mathcal{L}(\eta)=\mathcal{L}(\widetilde{\eta})$. Then, for sufficiently large $n$, 
the conditional type II error probability satisfies
\[
P_{\mathcal{L}(\widetilde{\eta})}
\left(\phi_n=0\mid\mathcal{C}\right)
\leq
\Psi_{n,\widetilde{\eta}}\left(
W_p\left(
\mathcal{L}(\widetilde{\eta}),
\mathcal{L}(\eta_0)
\right)
-T_n^{1-\gamma}-r_n(\widetilde{\eta})
\right)
\]
almost surely on the event
\[
W_p\left(
\mathcal{L}(\widetilde{\eta}),
\mathcal{L}(\eta_0)
\right)
>
T_n^{1-\gamma}+r_n(\widetilde{\eta}).
\]
\end{corollary}



We now turn to the simulation results. Let $B=500$. Figures~\ref{fig:MC_GOF1} and~\ref{fig:MC_GOF2} report the Monte Carlo estimates of the Type I and Type II error probabilities, respectively, together with approximations of the theoretical upper bounds on the conditional Type II error established in Corollary~\ref{coro:hypothesistesting_MC}. In contrast to the test based on the oracle critical value, the test based on the Monte Carlo critical value yields an estimated Type I error close to the nominal level of $0.05$, while the estimated Type II error is close to zero. This behavior is expected because the null and alternative models, corresponding to Poisson rates $2$ and $8$, respectively, are well separated. The theoretical upper bounds on the conditional Type II error are slightly sharper than those reported in Figure~\ref{fig:GOF2}, although they remain considerably larger than the corresponding Monte Carlo estimates.

Throughout these illustrations, we have used the moderate choice $\alpha=1$. The choice of the auxiliary point may affect both the theoretical bounds and the non-explicit constants appearing in them. For example, a larger $\alpha$ may yield sharper theoretical Type II error bounds. Deriving explicit expressions for these constants and systematically investigating how the choice of the auxiliary point affects the resulting bounds are interesting directions for future research.


\begin{figure}[!ht]
    \centering
    \includegraphics[width=0.6\linewidth]{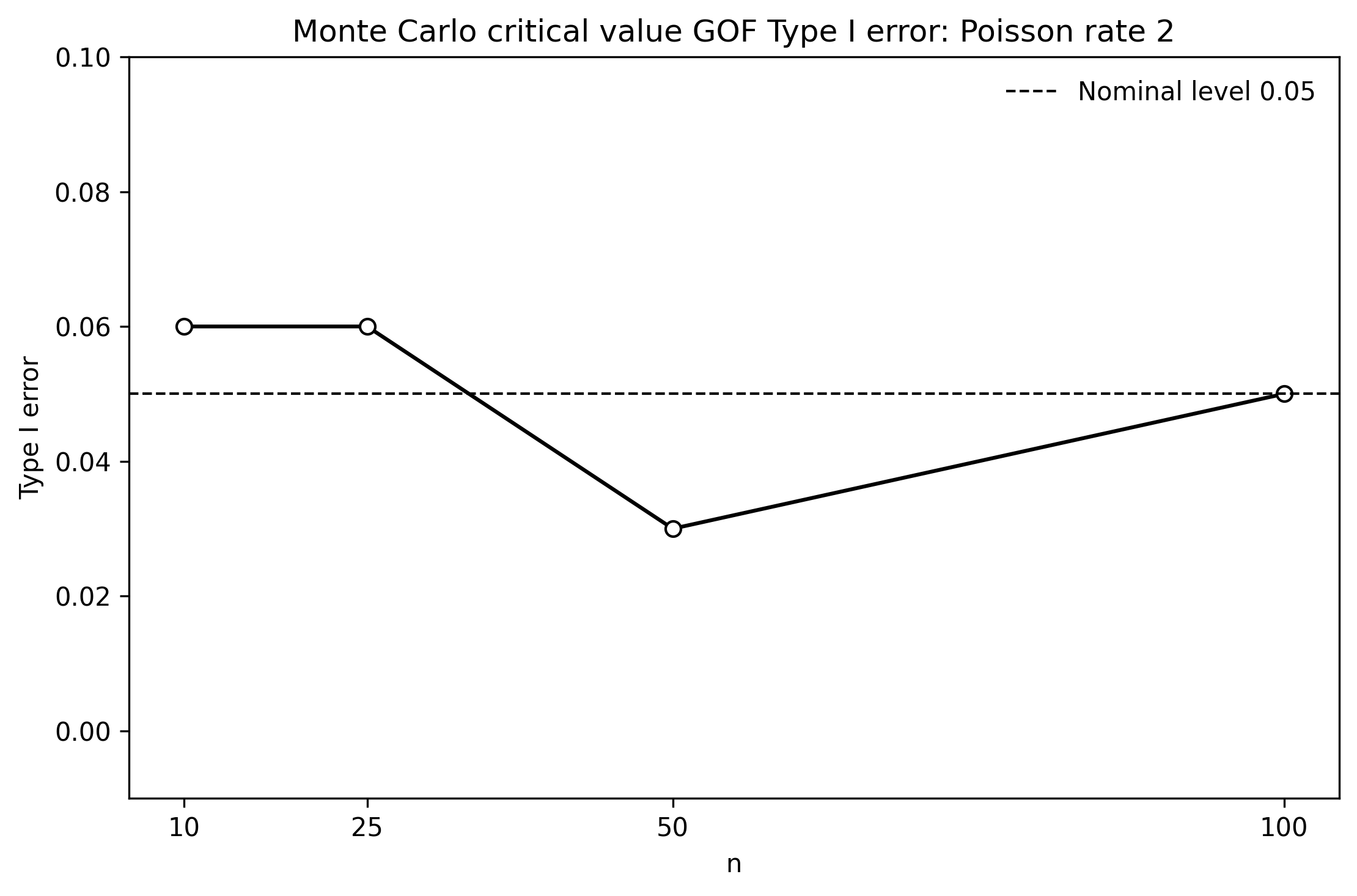}
    \caption{Type I error probabilities of the Monte Carlo goodness-of-fit test under the null.}
    \label{fig:MC_GOF1}
\end{figure}

\begin{figure}[!ht]
    \centering
    \includegraphics[width=0.6\linewidth]{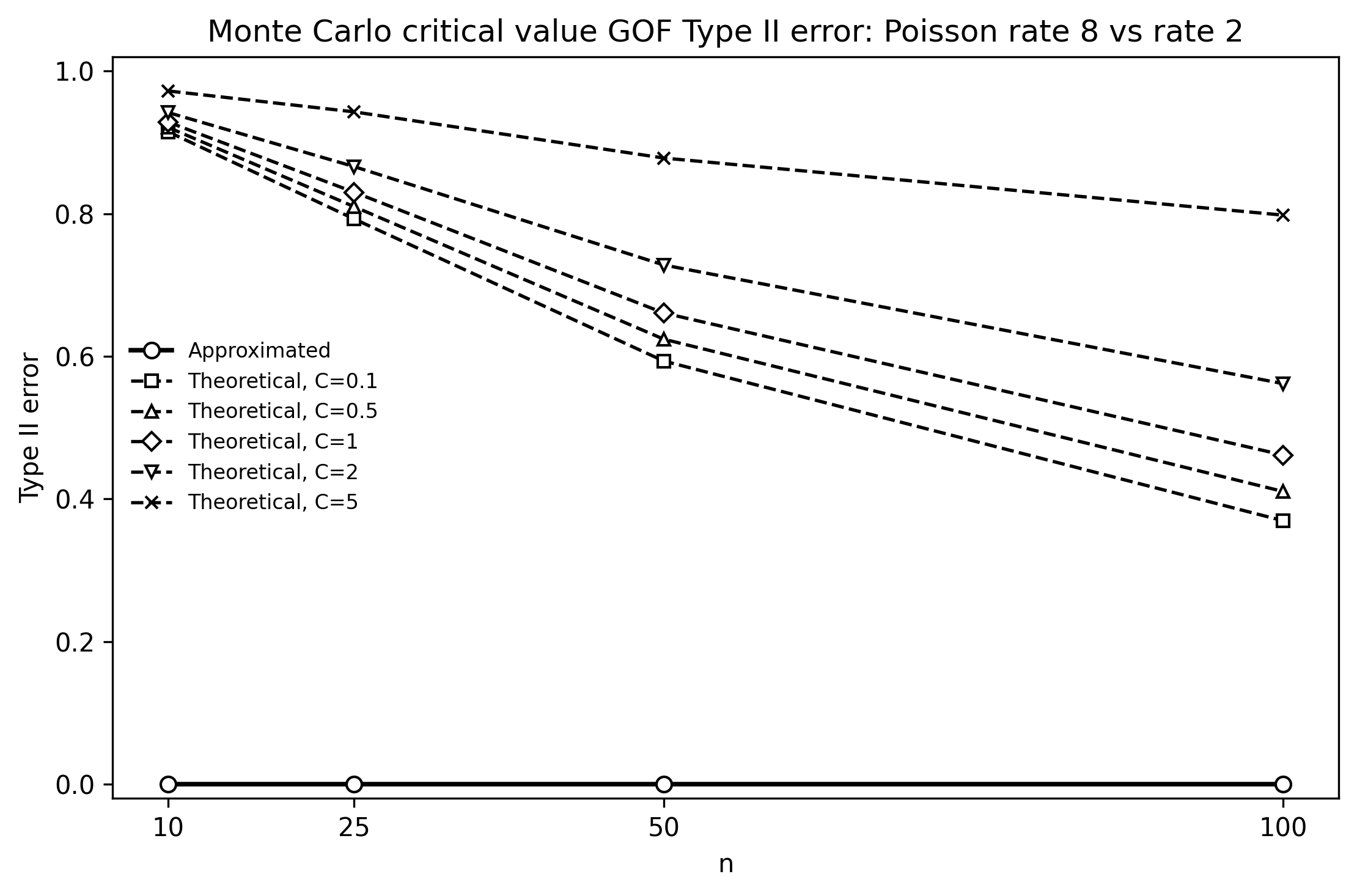}
    \caption{Type II error probabilities of the Monte Carlo goodness-of-fit test under the alternative.}
    \label{fig:MC_GOF2}
\end{figure}

\begin{remark}
Corollary~\ref{coro:hypothesistesting} and Corollary~\ref{coro:hypothesistesting_MC} concern a simple null hypothesis
for which $\mathcal{L}(\eta_0)$ is fully specified. For a composite null,
\begin{equation*}
H_0:\mathcal{L}(\eta)\in
\left\{\mathcal{L}(\eta_\theta):\theta\in\Theta\right\},
\end{equation*}
a corresponding result would require uniform control over $\theta\in\Theta$
of both the expected Wasserstein distance and the concentration bound,
including suitable uniform bounds on the tail constants in Assumption (A2).
We leave this extension for future work.
\end{remark}

\subsection{Approximation of the Campbell measure}
The Campbell measure is a fundamental and essential tool in the theory of point processes, particularly in the context of Palm theory \cite{Baccelli2020RandomMP}.

Let $\eta$ be a point process on $S$. The Campbell measure is a measure on the product space $S \times \mathfrak{N}(S)$, commonly defined by
\begin{equation*}
C_{\eta}(f) = E \left[ \int_{S} f(s,\eta) \eta(ds) \right].   
\end{equation*}
for any nonnegative measurable function $f: S \times \mathfrak{N}(S) \rightarrow \bar{\mathbb{R}}_{+}$. Note that the integral inside the expectation can be regarded as a function of $\eta$. Denote $g_{f}(\eta) = \int_{S} f(s,\eta) \eta(ds)$, so that $g_{f}$ is a function on $\mathfrak{N}(S)$ and
\begin{equation*}
C_{\eta}(f) = E[g_{f}(\eta)] = \int_{\mathfrak{N}(S)} g_{f}(\mu) d \mathcal{L}(\eta) (\mu).
\end{equation*}
Suppose that $g_{f}$ is a Lipschitz function with Lipschitz constant $1$, that is, $|g_{f}(\mu) - g_{f}(\nu)| \le D(\mu,\nu)$. Then, by invoking the duality of the Wasserstein distance and Theorem \ref{thm:upperbound}, one can justify the approximation of the Campbell measure using i.i.d. realizations of $\eta$.

Before proceeding to the details, we state without proof a sufficient condition under which $g_{f}$ is Lipschitz. Extend $f$ to $\tilde{f}$ on $\left( S \cup \{s_{\alpha}\} \right) \times \mathfrak{N}(S) $ by setting $\tilde{f}(s_{\alpha}, \mu) = 0$ for all $\mu \in \mathfrak{N}(S)$. If $\tilde{f}$ is a Lipschitz function with compact support, then $g_{f}$ is also Lipschitz, with a Lipschitz constant depending on both the Lipschitz constant of $\tilde{f}$ and the size of its support. Moreover, if $S$ possesses suitable topological properties---for example, if $S$ is a compact subset of Euclidean space---then the condition reduces to requiring that $f$ itself be Lipschitz with compact support. In this case, such a collection of functions suffices to determine the measure on $S \times \mathfrak{N}(S)$. 

We then turn to the duality of the Wasserstein distance. We use $\mathcal{P}_{1} (\mathfrak{N}) $ to denote the space of probability measures on $\mathfrak{N}$ with finite first moment, that is, if $ \mathcal{L}(\eta) \in \mathcal{P}_{1} ( \mathfrak{N} ) $, then $\int_{\mathfrak{N}}  D( \mu, \mu_0 ) \, d\mathcal{L}(\eta) (\mu) < \infty$. In the setting of point processes, the duality can be formulated as follows: for any $\mathcal{L}(\eta), \mathcal{L}(\xi) \in \mathcal{P}_1(\mathfrak{N})$,
\begin{equation*}
W_1(\mathcal{L}(\eta),\mathcal{L}(\xi)) = \sup_{ \left\| h \right\| _{\mathrm{Lip} } \leq 1 }\left|\int_{\mathfrak{N}(S)} h(\mu) d\mathcal{L}(\eta)(\mu) - \int_{\mathfrak{N}(S)} h(\mu) d\mathcal{L}(\xi)(\mu)\right|,
\end{equation*}
where $\| \cdot\|_{\mathrm{Lip}}$ represents the Lipschitz constant. If $\eta$ satisfies Assumption (A2), then the condition of equation (\ref{pp_space}) holds with $p=1$, implying that $\mathcal{L}(\eta) \in \mathcal{P}_{1}(\mathfrak{N})$. Moreover, it is immediate that $\widehat{\mathcal{L}}_{n}(\eta) \in \mathcal{P}_{1}(\mathfrak{N})$ almost surely. Therefore, for all $f$ such that $ \|g_{f}\|_{\mathrm{Lip}} \leq 1 $, it follows by the duality of the Wasserstein distance that
\begin{eqnarray*}
&& \left| C_{\eta}(f) - \frac{1}{n} \sum_{i=1}^{n} \int_{S} f\left( s, \eta_i \right) \, \eta_{i}(ds)  \right| 
= \left| \int_{\mathfrak{N}(S)} g_{f}(\mu) d\mathcal{L}(\eta)(\mu) - \frac{1}{n} \sum_{i=1}^{n} g_{f}(\eta_i) \right| \\
&=&\left| \int_{\mathfrak{N}(S)} g_{f}(\mu) d\mathcal{L}(\eta)(\mu) - \int_{\mathfrak{N}(S)} g_{f}(\mu) d\widehat{\mathcal{L}}_{n}(\eta)(\mu) \right| \\
&\leq& \sup_{\|h\|_{\mathrm{Lip} \leq 1}}\left| \int_{\mathfrak{N}(S)} h(\mu) d\mathcal{L}(\eta)(\mu) - \int_{\mathfrak{N}(S)} h(\mu) d\widehat{\mathcal{L}}_{n}(\eta)(\mu) \right| \\
&=& W_{1} \left(\mathcal{L}(\eta),\widehat{\mathcal{L}}_{n}(\eta)\right).
\end{eqnarray*}
Suppose that $\eta$ satisfies Assumptions (A1) and (A2).
Applying Corollary \ref{coro:concentration} with $\varepsilon = e^{ -2 \sqrt{\lambda/(\mathrm{dim}_{M}(S) + 2\kappa)} \sqrt{\log n} }$ and $p=1$, together with Theorem \ref{thm:upperbound}, we have, with probability at least $1 - e^{-c n e^{- 4 \sqrt{\lambda/(\mathrm{dim}_{M}(S) + 2\kappa)} \sqrt{\log n} } }$ for some constant $C$ and for sufficiently large $n$,
\begin{equation*}
W_{1} \left(\mathcal{L}(\eta),\widehat{\mathcal{L}}_{n}(\eta) \right) \leq C e^{  -2 \sqrt{\frac{\lambda}{ \mathrm{dim}_{M}(S) + 2\kappa }} \sqrt{\log n} },
\end{equation*}
and hence,
\begin{equation*}
\left| C_{\eta}(f) - \frac{1}{n} \sum_{i=1}^{n} \int_{S} f\left( s, \eta_i \right) \, \eta_{i}(ds)  \right|  \leq C e^{  -2 \sqrt{\frac{\lambda}{ \mathrm{dim}_{M}(S) + 2\kappa }} \sqrt{\log n} }.
\end{equation*}
In other words, the Campbell measure can be approximated using realizations of $\eta$. Furthermore, if $ \|g_{f}\|_{\mathrm{Lip}} > 1$, the above result remains applicable after an appropriate rescaling of $f$ so that $ \|g_{f}\|_{\mathrm{Lip}} \leq 1$.

\subsection{Generative point process models}

Wasserstein learning of generative point process models provides a natural motivation for studying empirical approximations to point-process distributions. In \cite{xiao2017wasserstein}, the authors proposed a generative model for temporal point processes. Their goal is to transform an input noise process, taken to be a homogeneous Poisson process $\xi$, through a parameterized map $g_\theta$, so that the resulting process $g_\theta(\xi)$ has a distribution close to that of the target point process $\eta$. The transformation $g_\theta$, referred to as the generator, is represented by a recurrent neural network. We omit its specific form and focus on the underlying modeling idea. The generator parameters are learned by approximately minimizing the Wasserstein distance between the generated and target point-process distributions. Formally, the learning problem and its dual form are given by
\begin{align*}
\mathop{\mathrm{argmin}}_{\theta}  \ W_1(\mathcal{L}(\eta),\mathcal{L}(g_{\theta}(\xi))) &= \mathop{\mathrm{argmin}}_{\theta}  \ \sup_{ \left\| h \right\| _{\mathrm{Lip} } \leq 1 }\left|\int_{\mathfrak{N}(S)} h(\mu) d\mathcal{L}(\eta)(\mu) - \int_{\mathfrak{N}(S)} h(g_{\theta}(\mu)) d\mathcal{L}(\xi)(\mu)\right| \\
&= \mathop{\mathrm{argmin}}_{\theta}  \ \sup_{ \left\| h \right\| _{\mathrm{Lip} } \leq 1 }\left(\int_{\mathfrak{N}(S)} h(\mu) d\mathcal{L}(\eta)(\mu) - \int_{\mathfrak{N}(S)} h(g_{\theta}(\mu)) d\mathcal{L}(\xi)(\mu)\right)
\end{align*}
where in the first equality, the second integral follows from a standard change-of-variable argument. It is worth noting that Wasserstein distances are understood to be defined with respect to an underlying metric $D$ satisfying Condition~\ref{assume:metric}.

Since the family of Lipschitz functions with constant at most $1$ is too broad to optimize over directly, the authors approximate the Wasserstein distance by restricting to a parameterized family of discriminator functions $\{h_w : w \in \mathcal{W}\}$: 
\begin{equation*}
W_1(\mathcal{L}(\eta),\mathcal{L}(g_{\theta}(\xi))) \approx \ \max_{w \in \mathcal{W}, \, \| h_w \|_{\text{Lip}} \leq 1} 
\left(\int_{\mathfrak{N}(S)} h_w(\mu) d\mathcal{L}(\eta)(\mu) - \int_{\mathfrak{N}(S)} h_w(g_{\theta}(\mu)) d\mathcal{L}(\xi)(\mu)\right).
\end{equation*}
In \cite{xiao2017wasserstein}, the discriminator function is also modeled using a recurrent neural network, with a regularization term imposed to enforce the Lipschitz constraint.

Since the expectations cannot be evaluated exactly in practice, they are replaced by their empirical counterparts. Given the training data, consisting of $n$ realizations $\eta_1,\ldots,\eta_n$ from $\eta$ and $\xi_1,\ldots,\xi_n$ from $\xi$, the authors formulate the following empirical learning problem:
\begin{equation*}
\mathop{\mathrm{argmin}}_{\theta}  \ W_1 \left(\widehat{\mathcal{L}}_{n}(\eta),\widehat{\mathcal{L}}_{n}(g_{\theta}(\xi)) \right)  \approx \mathop{\mathrm{argmin}}_{\theta}  \ \max_{w \in \mathcal{W}, \, \| h_w \|_{\text{Lip}} \leq 1} 
\left( \frac{1}{n} \sum_{i=1}^{n}   h_w(\eta_{i})  
- \frac{1}{n} \sum_{i=1}^{n}  h_w(g_{\theta}(\xi_{i})) \right) .
\end{equation*}


Our results, such as Theorem~\ref{thm:upperbound}, provide quantitative bounds on the Wasserstein distance between the empirical distribution of a point process and its underlying law. Under the stated assumptions, these bounds may help assess the population Wasserstein distance between the target distribution and the distribution induced by the trained generator by quantifying the error arising from empirical approximation, thereby providing useful information for evaluating how well the trained generator approximates the target point-process distribution.

More specifically, provided that one has the optimal solution 
\begin{equation*}
 \hat{\theta} = \mathop{\mathrm{argmin}}_{\theta} 
 \max_{w \in \mathcal{W}, \, \| h_w \|_{\text{Lip}} \leq 1} 
\left( \frac{1}{n} \sum_{i=1}^{n} h_w(\eta_{i}) - \frac{1}{n} \sum_{i=1}^{n} h_w(g_{\theta}(\xi_{i}))  \right).
\end{equation*}
Then, by the triangle inequality, one has
\begin{equation*}
\left| W_1(\mathcal{L}(\eta),\mathcal{L}(g_{\hat{\theta}}(\xi))) - W_1 \left(\widehat{\mathcal{L}}_{n}(\eta),\widehat{\mathcal{L}}_{n}(g_{\hat{\theta}}(\xi)) \right) \right| \leq   W_1 \left( \mathcal{L}(g_{\hat{\theta}}(\xi)),\widehat{\mathcal{L}}_{n}(g_{\hat{\theta}}(\xi)) \right) + W_1 \left(\mathcal{L}(\eta),\widehat{\mathcal{L}}_{n}(\eta)\right).
\end{equation*}
The empirical Wasserstein distance $W_1(\widehat{\mathcal{L}}_n(\eta),\widehat{\mathcal{L}}_n(g_{\hat{\theta}}(\xi)))$ can be computed from the available samples. To use it to assess the corresponding population Wasserstein distance, we also need to control the two empirical approximation terms in the above bound. However, since $\hat{\theta}$ depends on the training samples, Theorem~\ref{thm:upperbound} cannot generally be applied directly to the generated empirical approximation error $W_1 ( \mathcal{L}(g_{\hat{\theta}}(\xi)),\widehat{\mathcal{L}}_{n}(g_{\hat{\theta}}(\xi)) )$ when the same sample is used both to obtain $\hat{\theta}$ and to form the empirical measure. One possible approach is to establish uniform bounds on the empirical approximation errors over the class of generators ${g_\theta}$, which constitutes an interesting direction for future research.

Our results can, however, be applied directly to independent evaluation samples. Let $\eta_1^{\mathrm{ev}},\ldots,\eta_m^{\mathrm{ev}}$ and $\xi_1^{\mathrm{ev}},\ldots,\xi_m^{\mathrm{ev}}$ be fresh i.i.d. samples from $\eta$ and $\xi$, respectively, independent of each other and of all samples used for training. Define
\begin{equation*}
\widehat{\mathcal{L}}_{m}^{\mathrm{ev}}(\eta) = \frac{1}{m} \sum_{j=1}^{m}\delta_{\eta_j^{\mathrm{ev}}}, \qquad
\widehat{\mathcal{L}}_{m}^{\mathrm{ev}}(g_{\hat{\theta}}(\xi)) = \frac{1}{m} \sum_{j=1}^{m} \delta_{g_{\hat{\theta}}(\xi_j^{\mathrm{ev}})}.
\end{equation*}
Conditional on $\hat{\theta}$, these are empirical distributions based on independent samples from $\mathcal{L}(\eta)$ and $\mathcal{L}(g_{\hat{\theta}}(\xi))$. Again, the triangle inequality gives
\begin{equation*}
\left| W_1(\mathcal{L}(\eta),\mathcal{L}(g_{\hat{\theta}}(\xi))) - W_1 \left(\widehat{\mathcal{L}}_{m}^{\mathrm{ev}}(\eta),\widehat{\mathcal{L}}_{m}^{\mathrm{ev}}(g_{\hat{\theta}}(\xi)) \right) \right| \leq   W_1 \left( \mathcal{L}(g_{\hat{\theta}}(\xi)),\widehat{\mathcal{L}}_{m}^{\mathrm{ev}}(g_{\hat{\theta}}(\xi)) \right) + W_1 \left(\mathcal{L}(\eta),\widehat{\mathcal{L}}_{m}^{\mathrm{ev}}(\eta)\right).
\end{equation*}
Theorem~\ref{thm:upperbound} and Corollary~\ref{coro:concentration} can therefore be applied to the two empirical approximation terms, conditional on $\hat{\theta}$, provided their assumptions hold for both processes. The resulting bounds help assess how closely the trained generator $g_{\hat{\theta}}$ transforms the input noise process into a process whose distribution approximates that of the target point process.


\section*{Acknowledgments}
The authors would like to thank Professor Bojan Basrak and Professor Piotr Kokoszka for the valuable discussions. We also sincerely thank the two anonymous referees for their constructive comments and suggestions, which have substantially improved the quality of this manuscript.


\appendix
\section*{Appendix}
\addcontentsline{toc}{section}{Appendix}
\renewcommand{\thesubsection}{A.\arabic{subsection}}
\subsection{Proofs in Section~\ref{sec: ANewMetric}} \label{appendix1}
Here, we provide the proofs of Theorem~\ref{thm:d1metric} and Proposition~\ref{topology1} stated in Section~\ref{sec: ANewMetric}.

\begin{proof}[Proof of Theorem~\ref{thm:d1metric}]


First, by definition, $D_1(\mu_1, \mu_1)=0$. Now, consider $\mu_1\neq \mu_2$. Note that, when $|\mu_1| \neq |\mu_2|$, we clearly have $D_{1}(\mu_1,\mu_2) > 0$. When $|\mu_1| = |\mu_2|$, the positive definiteness of $D_1$ is inherited from that of $D_{w}$. In addition, the symmetry follows directly from the definition. 

To establish the triangle inequality, it suffices to show that $D_{1}(\mu_1, \mu_2) \leq D_{1}(\mu_1, \mu_3) + D_{1}(\mu_2, \mu_3)$ for any three counting measures $\mu_1, \mu_2$ and $\mu_3$. The cardinalities of $\mu_1$, $\mu_2$, and $\mu_3$ are denoted by $n$, $m$, and $l$, respectively.
For ease of exposition, we assume without loss of generality that $n \leq m \leq l$; the remaining cases follow by the same argument.

By (\ref{eqn:equaltyD1}), we have $D_{1}(\mu_1,\mu_2) = D_{w}(\mu_1 + \sum_{i=n+1}^{m}\delta_{s_{\alpha}}, \mu_{2})$. We first claim that $D_{w}(\mu_1 + \sum_{i=n+1}^{m}\delta_{s_{\alpha}}, \mu_{2}) = D_{w}(\mu_1 + \sum_{i=n+1}^{l}\delta_{s_{\alpha}}, \mu_{2}+\sum_{i=m+1}^{l}\delta_{s_{\alpha}})$. We briefly outline the proof as follows. Let $\pi$ be the optimal permutation of the left-hand side and $\pi^{\prime}$ be the optimal permutation of the right-hand side. 

To show $D_{w}(\mu_1 + \sum_{i=n+1}^{m}\delta_{s_{\alpha}}, \mu_{2}) \ge D_{w}(\mu_1 + \sum_{i=n+1}^{l}\delta_{s_{\alpha}}, \mu_{2}+\sum_{i=m+1}^{l}\delta_{s_{\alpha}})$, one can construct a valid permutation for the right-hand side by extending $\pi$, pairing each additional $s_{\alpha}$ with another $s_{\alpha}$. Such pairing contributes zero value to $D_w$. Hence the value obtained from this extended permutation equals that of the left-hand side. Since the right-hand side is defined as the minimum over all valid permutations, it cannot be larger than the left-hand side. 

Next, we show $D_{w}(\mu_1 + \sum_{i=n+1}^{m}\delta_{s_{\alpha}}, \mu_{2}) \le D_{w}(\mu_1 + \sum_{i=n+1}^{l}\delta_{s_{\alpha}}, \mu_{2}+\sum_{i=m+1}^{l}\delta_{s_{\alpha}})$. Suppose $\pi^{\prime}$ pairs some $s_\alpha$ with an element $x$ in the support of $\mu_1$, then it must pair some $y$ in the support of $\mu_2$ with 
another $s_\alpha$. Therefore, replacing the original matching with one that pairs $s_{\alpha}$ with itself and $x$ with $y$ does not increase the total cost, since $\rho(s_{\alpha},s_{\alpha}) + \rho(x,y) \le \rho(s_{\alpha},x) + \rho(s_{\alpha},y)$. Thus, in optimal matching $\pi^{\prime}$, all $s_{\alpha}$ must be matched with another $s_{\alpha}$ whenever possible. The remaining terms then form a valid permutation for the left-hand side. Again, since pairing  $s_{\alpha}$ with each other does not increase the value and the left-hand side is defined as the minimum over all permutations, the result follows. Thus, $D_{1}(\mu_1,\mu_2) = D_{w}(\mu_1 + \sum_{i=n+1}^{l}\delta_{s_{\alpha}}, \mu_{2}+\sum_{i=m+1}^{l}\delta_{s_{\alpha}})$.

By the triangle inequality of $D_{w}$, we have $D_{w}(\mu_1 + \sum_{i=n+1}^{l}\delta_{s_{\alpha}}, \mu_{2}+\sum_{i=m+1}^{l}\delta_{s_{\alpha}}) \le D_{w}(\mu_1 + \sum_{i=n+1}^{l}\delta_{s_{\alpha}}, \mu_{3}) + D_{w}(\mu_{2}+\sum_{i=m+1}^{l}\delta_{s_{\alpha}},\mu_{3}) = D_{1}(\mu_{1},\mu_{3}) + D_{1}(\mu_{2},\mu_{3})$, where the last equality again uses the equality~\eqref{eqn:equaltyD1}.
 The triangle inequality for $D_{1}$ follows by combining the above arguments. Thus far, we have established all properties required for $D_1$ to be a valid metric.
 
Next, we prove the upper bound of the metric. By symmetry, we may assume that $ |\mu_1| = n \leq |\mu_2| = m$. It follows by the definition of $D_1$ and $s_{\alpha}$ that
\begin{align*}
D_{1}(\mu_{1},\mu_{2}) &= \min_{\pi \in \Pi_{m}}\left[\sum_{i=1}^{n}\rho(x_i,y_{\pi(i)}) + \sum_{i=n+1}^{m}\rho(s_{\alpha},y_{\pi(i)})  \right]
\leq \sum_{i=1}^{n}\rho(x_i,y_{i}) + \sum_{i=n+1}^{m}\rho(s_{\alpha},y_{i})  \\
&\leq \sum_{i=1}^{n }\mathrm{diam}(S) + \sum_{i=n+1}^{m} \left(\rho(s_{\alpha},s^{*}) + \rho(s^{*},y_{i}) \right) 
\leq \sum_{i=1}^{n }\mathrm{diam}(S) + \sum_{i=n+1}^{m} \left( \alpha + \mathrm{diam}(S) \right)\\
& \leq  m \left( \mathrm{diam}(S) + \alpha \right)
 = \left( |\mu_1| \vee |\mu_2|  \right) \left( \mathrm{diam}(S) + \alpha \right),
\end{align*}
where $s^{*} \in S$ satisfies $\rho(s^{*},s_{\alpha}) = \alpha$. Note that $s^{*}$ is attainable due to the fact that $S$ is compact.

Finally, the lower bound is trivial since $\sum_{i=n+1}^{m}\rho(s_{\alpha},y_{\pi(i)}) \ge (m-n)\alpha$.
\end{proof}

\begin{proof}[Proof of Proposition~\ref{topology1}]
(i) We begin by introducing the relative Wasserstein distance $\tilde{D}_{w}$ as defined in \cite{schuhmacher2005}. For two finite counting measures $\mu_1$ and $\mu_2$, if $|\mu_1| = |\mu_2| \geq 1$, then $\tilde{D}_{w}(\mu_1, \mu_2)$ is defined as $D_{w}(\mu_1, \mu_2)/( C_D|\mu_1|)$. By Condition~\ref{assume:metric}(i), this is equivalent to $D(\mu_1,\mu_2)/( C_D|\mu_1|)$. In addition, we set $\tilde{D}_{w}(\mu_1, \mu_2) = 1$ if $|\mu_1| \neq |\mu_2|$, and $\tilde{D}_{w}(\mu_1, \mu_2) = 0$ if $|\mu_1| = |\mu_2| = 0$.

We now proceed with the proof. To establish the first part, it suffices to show that $\mu_n \rightarrow \mu$ vaguely (namely, $\int f \, d\mu_n  \rightarrow \int f \, d\mu$ for all continuous functions $f$ with compact support) is equivalent to $D( \mu_n, \mu ) \rightarrow 0$. Suppose first that $D(\mu_n, \mu) \to 0$. Then there exists $N$ such that $D(\mu_{n},\mu) < c_D$ for all $n \geq N$. By Condition~\ref{assume:metric}(iii), this implies that $|\mu_{n}| = |\mu|$ for all $n \geq N$, and therefore, $
\tilde{D}_{w} ( \mu_n, \mu )= D( \mu_n, \mu )/( C_D|\mu|)$. Hence $\tilde{D}_{w}(\mu_n, \mu) \to 0$.

Conversely, if $\tilde{D}_{w}(\mu_n, \mu) \to 0$, then $|\mu_n|=|\mu|$ for all sufficiently large $n$, and by Condition~\ref{assume:metric}(i),
$D(\mu_n,\mu)=( C_D|\mu|)\,\tilde{D}_{w}(\mu_n,\mu)\to0.$
Thus, $\tilde{D}_{w}(\mu_n, \mu) \to 0$ if and only if $D(\mu_n, \mu) \to 0$. Finally, by Theorem 3.1.B in \cite{schuhmacher2005}, $\tilde{D}_{w}(\mu_n, \mu) \to 0$ is equivalent to $\mu_n \to \mu$ vaguely, which establishes the first part.

Since $D$ induces the vague topology $\mathcal{T}$, and the $\sigma$-algebra $\mathcal{N}_{\mathcal{T}}$ is generated by the vague topology, we conclude the second part of the statement.

(ii) Let $\{\mu_{n}\}_{n=1}^{\infty}$ be a sequence of counting measures on $S$ such that $|\mu_n| \leq m$. Then there exists an integer $N \leq m$ and a subsequence $\left\{\mu_{n_k}\right\}_{k=1}^{\infty}$ such that $|\mu_{n_k}| = N$ for all $k$. Accordingly, we may write $\mu_{n_{k}} = \sum_{i=1}^{N} \delta_{x_{n_{k},i}}$. Since $(S,\rho)$ is compact, the sequence $(x_{n_{k},1},...,x_{n_{k},N})$ admits a convergent subsequence, say converging to $(x_1,...,x_N)$. For notational convenience, we continue to denote the indices of this subsequence by $\{n_k\}_{k=1}^{\infty}$. Define $\mu = \sum_{i=1}^{N} \delta_{x_{i}}$. Since $\sum_{i=1}^{N} \rho(x_{n_{k},i}, x_i) \rightarrow 0$, by Condition~\ref{assume:metric}(i), we have
\[
 D(\mu_{n_k},\mu) = D_{w}(\mu_{n_k},\mu)
= \min_{\pi \in \Pi_{N}} \sum_{i=1}^{N}\rho(x_{n_{k},i},x_{\pi(i)}) \rightarrow 0.
\]

(iii) By Proposition 2.6 of \citep{Xia2005}, $\mathfrak{N}(S)$ is Polish in its vague topology, and therefore, it follows by the first statement of Proposition \ref{topology1} that $(\mathfrak{N}(S), D)$ is separable. Now consider a Cauchy sequence $\{\mu_n\}_{n=1}^{\infty}$. There exists $N$ such that $D(\mu_{n},\mu_{m}) < c_D$ for all $n,m \ge N$. By Condition~\ref{assume:metric}(iii), we must have $|\mu_{n}| = |\mu_{m}|$ for all $n,m \ge N$. Completeness can be established by a similar argument to that in the second statement.
\end{proof}

\subsection{Useful existing results for estimating Wasserstein distances} \label{appendix2}
We restate Proposition 1 and Proposition 3 in \cite{weed2019sharp} for the ease of readability and for the completeness of proofs.
\begin{proposition}  \label{Prop:firstinWeed} (Proposition 1 in \cite{weed2019sharp})
Let $(X, D)$ be a compact metric space with $\mathrm{diam}(X) \leq 1$, and let $W_p$ denote the Wasserstein distance of order $p$ between probability measures on $(X, D)$. Let $\mu$ and $\nu$ be two probability measures supported on $X$. Suppose $k^{*} \in \{0\} \cup \mathbb{N}^{+}$, and let $\mathcal{A}_1, \dots, \mathcal{A}_{k^{*}}$ be a sequence of nested partitions of $X$ into measurable subsets such that, for each $1 \leq k \leq k^{*}$, $\sup_{F \in \mathcal{A}_{k}} \mathrm{diam}(F) \leq 3^{-k}$. Then
\begin{equation*}
W_{p}^{p}(\mu, \nu) \leq 3^{-k^{*} p} + \sum_{k=1}^{k^{*}} 3^{-(k-1)p} \sum_{F \in \mathcal{A}_{k}} | \mu(F) - \nu(F) |.
\end{equation*}
\end{proposition}

\begin{proposition}  \label{prop:thirdinWeed} (Proposition 3 in \cite{weed2019sharp})
Let $(X, D)$ denote the same metric space as in Proposition~\ref{Prop:firstinWeed}. Then there exists a sequence of nested partitions $\{ \mathcal{A}_{k} \}_{1 \leq k \leq k^{*}}$ of $X$ into measurable subsets such that, for each $1 \leq k \leq k^{*}$, we have $\sup_{F \in \mathcal{A}_{k}} \mathrm{diam}(F) \leq 3^{-k}$ and $\mathrm{card}(\mathcal{A}_{k}) \leq \mathcal{M}_{3^{-(k+1)}}(X)$.
\end{proposition}

\subsection{Proofs of Lemmas~\ref{lm:resultaftertelescope} and \ref{lm:exactboundWassersteinxi}} \label{appendix3}

This subsection presents the proofs of Lemmas~\ref{lm:resultaftertelescope} and \ref{lm:exactboundWassersteinxi}. We begin with the proof of Lemma~\ref{lm:resultaftertelescope}.

\begin{proof}[Proof of Lemma~\ref{lm:resultaftertelescope}]
Let $\gamma_j$ denote the optimal coupling between $\widehat{\mathcal{L}}_{n}(\xi)^{\mathfrak{N}_j}$ and $\mathcal{L}(\xi)^{\mathfrak{N}_j}$ for each $j = 0, 1, 2, \dots, m$. The existence of such couplings is guaranteed by the completeness of $\mathfrak{N}$ with respect to the metric $D$ (see, for example, \cite[Theorem 11.8.2]{Dudley_2002}). Observe that both $\widehat{\mathcal{L}}_{n}(\xi)^{\mathfrak{N}_0}$ and $\mathcal{L}(\xi)^{\mathfrak{N}_0}$ are supported on $\mathfrak{N}_0$, which consists solely of the zero counting measure $\mathbf{0}$. Consequently, $\gamma_0$ is the Dirac measure concentrated at $\mathbf{0} \times \mathbf{0}$. Let $ K = \frac{1}{2} \sum_{j=0}^{m} | \widehat{\mathcal{L}}_{n}(\xi)(\mathfrak{N}_{j}) - \mathcal{L}(\xi)(\mathfrak{N}_{j}) | $. We then define
\begin{equation*}
\gamma = \sum_{j=0}^{m} \left( \widehat{\mathcal{L}}_{n}(\xi)(\mathfrak{N}_{j}) \wedge \mathcal{L}(\xi)(\mathfrak{N}_{j}) \right) \gamma_j + K^{-1} \mathds{1}_{\{ K \neq 0 \}} \,  H \times \Xi  ,
\end{equation*}
where
\begin{equation*}
H = \sum_{j=0}^{m} \left(  \widehat{\mathcal{L}}_{n}(\xi)(\mathfrak{N}_{j}) - \mathcal{L}(\xi)(\mathfrak{N}_{j}) \right)_{+} \widehat{\mathcal{L}}_{n}(\xi)^{ \mathfrak{N}_j }, \quad
\Xi = \sum_{j=0}^{m} \left(  \widehat{\mathcal{L}}_{n}(\xi)(\mathfrak{N}_{j}) - \mathcal{L}(\xi)(\mathfrak{N}_{j}) \right)_{-} \mathcal{L}(\xi)^{ \mathfrak{N}_j }.
\end{equation*}
It can be easily verified that $ H( \bigcup_{j=0}^{m} \mathfrak{N}_j ) = \Xi( \bigcup_{j=0}^{m} \mathfrak{N}_j ) = K $, and thus, $\gamma$ is a probability measure which is supported on $\left(\bigcup_{j=0}^{m} \mathfrak{N}_j\right) \times \left( \bigcup_{j=0}^{m} \mathfrak{N}_j \right) $ and whose marginal measures are $\widehat{\mathcal{L}}_{n}(\xi)  $
and $\mathcal{L}(\xi)$, respectively.

It follows by the definition of $W_{p}$ that
\begin{eqnarray*}
&&W_{p}^{p} ( \widehat{\mathcal{L}}_{n}(\xi), \mathcal{L}(\xi) )
\leq \int \int D^{p} ( \mu, \nu ) \, d\gamma(\mu, \nu) \\
&=& \sum_{j=0}^{m} \left( \widehat{\mathcal{L}}_{n}(\xi)(\mathfrak{N}_{j}) \wedge \mathcal{L}(\xi)(\mathfrak{N}_{j}) \right) \int \int D^{p} ( \mu, \nu ) \, d\gamma_{j}(\mu, \nu) 
 + K^{-1} \mathds{1}_{\{ K \neq 0 \}} \int \int D^{p} ( \mu, \nu ) \, dH(\mu) d \Xi(\nu).
\end{eqnarray*}
Since $ \gamma_0 = \delta_{ \mathbf{0} \times \mathbf{0} } $, we have $ \int \int D^{p} ( \mu, \nu ) \, d\gamma_{0}(\mu, \nu) =0 $. Therefore,
\begin{eqnarray*}
&& \sum_{j=0}^{m} \left( \widehat{\mathcal{L}}_{n}(\xi)(\mathfrak{N}_{j}) \wedge \mathcal{L}(\xi)(\mathfrak{N}_{j}) \right) \int \int D^{p} ( \mu, \nu ) \, d\gamma_{j}(\mu, \nu) \\
&=& \sum_{j=1}^{m} \left( \widehat{\mathcal{L}}_{n}(\xi)(\mathfrak{N}_{j}) \wedge \mathcal{L}(\xi)(\mathfrak{N}_{j}) \right) \int \int D^{p} ( \mu, \nu ) \, d\gamma_{j}(\mu, \nu) \\
&=& \sum_{j=1}^{m} \left( \widehat{\mathcal{L}}_{n}(\xi)(\mathfrak{N}_{j}) \wedge \mathcal{L}(\xi)(\mathfrak{N}_{j}) \right) W_{p}^{p}  \left( \widehat{\mathcal{L}}_{n}(\xi)^{\mathfrak{N}_j}, \mathcal{L}(\xi)^{\mathfrak{N}_j} \right) \\
&=& \sum_{j=1}^{m} \left( \widehat{\mathcal{L}}_{n}(\xi)(\mathfrak{N}_{j}) \wedge \mathcal{L}(\xi)(\mathfrak{N}_{j}) \right) W_{p}^{p} \left( \widehat{\mathcal{L}}_{n}(\xi)^{\mathfrak{N}_j}, \mathcal{L}(\xi)^{\mathfrak{N}_j} \right) \mathds{1}_{\{ \widehat{\mathcal{L}}_{n}(\xi)(\mathfrak{N}_j) >0 \}},
\end{eqnarray*}
where the second equality follows from the fact that $\gamma_j$ is the optimal coupling between $\widehat{\mathcal{L}}_{n}(\xi)^{\mathfrak{N}_j}$ and $\mathcal{L}(\xi)^{\mathfrak{N}_j}$, and the last equality holds because $\widehat{\mathcal{L}}_{n}(\xi)^{\mathfrak{N}_j} = \mathcal{L}(\xi)^{\mathfrak{N}_j}$ when $\widehat{\mathcal{L}}_{n}(\xi)(\mathfrak{N}_j) = 0$.

Furthermore, noting that $ D^{p} ( \mu, \nu ) \leq 2^{p} \left(D^{p} ( \mu, \mathbf{0} ) +  D^{p} ( \mathbf{0}, \nu )\right)$, it follows that
\begin{eqnarray*}
&& K^{-1} \mathds{1}_{\{ K \neq 0 \}} \int \int D^{p} ( \mu, \nu ) \, dH(\mu) d \Xi(\nu)  \\
&\leq& 2^{p}  K^{-1} \mathds{1}_{\{ K \neq 0 \}} \int \int \left(D^{p} ( \mu, \mathbf{0} ) +  D^{p} ( \mathbf{0}, \nu )\right) \, dH(\mu) d \Xi(\nu) \\
&=& 2^{p}  K^{-1} \mathds{1}_{\{ K \neq 0 \}} \left[  \int D^{p} ( \mu, \mathbf{0} ) \,  dH(\mu) \int  d \Xi(\nu) + \int dH(\mu) \int D^{p} ( \mathbf{0}, \nu ) \,  d \Xi(\nu) \right] \\
&\leq& 2^{p}  \left[ \int D^{p} ( \mu, \mathbf{0} ) \,  dH(\mu) + \int D^{p} ( \mathbf{0}, \nu )\ \,  d \Xi(\nu) \right] \\
&=& 2^{p} \sum_{j=0}^{m}  \left(  \widehat{\mathcal{L}}_{n}(\xi)(\mathfrak{N}_{j}) - \mathcal{L}(\xi)(\mathfrak{N}_{j}) \right)_{+} \int D^{p} ( \mu, \mathbf{0} ) \,  d\widehat{\mathcal{L}}_{n}(\xi)^{ \mathfrak{N}_j }(\mu) \\
&& + 2^{p} \sum_{j=0}^{m} \left(  \widehat{\mathcal{L}}_{n}(\xi)(\mathfrak{N}_{j}) - \mathcal{L}(\xi)(\mathfrak{N}_{j}) \right)_{-} \int  D^{p} ( \mathbf{0}, \nu )  \, d\mathcal{L}(\xi)^{ \mathfrak{N}_j }(\nu) \\
&\leq& 2^{p} C_D^{p} \sum_{j=1}^{m} j^{p} \left[ \left(  \widehat{\mathcal{L}}_{n}(\xi)(\mathfrak{N}_{j}) - \mathcal{L}(\xi)(\mathfrak{N}_{j}) \right)_{+} + \left(  \widehat{\mathcal{L}}_{n}(\xi)(\mathfrak{N}_{j}) - \mathcal{L}(\xi)(\mathfrak{N}_{j}) \right)_{-} \right] \\
&=& 2^{p} C_D^{p} \sum_{j=1}^{m} j^{p} \left|\widehat{\mathcal{L}}_{n}(\xi)(\mathfrak{N}_{j}) - \mathcal{L}(\xi)(\mathfrak{N}_{j}) \right|,
\end{eqnarray*}
where the last inequality follows from the fact that $\widehat{\mathcal{L}}_{n}(\xi)^{\mathfrak{N}_j}$ and $\mathcal{L}(\xi)^{\mathfrak{N}_j}$ are supported on $\mathfrak{N}_j$, together with the identity $\mathfrak{N}_0=\{\mathbf{0}\}$ and Condition~\ref{assume:metric} (ii).
Combining the last three displays, the desired result follows.
\end{proof}

We proceed to prove Lemma~\ref{lm:exactboundWassersteinxi}.

\begin{proof}[Proof of Lemma~\ref{lm:exactboundWassersteinxi}]
The proof is similar to that of \cite{LeiJing2020Caco}.

For $j =1, \dots, m$, consider $\mathfrak{N}_j$ with the metric replaced by $D/ \left( j (1 \vee  C_D ) \right)$. Then,\\
 $ \left(\mathfrak{N}_j,  D/ \left( j (1 \vee  C_D ) \right) \right)$ satisfies the conditions in Proposition \ref{Prop:firstinWeed}. Therefore, applying Proposition \ref{prop:thirdinWeed}, there exists a sequence of nested partitions $\{ \mathcal{A}_{k}^{j} \}_{ 1 \leq k \leq k^{*}_{j} }$ of $\mathfrak{N}_j$ such that $ \mathrm{card}( \mathcal{A}_{k}^{j} ) \leq \widetilde{\mathcal{M}}_{3^{-(k+1)}} ( \mathfrak{N}_{j} ) $ and
\begin{equation*}
\sup_{\mathfrak{B} \in \mathcal{A}_{k}^{j} } \sup_{\mu, \nu \in \mathfrak{B}} \frac{ D( \mu, \nu) }{ j(1 \vee C_D ) } \leq 3^{-k}.
\end{equation*}
Here, we use $\widetilde{\mathcal{M}}_{3^{-(k+1)}} ( \mathfrak{N}_{j} )$ to denote the covering number of $\mathfrak{N}_j$ with respect to the scaled metric $ D/ \left( j (1 \vee  C_D ) \right)$, in order to distinguish it from $\mathcal{M}_{3^{-(k+1)}}(\mathfrak{N}_j)$, which denotes the corresponding covering number with respect to the original metric $D$. It is straightforward that 
\begin{equation*}
\widetilde{\mathcal{M}}_{3^{-(k+1)}} ( \mathfrak{N}_{j} ) = \mathcal{M}_{3^{-(k+1)} \left( j (1 \vee  C_D ) \right)}( \mathfrak{N}_{j} ) \leq \mathcal{M}_{3^{-(k+1)}}( \mathfrak{N}_{j} ),
\end{equation*}
and hence, $ \mathrm{card}( \mathcal{A}_{k}^{j} ) \leq \mathcal{M}_{3^{-(k+1)}}( \mathfrak{N}_{j} ) $.

For convenience of notation, we also use $\mathcal{A}_{0}^{j}$ to denote $\{ \mathfrak{N}_{j} \}$. It is obvious that $ \mathrm{card}( \mathcal{A}_{0}^{j} ) =1 \leq \mathcal{M}_{3^{-1}}( \mathfrak{N}_{j} )  $. With the above sequence of nested partitions, it follows by Proposition \ref{Prop:firstinWeed} that
\begin{eqnarray*}
&& W_{p}^{p} \left( \widehat{\mathcal{L}}_{n}(\xi)^{\mathfrak{N}_j}, \mathcal{L}(\xi)^{\mathfrak{N}_j} \right)
= \inf_{\gamma \in \Upsilon( \widehat{\mathcal{L}}_{n}(\xi)^{\mathfrak{N}_j}, \mathcal{L}( \xi )^{\mathfrak{N}_j} )} \int \int D^{p}(\mu, \nu ) \, d \gamma(\mu, \nu) \\
&=& j^{p} (1 \vee C_D )^{p} \inf_{\gamma \in \Upsilon( \widehat{\mathcal{L}}_{n}(\xi)^{\mathfrak{N}_j}, \mathcal{L}( \xi ) ^{\mathfrak{N}_j})} \int \int  \left(\frac{D(\mu, \nu )}{j (1 \vee C_D ) }\right)^{p}  \, d \gamma(\mu, \nu) \\
&\leq& j^{p} ( 1 \vee C_D )^{p} \left( 3^{-k_{j}^{*} p} + \sum_{k=1}^{k^{*}_{j}} 3^{-(k-1)p} \sum_{ \mathfrak{B} \in \mathcal{A}_{k}^{j} } \left| \widehat{\mathcal{L}}_{n}(\xi)^{ \mathfrak{N}_j } ( \mathfrak{B} ) - \mathcal{L}( \xi )^{ \mathfrak{N}_j } (\mathfrak{B}) \right|  \right).
\end{eqnarray*}
If $ \mathcal{L}(\xi)(\mathfrak{N}_j)>0 $, it follows by the definitions of $ \mathcal{L}(\xi)^{ \mathfrak{N}_j } $ and $ \widehat{\mathcal{L}}_{n}(\xi)^{ \mathfrak{N}_j } $ that
\begin{eqnarray*}
&& \sum_{ \mathfrak{B} \in \mathcal{A}_{k}^{j} } \left( \widehat{\mathcal{L}}_{n}(\xi)( \mathfrak{N}_j ) \wedge \mathcal{L}( \xi )( \mathfrak{N}_j ) \right) \left| \widehat{\mathcal{L}}_{n}(\xi)^{ \mathfrak{N}_j } (\mathfrak{B}) - \mathcal{L}( \xi )^{ \mathfrak{N}_j } (\mathfrak{B}) \right| \mathds{1}_{\{ \widehat{\mathcal{L}}_{n}(\xi)(\mathfrak{N}_j) >0 \}} \\
&\leq& \sum_{ \mathfrak{B} \in \mathcal{A}_{k}^{j} } \mathcal{L}( \xi )( \mathfrak{N}_j ) \left|  \frac{\widehat{\mathcal{L}}_{n}(\xi) ( \mathfrak{B} \cap \mathfrak{N}_j ) }{\widehat{\mathcal{L}}_{n}(\xi) ( \mathfrak{N}_j )}  -  \frac{\mathcal{L}( \xi ) ( \mathfrak{B} \cap \mathfrak{N}_j ) }{\mathcal{L}( \xi ) ( \mathfrak{N}_j )} \right|  \mathds{1}_{\{ \widehat{\mathcal{L}}_{n}(\xi)(\mathfrak{N}_j) >0 \}} \\
&=& \sum_{ \mathfrak{B} \in \mathcal{A}_{k}^{j} } \left|  \mathcal{L}( \xi )( \mathfrak{N}_j ) \frac{ \widehat{\mathcal{L}}_{n}(\xi) ( \mathfrak{B} \cap \mathfrak{N}_j )}{\widehat{\mathcal{L}}_{n}(\xi) ( \mathfrak{N}_j )}    -  \mathcal{L}( \xi ) ( \mathfrak{B} \cap \mathfrak{N}_j ) \right|  \mathds{1}_{\{ \widehat{\mathcal{L}}_{n}(\xi)(\mathfrak{N}_j) >0 \}} \\
&\leq &   \sum_{ \mathfrak{B} \in \mathcal{A}_{k}^{j} } \left| \mathcal{L}( \xi )( \mathfrak{N}_j ) - \widehat{\mathcal{L}}_{n}(\xi)( \mathfrak{N}_j )   \right|  \frac{ \widehat{\mathcal{L}}_{n}(\xi)( \mathfrak{B} \cap \mathfrak{N}_j ) }{ \widehat{\mathcal{L}}_{n}(\xi)( \mathfrak{N}_j ) } \mathds{1}_{\{ \widehat{\mathcal{L}}_{n}(\xi)(\mathfrak{N}_j) >0 \}} \\
&& + \sum_{ \mathfrak{B} \in \mathcal{A}_{k}^{j} } \left|    \widehat{\mathcal{L}}_{n}(\xi) ( \mathfrak{B} \cap \mathfrak{N}_j )  -  \mathcal{L}( \xi ) ( \mathfrak{B} \cap \mathfrak{N}_j ) \right|  \mathds{1}_{\{ \widehat{\mathcal{L}}_{n}(\xi)(\mathfrak{N}_j) >0 \}}  \\
&=&  \left| \mathcal{L}( \xi )( \mathfrak{N}_j ) - \widehat{\mathcal{L}}_{n}(\xi)( \mathfrak{N}_j )   \right|   \mathds{1}_{\{ \widehat{\mathcal{L}}_{n}(\xi)(\mathfrak{N}_j) >0 \}} 
+ \sum_{ \mathfrak{B} \in \mathcal{A}_{k}^{j} } \left|    \widehat{\mathcal{L}}_{n}(\xi) ( \mathfrak{B} \cap \mathfrak{N}_j )  -  \mathcal{L}( \xi ) ( \mathfrak{B} \cap \mathfrak{N}_j ) \right|  \mathds{1}_{\{ \widehat{\mathcal{L}}_{n}(\xi)(\mathfrak{N}_j) >0 \}}  \\
&\leq&  \left| \mathcal{L}( \xi )( \mathfrak{N}_j ) - \widehat{\mathcal{L}}_{n}(\xi)( \mathfrak{N}_j )   \right| + \sum_{ \mathfrak{B} \in \mathcal{A}_{k}^{j} } \left|    \widehat{\mathcal{L}}_{n}(\xi) ( \mathfrak{B} \cap \mathfrak{N}_j )  -  \mathcal{L}( \xi ) ( \mathfrak{B} \cap \mathfrak{N}_j ) \right|,
\end{eqnarray*}
where the last equality follows by the fact that $ \mathcal{A}_{k}^{j} $ is a partition of $\mathfrak{N}_{j}$. If $\mathcal{L}(\xi)(\mathfrak{N}_j)=0$, the above inequality holds automatically, since the left-hand side is $0$. Therefore, Combining the last two displays, it follows after a rearrangement that
\begin{eqnarray*}
&& \left( \widehat{\mathcal{L}}_{n}(\xi)( \mathfrak{N}_j ) \wedge \mathcal{L}( \xi )( \mathfrak{N}_j ) \right) W_{p}^{p} \left( \widehat{\mathcal{L}}_{n}(\xi)^{\mathfrak{N}_j}, \mathcal{L}(\xi)^{\mathfrak{N}_j} \right)  \mathds{1}_{\{ \widehat{\mathcal{L}}_{n}(\xi)(\mathfrak{N}_j) >0 \}}  \\
&\leq& j^{p} (1 \vee C_D )^{p} \Bigg[  3^{-k_{j}^{*} p} \left( \widehat{\mathcal{L}}_{n}(\xi)( \mathfrak{N}_j ) \wedge \mathcal{L}( \xi )( \mathfrak{N}_j ) \right)  \\
&& \qquad +  \sum_{k=1}^{k^{*}_{j}} 3^{-(k-1)p} \bigg(\left| \mathcal{L}( \xi )( \mathfrak{N}_j ) - \widehat{\mathcal{L}}_{n}(\xi)( \mathfrak{N}_j )   \right| + \sum_{ \mathfrak{B} \in \mathcal{A}_{k}^{j} } \left|    \widehat{\mathcal{L}}_{n}(\xi) ( \mathfrak{B} \cap \mathfrak{N}_j )  -  \mathcal{L}( \xi ) ( \mathfrak{B} \cap \mathfrak{N}_j ) \right|\bigg) \Bigg] \\
&\leq& j^{p} (1 \vee C_D )^{p} \Bigg[  3^{-k_{j}^{*} p}  \mathcal{L}( \xi )( \mathfrak{N}_j ) + \frac{1}{1- 3^{-p}} \left| \mathcal{L}( \xi )( \mathfrak{N}_j ) - \widehat{\mathcal{L}}_{n}(\xi)( \mathfrak{N}_j )   \right| \\
&& \qquad \qquad \qquad \qquad \qquad + 3^{p} \sum_{k=1}^{k^{*}_{j}} 3^{-kp} \sum_{ \mathfrak{B} \in \mathcal{A}_{k}^{j} } \left|    \widehat{\mathcal{L}}_{n}(\xi) ( \mathfrak{B} \cap \mathfrak{N}_j )  -  \mathcal{L}( \xi ) ( \mathfrak{B} \cap \mathfrak{N}_j ) \right| \Bigg].
\end{eqnarray*}
Combining the last display with Lemma \ref{lm:resultaftertelescope}, we have
\begin{align*}
& W_{p}^{p} ( \widehat{\mathcal{L}}_{n}(\xi), \mathcal{L}(\xi) ) \\
\leq &  (1 \vee C_D )^{p}  \sum_{j=1}^{m}  j^{p}\Bigg[  3^{-k_{j}^{*} p}  \mathcal{L}( \xi )( \mathfrak{N}_j ) + \left(\frac{1}{1- 3^{-p}} + 2^{p}\right) \left| \mathcal{L}( \xi )( \mathfrak{N}_j ) - \widehat{\mathcal{L}}_{n}(\xi)( \mathfrak{N}_j )   \right| \\
&  \qquad \qquad \qquad \qquad \qquad + 3^{p} \sum_{k=1}^{k^{*}_{j}} 3^{-kp} \sum_{ \mathfrak{B} \in \mathcal{A}_{k}^{j} } \left|    \widehat{\mathcal{L}}_{n}(\xi) ( \mathfrak{B} \cap \mathfrak{N}_j )  -  \mathcal{L}( \xi ) ( \mathfrak{B} \cap \mathfrak{N}_j ) \right| \Bigg] \\
\leq& 2 \times 3^{p} (1 \vee C_D )^{p} \sum_{j=1}^{m}  j^{p}\Bigg[  3^{-k_{j}^{*} p}  \mathcal{L}( \xi )( \mathfrak{N}_j ) +  \left| \mathcal{L}( \xi )( \mathfrak{N}_j ) - \widehat{\mathcal{L}}_{n}(\xi)( \mathfrak{N}_j )   \right| \\
&  \qquad \qquad \qquad \qquad \qquad \qquad +  \sum_{k=1}^{k^{*}_{j}} 3^{-kp} \sum_{ \mathfrak{B} \in \mathcal{A}_{k}^{j} } \left|    \widehat{\mathcal{L}}_{n}(\xi) ( \mathfrak{B} \cap \mathfrak{N}_j )  -  \mathcal{L}( \xi ) ( \mathfrak{B} \cap \mathfrak{N}_j ) \right| \Bigg] \\
=& 2 \times 3^{p} (1 \vee C_D )^{p} \sum_{j=1}^{m}  j^{p}\Bigg[  3^{-k_{j}^{*} p}  \mathcal{L}( \xi )( \mathfrak{N}_j ) + \sum_{k=0}^{k^{*}_{j}} 3^{-kp} \sum_{ \mathfrak{B} \in \mathcal{A}_{k}^{j} } \left|    \widehat{\mathcal{L}}_{n}(\xi) ( \mathfrak{B} \cap \mathfrak{N}_j )  -  \mathcal{L}( \xi ) ( \mathfrak{B} \cap \mathfrak{N}_j ) \right| \Bigg],
\end{align*}
where in the last equality, we applied the identity that $\mathcal{A}_{0}^{j} = \{ \mathfrak{N}_j \} $.

Finally, since $n \widehat{\mathcal{L}}_{n}(\xi) ( \mathfrak{B} \cap \mathfrak{N}_j )$ is a binomial random variable with parameters $ \mathcal{L}( \xi ) ( \mathfrak{B} \cap \mathfrak{N}_j ) $ and $n$, we have 
\begin{equation*}
E\left[ \left|  \widehat{\mathcal{L}}_{n}(\xi) ( \mathfrak{B} ) - \mathcal{L}( \xi ) ( \mathfrak{B} ) \right| \right] \leq \left( 2 \mathcal{L}( \xi ) ( \mathfrak{B} ) \right) \wedge \sqrt{ \mathcal{L}( \xi ) ( \mathfrak{B} )/n },
\end{equation*}
for all $\mathfrak{B} \in \mathcal{N}_{\mathcal{T}}$. Therefore,
\begin{align*}
& \sum_{ \mathfrak{B} \in \mathcal{A}_{k}^{j} } E \left[\left|    \widehat{\mathcal{L}}_{n}(\xi) ( \mathfrak{B} \cap \mathfrak{N}_j )  -  \mathcal{L}( \xi ) ( \mathfrak{B} \cap \mathfrak{N}_j ) \right|\right]
\leq  \sum_{ \mathfrak{B} \in \mathcal{A}_{k}^{j} }  \left( 2 \mathcal{L}( \xi ) ( \mathfrak{B} \cap \mathfrak{N}_j ) \right) \wedge \sqrt{ \mathcal{L}( \xi ) ( \mathfrak{B} \cap \mathfrak{N}_j )/n } \\
 \leq & \left( 2 \mathcal{L}( \xi ) ( \mathfrak{N}_j ) \right) \wedge \left( \sqrt{ \frac{ \mathrm{card}( \mathcal{A}_{k}^{j} ) \mathcal{L}( \xi ) ( \mathfrak{N}_j ) }{n} } \right)
 \leq  \left( 2 \mathcal{L}( \xi ) ( \mathfrak{N}_j ) \right) \wedge \left( \sqrt{ \frac{ \mathcal{M}_{3^{-(k+1)}}( \mathfrak{N}_{j} ) \mathcal{L}( \xi ) ( \mathfrak{N}_j ) }{n} } \right),
\end{align*}
where the second inequality follows by Cauchy-Schwarz inequality and the last inequality is a consequence of $ \mathrm{card}( \mathcal{A}_{k}^{j} ) \leq \mathcal{M}_{3^{-(k+1)}}( \mathfrak{N}_{j} ) $.

Combining the preceding display with the result given three displays above, we obtain the desired conclusion and thus complete the proof.
\end{proof}

\bibliographystyle{refstyle}
\bibliography{ref}

\end{document}